\documentclass[11pt]{article}
\usepackage{amsmath,amsfonts,amsthm,amssymb, graphicx}
\usepackage{subfigure}

\theoremstyle{plain}
\newtheorem{thm}{Theorem}
\newtheorem{proposition}{Proposition}
\newtheorem{lemma}{Lemma}

\newtheorem{remark}{Remark}

\input{tcilatex}
\begin{document}

\title{Rare-Event Simulation for Many-Server Queues}
\author{Jose Blanchet and Henry Lam\\\emph{Columbia University and Boston University}}
\date{}
\maketitle

\begin{abstract}
We develop rare-event simulation methodology for the analysis of loss events
in a many-server loss system under quality-driven regime, focusing on the steady-state loss probability (i.e. fraction
of lost customers over arrivals) and the behavior of the whole system
leading to loss events. The analysis of these events requires working with
the full measure-valued process describing the system. This is the first
algorithm that is shown to be asymptotically optimal, in the rare-event
simulation context, under the setting of many-server queues involving a full
measure-valued descriptor.
\end{abstract}
\normalsize

While there is vast literature on rare-event simulation algorithms for
queues with fixed number of servers, few algorithms exist for queueing
systems with many servers. In systems with single or a fixed number of
servers, random walk representations are often used to analyze associated
rare events (see for example Siegmund (1976), Asmussen (1985), Anantharam\
(1988), Sadowsky (1991) and Heidelberger (1995)). The difficulty in these
types of systems arises from the boundary behavior induced by the positivity
constraints inherent to queueing systems. Many-server systems are, in some
sense, less sensitive to boundary behavior (as we shall demonstrate in the
basic development of our ideas) but instead the challenge in their
rare-event analysis lies on the fact that the system description is
typically infinite dimensional (measure-valued). One of the goals of this
paper, broadly speaking, is to propose methodology and techniques that we
believe are applicable to a wide range of rare-event problems involving
many-server systems. In particular, we will demonstrate how measure-valued description is both necessary and useful for efficient simulation. This arises primarily from the intimate relation between the steady-state large deviations behavior and the measure-valued diffusion approximation of many-server systems. As far as we know, the algorithm proposed in this paper is the first provably asymptotically optimal algorithm (in a sense that we will explain shortly) that involves such measure-valued descriptor in the rare-event simulation literature.

In order to illustrate our ideas we focus on the problem of estimating the
steady-state loss probability in many-server loss systems. We
consider a system with general i.i.d. interarrival times and service times
(both under suitable tail conditions). The system has $s$ servers and no
waiting room. If a customer arrives and finds a server empty, he immediately
starts service occupying a server. If the customer finds all the servers
busy, he leaves the system immediately and the system incurs a
\textquotedblleft loss\textquotedblright . The steady-state loss probability
(i.e. the long term proportion of customers that are lost) is rare if the
traffic intensity (arrival rate into the system / total service rate) is less
than one and the number of servers is large. This is precisely the
asymptotic environment that we consider.

Related large deviations and simulation results include the work of Glynn
(1995), who developed large deviations asymptotics for the number-in-system of an infinite-server queue with high arrival rates. Based on this
result, Szechtman and Glynn\ (2002) developed a corresponding rare-event
algorithm for the same quantity of an infinite-server queue, using a sequential tilting scheme that mimics the optimal exponential change of measure. Related
results for first passage time probabilities have also been obtained by
Ridder (2009) in the setting of Markovian queues. Blanchet, Glynn and Lam
(2009) constructed an algorithm for the steady-state loss probability of a
slotted-time $M/G/s$ system with bounded service time. The algorithm in
Blanchet, Glynn and Lam (2009) is the closest in spirit to our methodology
here, but the slotted-time nature, the Markovian structure and the fact that
the service times were bounded were used in a crucial way to avoid the main
technical complications involved in dealing with measure-valued descriptors.

In this paper we focus on the steady-state loss estimation of a fully
continuous $GI/G/s$ system with service times that accommodate most
distributions used in practice, including mixtures of exponentials, Weibull
and lognormal distributions. A key element of our algorithm, in addition to
the use of measure-valued process, is the application of weak convergence
limits by Krichagina and Puhalskii\ (1997) and Pang and Whitt (2009). As we
shall see, the weak convergence results are necessary because via a suitable
extension of regenerative-type simulation\ (see Section 2) the steady-state
loss probability of the system can be transformed to a first passage problem
of the measure-valued process starting from an appropriate set, suitably
chosen by means of such weak convergence analysis. However, unlike
infinite-server system, the capacity constraint ($s$ servers) introduces a
boundary that forces us to work with the sample path and to tract the whole
process history. We will also see that the properties (and especially
\textquotedblleft decay\textquotedblright\ behavior) of the steady-state
measure plays an important role in controlling the efficiency of the
algorithm in the case of unbounded service time. In fact, new logarithmic
asymptotic results of steady-state convergence (in the sense described in\
Section 4) are derived along our way to prove algorithmic efficiency.

Our main methodology to construct an efficient algorithm is based on
importance sampling, which is a variance reduction technique that biases the
probability measure of the system (via a so-called change of measure)\ to
enhance the occurrence of rare event. In order to correct for the bias, a
likelihood ratio is multiplied to the sample output to maintain
unbiasedness. The key to efficiency is then to control the likelihood ratio,
which is typically small, and hence favorable, when the change of measure resembles the
conditional distribution given the occurrence of rare event. Construction of
good changes of measure often draws on associated large deviations theory
(see Asmussen and Glynn (2007), Chapter 6). We will carry out this scheme of
ideas in subsequent sections.

The criterion of efficiency that we will be using is the so-called
asymptotic optimality (or logarithmic efficiency). More concretely, suppose
we want to estimate some probability $\alpha :=\alpha (s)$ that goes to 0 as
$s\nearrow\infty $. For any unbiased estimator $X$ of $\alpha $ (i.e. $%
\alpha =EX$) one must have $EX^{2}\geq (EX)^{2}=\alpha ^{2}$ by Jensen's
inequality. Asymptotic optimality requires that $\alpha ^{2}$ is also an
upper bound of the estimator's variance in terms of exponential decay rate.
In other words,%
\begin{equation*}
\liminf_{s\rightarrow \infty }\frac{\log EX^{2}}{\log \alpha ^{2}}=1.
\end{equation*}%
This implies that the estimator $X$ possesses the optimal exponential decay
rate any unbiased estimator can possibly achieve. See, for example, Bucklew
(2004), Asmussen and Glynn (2007) and Juneja and Shahabuddin (2006) for
further details on asymptotic optimality.

Finally, we emphasize the potential applications of loss estimation in
many-server systems. One prominent example is call center analysis. Customer
support centers, intra-company phone systems and emergency rooms, among
others, typically have fixed system capacity above which calls would be
lost. In many situations losses are rare, yet their implications can be
significant. The most extreme example is perhaps 911 center in which any
call loss can be life-threatening. In view of this, an accurate estimate (at
least to the order of magnitude) of loss probability is often an
indispensable indicator of system performance. While in this paper we focus
on i.i.d. interarrival and service times, under mild modifications, our
methodology can be adapted to different model assumptions such as
Markov-modulation and time inhomogeneity that arise naturally in certain
application environments. As a side tale, a rather surprising and novel
application of the present methodology is in the context of actuarial loss
in insurance and pension funds. In such systems the policyholders (insurance
contract or pension scheme buyers) are the \textquotedblleft customers", and
\textquotedblleft loss" is triggered not by an exceedence of the number of
customers but rather by a cash overflow of the insurer. Under suitable model
assumptions, the latter can be expressed as a functional of the past system
history whereby the measure-valued descriptor becomes valuable. The full
development of this application is presented in Blanchet and Lam (2011).

The organization of the paper is as follows.\ In Section 1 we will indicate our main results and lay out
our $GI/G/s$ model assumptions. In Section 2
we will explain and describe in detail our simulation methodology. Section 3
will focus on the proof of algorithmic efficiency and large deviations
asymptotics, while Section 4 will be devoted to the use of weak convergence
results mentioned earlier for the design of an appropriate recurrent set.
Finally, we will provide numerical results in Section 5, and technical details are left to the appendix.

\bigskip

\section{Main Results and Contributions}

\subsection{Problem Formulation and Main Results}

In this subsection we describe our problem formulation, and discuss our main
results. At a general level, our main contribution in this paper is the
development of methodology for efficient \textit{rare-event analysis of the
steady-state behavior} of many-server systems in a \textit{quality driven
regime}. Our methodology, however, is suitable for transient rare-event
analysis assuming the initial condition of the system is within the
diffusion scale from the fluid limit of the system.

The main idea of our methodology is to first introduce a coupling with the
infinite server queue. Second, take advantage of a suitable ratio
representation for the associated probability of interest for the system in
consideration (in our case a loss system). Third, identify a suitable
regenerative-like set based on available results in the literature on
diffusion approximations for the system in consideration. Finally, identify
a rare-event of interest inside a cycle that is common to both the system in
consideration and the infinite-server system, and that has the same
asymptotics as the probability of interest. It is crucial for the last step
to select the regenerative-like set carefully. We concentrate on loss
probabilities in this paper, but an almost identical (asymptotically
optimal) algorithm can be obtained for the steady-state probability of delay
in a many-server queue under the quality driven regime (when the traffic
intensity is bounded away from 1 as the number of servers and the arrival
rate grow to infinity at the same rate).

Throughout the rest of the paper we concentrate on loss systems and develop
the four elements outlined in the previous paragraph for the evaluation of
steady-state loss probabilities, which are defined as
\begin{equation}
P_{\pi }(\text{loss})=\lim_{T\rightarrow \infty }\frac{\text{number of
losses up to}\ T}{\text{number of arrivals up to}\ T}.
\label{loss probability definition}
\end{equation}%
Kac's formula (see Breiman (1968)) allows to express the loss probability as%
\begin{equation}
P_{\pi }(\text{loss})=\frac{E_{A}N_{A}}{\lambda sE_{A}\tau _{A}},
\label{Kac}
\end{equation}%
where $A$ is a set that is visited by the chain infinitely often. The
expectation $E_{A}[\cdot ]$ denotes the expectation with initial state
distributed according to the steady-state distribution conditioned on being
in $A$. The quantity $N_{A}$ is the number of loss before returning to set $A$, and $\tau
_{A}$ is the time back to $A$. Moreover, $\lambda s$ is the arrival rate (which is assumed to scale linearly with the number of servers $s$; the full discussion of our scaling assumptions will be laid out in the next subsection). For now, let us mention that both $E_AN_A$ and $E_A\tau_A$ are also dependent on the parameter $s$ because of the scaling.

%[[[[[Our first result is on the large deviations behavior of \eqref{Kac}:
%
%\begin{thm}
%The steady-state loss probability \eqref{Kac} is exponentially decaying in $s
%$ with decay rate $I^*$ defined in \eqref{I optimal}. In other words,
%\begin{equation}
%\lim_{s\to\infty}\frac{1}{s}\log P_\pi(\text{loss})=-I^*
%\label{large deviations loss probability}
%\end{equation}
%\label{main large deviations}
%\end{thm}
%
%\bigskip
%
%A key element of this result is the computability of $I^{\ast }$, as shown
%in \eqref{I optimal}. In fact, \eqref{I optimal} comes from the fact that $%
%I^{\ast }$ is the infimum of the rate functions for the probabilities that a
%coupled $GI/G/\infty $ system (defined in Section 1.4 below) hits level $s$
%at different fixed times, assuming the system starts at any point in a
%suitably chosen recurrent set $A$. These individual fixed-time rate
%functions are readily computable with formula given in \eqref{I_t}.
%Moreover, they elicit monotone properties (see Lemma \ref{rate}), which
%leads to the convenient formula for $I^{\ast }$ in \eqref{I optimal}.]]]]]

Note that
\eqref{loss
probability definition} cannot be directly simulated, but formula \eqref{Kac}
provides a basis for regenerative-type simulation (see Asmussen and Glynn
(2007), Chapter 4). After identifying a recurrent set $A$, a straightforward
crude Monte Carlo strategy would be to run the system for a long time from
some initial state, take a record of $N_{A}$ and $\tau _{A}$ every time it
hits $A$, and output the sample means of $N_{A}$ and $\tau _{A}$. This
strategy is valid as long as the running time is long enough to allow for
the system to be close to stationarity. Moreover, this strategy is basically
the same as merely outputting the number of loss events divided by the run
time times $\lambda s$ (excluding the uncompleted last $A$-cycle).

However, recognizing that loss is a rare event (with exponential decay rate in $s$
as we will show as a by-product of our analysis), this method will take an
exponential amount of time in $s$ to get a specified relative error. This is
regardless of the choice of $A$: if $A$ is large, it takes short time to
regenerate i.e. $\tau _{A}$ is small, and consequently the number of losses
reported as the numerator $E_{A}N_{A}$ of \eqref{Kac} is almost always zero;
whereas if $A$ is small, it takes a long time to regenerate. In order to
dramatically speed up the computation time, our strategy is the following.
We choose $A$ to be a \textquotedblleft central limit\textquotedblright\ set
so that $E_{A}\tau _{A}$ is not exponentially large in $s$ (and not
exponentially small either; see Section 2.1). This isolates the rarity of
loss to the numerator $E_{A}N_{A}$. In other words, it is very difficult for
the process to reach overflow in an $A$-cycle. \textit{The key, then, is to
construct an efficient importance sampling scheme to induce overflow and to
estimate the number of losses in each }$A$\textit{-cycle}.

We point out two practical observations using this approach: First, $\tau
_{A}$ and $N_{A}$ can be estimated separately i.e. one can \textquotedblleft
split\textquotedblright\ the process every time it hits $A$:\ one of which
we apply importance sampling to get one sample of $N_{A}$ and is then
discarded, to the other one we apply the original measure to get one sample
of $\tau _{A}$ and also set the initial position for the next $A$-cycle (see
Asmussen and Glynn (2007), Chapter 4). Secondly, to get an estimate of standard
deviation one has to use batch estimates since the samples obtained this way
possess serial correlations (Asmussen and Glynn (2007), Chapter 4). In other
words, one has to divide the simulated chain into several segments of equal
number of time units. Then an estimate of the steady-state loss probability
is computed from each chain segment. These estimates are regarded as
independent samples of loss probability. The details of batch sampling will
be provided in Section 5 when we discuss numerical results.

We summarize our approach as follows:
\bigskip

\noindent\textbf{Algorithm 1}

\begin{enumerate}
\item Choose a recurrent set $A$. Initialize the $GI/G/s$ queue's status as
any point in $A$.

\item Run the queue. Each time the queue hits a point in $A$, say $x$, do
the following: Starting from $x$,

\begin{enumerate}
\item Use importance sampling to sample one $N_A$, the number of loss in a
cycle.

\item Use crude Monte Carlo to sample one $\tau_A$, the return time. The
final position of this queue is taken as the new $x$.
\end{enumerate}

\item Divide the queue into several segments of equal time length. Compute
the estimate of steady-state loss probability using the batch samples.
\end{enumerate}

\bigskip

The main result of this paper is the construction and the asymptotic
optimality proof of an efficient importance sampling scheme together. In
order to show the optimality of the algorithm, on our way, we obtain large
deviations asymptotics for loss probabilities that might be of independent
interest.

\begin{thm}
The estimator using the recurrent set $A$ in \eqref{A} and the importance
sampler given by Algorithm 2 is asymptotically optimal. Moreover, the
steady-state loss probability \eqref{Kac} can be seen to be exponentially
decaying in $s $ with decay rate $I^*$ defined in \eqref{I optimal}. \label{main thm}
\end{thm}

%\bigskip

An important novel feature of the problem we consider (and our solution) is
that it requires a construction based on full measure-valued processes.
Intuitively, the steady-state loss probability of the $GI/G/s$ system
depends on its loss behavior starting from a \textquotedblleft
normal\textquotedblright\ or \textquotedblleft typical\textquotedblright\
state under stationarity (which comes from a diffusion limit). It turns out
that the loss behavior can vary substantially if one defines this initial
\textquotedblleft normal\textquotedblright\ state only through the system's
queue length (even though loss event is defined only through the queue
length). However, by defining the \textquotedblleft
normal\textquotedblright\ state through the whole description of the system
(which requires a measure) the loss behavior starting from this
measure-valued state is characterized by a natural optimal path in the large
deviations sense, and as a result we can identify the efficient importance
sampling scheme to induce such losses. These observations ultimately
translate to the need of a measure-valued recurrent set $A$ in the
simulation of $E_{A}N_{A}$ in \eqref{Kac}.
%This issue will be reported rigorously in
%Section 2 and 3 when we describe our importance sampling algorithm in
%detail. ]]]

We next point out two further methodological observations. First, our
importance sampling algorithm utilizes the representation of a (coupled) $%
GI/G/\infty $ as a point process.
%This is used for proving the \textquotedblleft
%continuity\textquotedblright\ of likelihood ratio that prevents a blow-up of
%\textquotedblleft overshoot\textquotedblright\ at the first passage time
%(see the proof of Theorem \ref{upper bound copy(1)} in Section 3)\TEXTsymbol{%
%<}---.
This point process representation, we believe, can also be used to prove
results on sample path large deviations for many-server systems; such
development will be reported in Blanchet, Chen and Lam (2012). Secondly, our
algorithm requires essentially the information of the whole sample path of
the system due to a randomization of time horizon, in contrast to the
algorithm proposed in Szechtman and Glynn (2002) for estimating fixed-time
probability.

Finally, the recurrent set $A$, given by \eqref{A}, can be seen to possess
the following properties:

\begin{proposition}
In the $GI/G/s$ system,
\begin{equation}
\lim_{s\rightarrow \infty }\frac{1}{s}\log E_{A}\tau _{A}^{p}=0
\label{tau limit}
\end{equation}
and
\begin{equation}
\limsup_{s\rightarrow \infty }\frac{1}{s}\log E_{A}N_{A}^{p}\leq0
\label{N limit}
\end{equation}
for any $p>0$. \label{asymptotic}
\end{proposition}

Briefly stated, Proposition \ref{asymptotic} stipulates that any moments of
the time length and number of losses of an $A$-cycle are subexponential in $s
$. When $p=1$, it in particular states that the expected time length of a
cycle is subexponential in $s$. As discussed above, this isolates the rarity
of loss to the numerator in \eqref{Kac} and ensures the validity of
Algorithm 1. The result on general $p$ in Proposition \ref{asymptotic} is
also used in the optimality proof of the importance sampling (as will be
seen in Section 3). Interestingly, the proof of Proposition \ref{asymptotic}
requires the use of the Borell-TIS inequality for Gaussian random fields.
The connection to Gaussian random fields arises in the diffusion limit of
the coupled $GI/G/\infty $ queue.

\subsection{Assumptions on Arrivals and Service Time Distribution}

We now state the assumptions of our model, namely a $GI/G/s$ loss system.
There are $s\geq 1$ servers in the system. We assume arrivals follow a
renewal process with rate $\lambda s$ i.e. the interarrival times are i.i.d.
with mean $1/(\lambda s)$. More precisely, we introduce a \textquotedblleft
base\textquotedblright\ arrival system, with $N^{0}(t),t\geq 0$ as its
counting process of the arrivals from time 0 to $t$, and $%
U_{k}^{0},k=0,1,2,\ldots $ as the i.i.d. interarrival times with $%
EU_{k}^{0}=1/\lambda $ (except the first arrival $U_{0}^{0}$, which can be
delayed). We then scale the system so that $N_{s}(t)$ $=N^{0}(st)$ is the
counting process of the $s$-th order system, and $U_{k}=U_{k}^{0}/s,k=0,1,2,%
\ldots $ are the interarrival times. Moreover, we let $A_{k},k=1,2,\ldots $
be the arrival times i.e. $A_{k}=\sum_{i=0}^{k-1}U_{i}$ (note the convention
$U_{k}=A_{k+1}-A_{k}$ and $A_{0}=0$). Note that for convenience we have
suppressed the dependence on $s$ in $U_{k}$ and $A_{k}$.

We assume that $U_{k}$ has exponential moments in a neighborhood of the
origin, and let $\kappa _{s}(\theta )=\log Ee^{\theta U_{k}}$ be the
logarithmic moment generating function of $U_{k}$. It is easy to see that $%
\kappa _{s}(\theta )=\kappa ^{0}(\theta /s)$ where $\kappa ^{0}(\theta
)=\log Ee^{\theta U_{k}^{0}}$ is the logarithmic moment generating function
of the interarrival time in the base system.

Since $\kappa ^{0}(\cdot )$ is increasing, we can let
\begin{equation}
\psi _{N}(\theta )=-\left( \kappa ^{0}\right) ^{-1}(-\theta )
\label{inverse}
\end{equation}%
where $\left( \kappa ^{0}\right) ^{-1}(\cdot )$ is the inverse of $\kappa
^{0}(\cdot )$. Note that $\kappa _{s}^{-1}(\theta )=s\left( \kappa
^{0}\right) ^{-1}(\theta )$. Also, $\psi _{N}(\cdot )$ is increasing and
convex; this is inherited from $\kappa ^{0}(\cdot )$.

Now we impose a few assumptions on $\psi_N(\cdot)$. First, we assume $\text{%
Dom}\ \psi_N\supset\mathbb{R}_+$ (that $\text{Dom}\ \psi_N\supset\mathbb{R}%
_- $ is obvious from the definition of $\psi_N(\cdot)$), and hence $\text{Dom%
}\ \psi_N=\mathbb{R}$. We also assume that $\psi _{N}(\cdot )$ is twice
continuously differentiable on $\mathbb{R}$, strictly convex and steep on
the positive side i.e. $\psi_N^{\prime }(\theta)\nearrow\infty$ as $%
\theta\nearrow\infty$. Thus $\psi _{N}^{\prime }(0)=\lambda $ and $\psi
_{N}^{\prime }\left( \mathbb{R}_+\right) =[\lambda ,\infty )$. Finally, we
insist the technical condition
\begin{equation}
\theta\frac{d}{d\theta}\log\psi_N(\theta)\to\infty  \label{technical}
\end{equation}
as $\theta\nearrow\infty$. This condition is satisfied by many common
interarrival distributions, such as exponential, Gamma, Erlang etc. (Its use
is in Lemma 4 as a regularity condition to prevent the blow-up of likelihood
ratio due to sample paths that hit overflow very early).

Under these assumptions we have for any $0=t_{0}<t_{1}<\cdots <t_{m}<\infty $
and $\theta _{1},\ldots ,\theta _{m}\in \text{Dom}\ \psi _{N}$,
\begin{equation}
\frac{1}{s}\log E\exp \left\{ \sum_{i=1}^{m}\theta
_{i}(N_{s}(t_{i})-N_{s}(t_{i-1}))\right\} \rightarrow \sum_{i=1}^{m}\psi
_{N}(\theta _{i})(t_{i}-t_{i-1})  \label{increment property}
\end{equation}%
as $s\nearrow \infty $. In particular, $\psi _{N}(\cdot )t$ is the so-called%
\textit{\ Gartner-Ellis limit} of\textit{\ }$N_{s}(t)$\textit{\ }for any%
\textit{\ }$t>0$ as $s\nearrow \infty $. See Glynn and Whitt (1991) and
Glynn (1995). In the case of Poisson arrival, for example, the interarrival
times are exponential and we have $\kappa (\theta )=\log (\lambda /(\lambda
-\theta ))$. This gives $\psi _{N}(\theta )=\lambda (e^{\theta }-1)$ and $%
\text{Dom}\ \psi _{N}=\mathbb{R}$.

We now state our assumptions on the service times. Denote $V_{k}$ as the
service time of the $k$-th arriving customer, and let $V_{k},k=1,2,\ldots $
be i.i.d. with distribution function $F(\cdot )$ and tail distribution
function $\bar{F}(\cdot )$. We assume that $F(\cdot )$ has a density $%
f(\cdot )$ that satisfies
\begin{equation}
\lim_{y\rightarrow \infty }yh(y)=\infty  \label{light-tail assumption}
\end{equation}%
where $h(y)=f(y)/\bar{F}(y)$ is the hazard rate function (with the
convention that $h\left( y\right) =\infty $ whenever $\bar{F}(y)=0$). In
particular, (\ref{light-tail assumption}) implies that for any $p>0$ we can
find $a>0$ such that $yh(y)>p$ as long as $y>a$. Hence,
\begin{equation}
\bar{F}(y)=e^{-\int_{0}^{y}h(u)du}\leq c_{1}e^{-\int_{a}^{y}\frac{p}{u}%
du}=\frac{c_{2}}{y^{p}}  \label{light tail}
\end{equation}%
for some $c_{1},c_{2}>0$. In other words, $\bar{F}(\cdot )$ decays faster
than any power law. It is worth pointing out that assumption \eqref{light-tail assumption} covers
Weibull and log-normal service times, which have been observed to be
important models in call center analysis (see e.g. Brown et al (2005)).

Note that service time distribution does not scale with $s$. Hence the
traffic intensity, defined by the ratio of arrival rate to service rate, is $%
\lambda EV$ (we sometimes drop the subscript $k$ of $V_{k}$ for
convenience). We assume that $\lambda EV<1$. This corresponds to a \textit{%
quality-driven regime }and implies that loss is rare. We will see the
importance of this assumption in our derivation of efficiency and large
deviations results in Section 3.

%\bigskip

\subsection{Representation of System Status}

Let $Q(t)$ be the number of customers in the $GI/G/s$ system at time $t$.
More generally, we let $Q(t,y)$ to be the number of customers at time $t$
who have residual service time larger than $y$, where residual service time
at time $t$ for the $k$-th customer is given by $\left( V_{k}+A_{k}-t\right)
^{+}$ (defined for customers that are not lost). We also keep track of the
age process $B(t)=\inf \{t-A_{k}:A_{k}\leq t\}$ i.e. the time elapsed since
the last arrival. We assume right-continuous sample path i.e. customers who
arrive at time $t$ and start service are considered to be in the system at
time $t$, while those who finish their service at time $t$ are outside the
system at time $t$. We also make the assumption that service time is
assigned and known upon arrival of each served customer. While not
necessarily true in practice, this assumption does not alter any output from
a simulation point of view as far as estimation of loss probabilities is
concerned. To insist on a Markov description of the process, we let $%
W_{t}=(Q(t,\cdot ),B(t))\in \mathcal{D[}0,\infty )\times \mathbb{R}_{+}$ as
the state of the process at time $t$. In the case of bounded service time
over $[0,M]$ the state-space is further restricted to $\mathcal{D}%
[0,M]\times \mathbb{R}_{+}$.

%\bigskip

\subsection{A Coupling $GI/G/\infty $ System}

As indicated briefly before, in multiple times in this paper we shall use a $GI/G/\infty $ system that is
naturally coupled with the $GI/G/s$ system under the above assumptions. This
$GI/G/\infty $ system has the same arrival process and service time
distribution as the $GI/G/s$ system but has infinite number of servers and
thus no loss can occur. Furthermore, it labels $s$ of its servers from the
beginning. When customer arrives, he would choose one of the idle labeled
servers in preference to the rest, and only choose unlabeled server if all
the $s$ labeled servers are busy. It is then easy to see that the evolution
of the $GI/G/\infty $ system restricted to the $s$ labeled servers follows
exactly the same dynamic of the $GI/G/s$ system that we are considering. The
purpose of introducing this system is to remove the nonlinear
\textquotedblleft boundary" condition on the queue, hence leading to
tractable analytical results that we can harness, while the coupling
provides a link from this system back to the original $GI/G/s$ system. In
this paper we shall use the superscript ``$\infty$" to denote quantities in
the $GI/G/\infty $ system, so for example $Q^{\infty }\left( t\right) $
denotes the number of customers at time $t$ for the $GI/G/\infty $ system,
and so on.

\bigskip

Throughout the paper we also use overline to denote quantities that exclude
the initial customers. So for example $\bar{Q}^\infty(t,y)$ denotes the
number of customers who arrive after time 0 in the $GI/G/\infty$ system and
are present at time $t$ having residual service time larger than $y$ i.e. $%
\bar{Q}^\infty(t,y)=Q^\infty(t,y)-Q^\infty(0,t+y)$.

\bigskip

\section{Simulation Methodology}

As we have discussed, two key issues in our algorithm are the choice of
recurrent set and the importance sampling algorithm. We will present them
in detail in Section 2.1 and Section 2.2 respectively.

\subsection{Recurrent Set}

First of all, note that one can pick $T=n\Delta $ for some $\Delta >0$ in
the definition of loss probability given by equation (\ref{loss probability
definition}) and send $n\rightarrow \infty $. The introduction of the
lattice of size $\Delta $ is useful to define return times to the set $A$
only at lattice points. So, let us pick a fixed small time interval $\Delta $
(one choice, for example, is say $1/5$ of the mean of service time). We
choose $A$ to be
\begin{equation}
A=\left\{ Q(t,y)\in J(y)\text{\ for all\ }y\in \lbrack 0,\infty ),\ t\in\{0,\Delta ,2\Delta ,\ldots\}\right\}.   \label{A}
\end{equation}%
Here $J(y)$ is the interval
\begin{equation}
J(y)=\left( \lambda s\int_{y}^{\infty }\bar{F}(u)du-\sqrt{s}C^{\ast }\xi
(y),~\lambda s\int_{y}^{\infty }\bar{F}(u)du+\sqrt{s}C^{\ast }\xi (y)\right)
\label{J}
\end{equation}%
for some well chosen constant $C^{\ast }>0$ (discussed in Remark 1 below and in\ Section
4) and
\begin{equation}
\xi (y)=\nu (y)+\gamma \int_{y}^{\infty }\nu (u)du  \label{xi}
\end{equation}%
where
\begin{equation}
\nu (y)=\left( \lambda \int_{y}^{\infty }\bar{F}(u)du\right) ^{1/(2+\eta )}
\label{nu}
\end{equation}%
with any constants $\eta ,\gamma >0$.

%A few comments are in place. First, we define $A$ in a generalized sense
%that it depends on both $Q(t,\cdot )$ and $t$; it is easy to check that
%Kac's formula \eqref{Kac} still holds. We introduce the lattice $\{\Delta
%,2\Delta ,\ldots \}$ to avoid exponentially small (and even zero) $\tau _{A}$%
%. In particular, $\Delta $ will set a lower bound for the regenerative time.
%Next,
The form of $J(y)$ comes from the heavy traffic limit of $GI/G/\infty $
queue. Pang and Whitt (2009) proved the fluid limit $Q^{\infty
}(t,y)/s\rightarrow \lambda \int_{y}^{t+y}\bar{F}(u)du$ a.s. and the
diffusion limit $(Q^{\infty }(t,y)-\lambda s\int_{y}^{t+y}\bar{F}(u)du)/%
\sqrt{s}\Rightarrow R(t,y)$ for some Gaussian process $R(t,y)$ on the state
space $\mathcal{D}[0,\infty )$ with $\text{var}(R(t,y))\rightarrow \lambda
c_{a}^{2}\int_{y}^{\infty }\bar{F}(u)^{2}du+\lambda \int_{y}^{\infty }F(u)%
\bar{F}(u)du$ as $t\rightarrow \infty $, where $c_{a}$ is the coefficient of
variation of the interarrival times. Our recurrent set $A$ is thus a
\textquotedblleft confidence band\textquotedblright\ of the steady state of $%
Q^{\infty }(t,y)$, with the width of the confidence band decaying slower
than the standard deviation of $Q^{\infty }(\infty ,\cdot )$. It can be
proved (see Proposition \ref{asymptotic}) that this choice of $A$ indeed
leads to a return time\ that is subexponential in $s$.  The slower decay
rate of the confidence band width is a technical adjustment to enlarge $A$
so that a subexponential (in $s$) return time for the $GI/G/\infty $ system
is guaranteed. In fact, for the case of bounded service time, it suffices to
set $\eta =0$.

%The first limit of Proposition \ref{asymptotic} (putting $p=1$) in
%particular concludes that it is sufficient to focus on $E_{A}N_{A}$ to
%construct an asymptotically optimal algorithm.

\begin{remark}
The interval $J(y)$ contains a non-negative integer for any value of $y$ if $%
C^*$ is chosen large enough. In fact, observe that the length of $J(y)$ is continuous
and decreasing in $y$, and let
\begin{equation}
l(s)=\sup \left\{ y>0:\sqrt{s}C^{\ast }\xi (y)\geq \frac{1}{2}\right\} .
\label{l}
\end{equation}%
If $y$ is such that the width of $J(y)$ is equal to $1$ (equivalently $y=l(s)
$) we have that the center of $J\left( y\right) $, namely $\lambda
s\int_{y}^{\infty }\bar{F}(u)du$ satisfies
\begin{equation*}
0\leq \lambda s\int_{y}^{\infty }\bar{F}(u)du\leq (\lambda /(C^{\ast
})^{2+\eta })(\sqrt{s}C^{\ast }\xi (y))^{2+\eta }/s^{\eta /2}=(\lambda
/(C^{\ast })^{2+\eta })(1/2)^{2+\eta }/s^{\eta /2}.
\end{equation*}%
The right hand side is less than 1/2 for $(C^{\ast })^{2+\eta}\geq \lambda $ and this
implies that $\{0\}\subset J\left( y\right) $ for $y=l\left( s\right) $.
Now, if $y>l\left( s\right) $, we can ensure that the half-width of $J(y)$,
namely $\sqrt{s}C^{\ast }\xi (y)$, is larger than the center, if $C^*$ is chosen sufficiently large. To see this, note that a sufficient condition is that
$$\lambda s\int_{y}^{\infty }\bar{F}(u)du \leq \sqrt{s}C^*\left(\lambda\int_y^\infty\bar{F}(u)du\right)^{1/(2+\eta)}$$
which is equivalent to
$$s^{1/2}\left( \int_{y}^{\infty }\bar{F}(u)du\right) ^{\left( 1+\eta \right)
/(2+\eta )} \leq C^{\ast }\lambda ^{-(1+\eta )/(2+\eta )}$$
or
$$s^{\left( 1+\eta /2\right) /(1+\eta )}\int_{y}^{\infty }\bar{F}%
(u)du  \leq\left( C^{\ast }\right) ^{\left( 2+\eta \right) /\left(
1+\eta \right) }\lambda ^{-1} $$
Now, choosing $C^{\ast }\geq \max \left( \lambda ,1\right) $, we have, for $y>l(s)$,
$$s^{\left( 1+\eta /2\right) /(1+\eta )} \int_{y}^{\infty }\bar{F}%
(u)du  \leq s^{1+\eta /2}\int_{y}^{\infty }\bar{F}(u)du\leq
1/(C^{\ast })^{2+\eta }(1/2)^{2+\eta }\leq \left( C^{\ast }\right) ^{\left(
2+\eta \right) /\left( 1+\eta \right) }\lambda ^{-1}$$
which gives the required implication. So $\{0\}\subset
J\left( y\right) $ for $y>l\left( s\right) $.
Obviously it includes at least one point when $y<l(s)$ (because the width of
$J\left( y\right) $ is larger than 1). Therefore $J(y)$ always contains a non-negative integer for
any $y\geq 0$, and the recurrent set $A$ is hence well-defined.
% for large enough $s$. When $s$ is small, we can modify any
%empty $J(y)$ to include the two closest integers to the two ends of the
%interval. This obviously does not affect the asymptotic behavior of $J(y)$
%as $s\rightarrow \infty $.
\label{validity}%
\end{remark}

%\bigskip

\begin{remark}
One may ask whether it is possible to define $A$ in a finite-dimensional
fashion, instead of introducing the functional \textquotedblleft confidence
band\textquotedblright\ in \eqref{A}. For example, one may divide the the
domain of $y$ into segments $[y_{i},y_{i+1}),i=0,1,2,\ldots ,r(s)-1$ for
some integer $r(s)$ with $y_{0}=0$ and $y_{r(s)}=\infty $, where the length
of each segment can be dependent on $s$ and non-identical. One then define
the recurrent set as $\{Q(t,\cdot ):Q(t,y_{i})-Q(t,y_{i+1})\in A_{i}\text{\
for\ }i=0,\ldots ,r(s)-1\}$ for some well-defined sets $A_{i}$'s. As we will
see in the arguments in the subsequent sections, the important criteria of a
good recurrent set is: 1) it consists of a significantly large region in the
central limit theorem, so that it is visited often enough, 2) its deviation
from the mean of $Q(t,y)$ is small, in the sense that the distance between
any element in this recurrent set and the mean of the steady-state of $Q(t,y)
$, at every $y\in \lbrack 0,\infty )$, has order $o(s)$. Criterion 2) is
important, otherwise the large deviations of loss starting from two
different elements in the recurrent set can be substantially different. We
want to avoid having to consider several substantially different paths that
can contribute to the loss event in a significant way as having such
variability would complicate the design of the importance sampling estimator.

Keeping criterion 2) in mind, we conclude that it is important to fine-tune
the scale of the segments $[y_{i},y_{i+1})$ to preserve the efficiency of
the algorithm. This suggests that a reasonable description of the recurrent
set would involve a dimension that grows at a suitable rate as $s\rightarrow
\infty $, thereby effectively obtaining a set of the form that we propose.
The functional definition of $A$ in \eqref{A} happens to balance both
criteria 1) and 2).
\end{remark}

\subsection{Simulation Algorithm}

First we shall explain some heuristic in constructing the algorithm. As we
discussed earlier, the choice of $A$ isolates the rarity of steady-state
loss probability to $E_{A}N_{A}$, which in turn is small because of the
difficulty in approaching overflow from $A$. So on an exponential scale, $%
E_{A}N_{A}\approx P_{A}(\tau _{s}<\tau _{A})$, where $P_{A}(\cdot )$ is the
probability measure with initial state distributed as the steady-state
distribution conditional on $A$, and $\tau _{s}=\inf \{t>0:Q(t)>s\}$ is the
first passage time to overflow. Observe that the probability $P_{A}(\tau
_{s}<\tau _{A})$ is identical for $GI/G/s$ and the coupled $GI/G/\infty $
system since the systems are identical before $\tau _{s}$. The key idea is
to leverage our knowledge of the structurally simpler $GI/G/\infty $ system.
In fact, one can show that the greatest contribution to $P_{A}(\tau
_{s}<\tau _{A})$ is the probability $P_{A}(Q^{\infty }(t^{\ast })>s)$ for
some optimal time $t^{\ast }$, whereas the contribution by other times is
exponentially smaller.

In view of this heuristic, one may think that the most efficient importance
sampling scheme is to exponentially tilt the process as if we are interested
in estimating the probability $P_{A}(Q^{\infty }(t^{\ast })>s)$. However,
doing so does not guarantee a small \textquotedblleft
overshoot\textquotedblright\ of the process at $\tau _{s}$. Instead, we
introduce a randomized time horizon following the idea of Blanchet, Glynn
and Lam (2009). The likelihood ratio will then comprise of a mixture of
individual likelihood ratios under different time horizons, and a bound on
the overshoot is attained by looking at the right horizon (namely $\lceil
\tau _{s}\rceil $ as explained in Section 3).

Hence our algorithm will take the following steps. Suppose we start from
some position in $A$. First we sample a randomized time horizon with some
well-chosen distribution. Then we tilt the coupled $GI/G/\infty $ process to
target overflow over this realized time horizon i.e. as if we are estimating
$P_{A}(Q^{\infty }(t)>s)$ for the realized time horizon $t$.\ This involves
sequential tilting of both the arrivals and service times. Once
overflow is hit, we switch back to the $GI/G/s$ system, drop the lost
customers, and change back to the arrival rate and service times under the
original measure to run the $GI/G/s$ system until $A$ is reached. At this
time one sample of $N_{A}$ is recorded together with the likelihood ratio.

The key questions now are:\ 1) the sequential tilting scheme of arrivals and
service times given a realized time horizon 2) the distribution of the
random time 3) likelihood ratio of this mixture scheme. In the following we
will explain these ingredients in detail and then lay out our algorithm. The
proof of efficiency will be deferred to Section 3.

%\bigskip

\subsubsection{Sequential Tilting Scheme}

Denote $P_r(\cdot)$ and $E_r[\cdot]$ as the probability measure and
expectation with initial system status $r$. Suppose we want to estimate $%
P_{r}(Q^{\infty }(t)>s)$ efficiently for a $GI/G/\infty $ system as $%
s\nearrow\infty $, where $r(\cdot )\in J(\cdot ) $ $\subset D[0,\infty )$
(so that $r(y)$ is the number of initial customers still in the system at
time $y$). An important clue is an invocation of Gartner-Ellis Theorem (see
Dembo and Zeitouni (1998)) to obtain large deviations result. Although this
may not give an immediate importance sampling scheme, it can suggest the
type of exponential tilting needed that can be verified to be efficient.
This is proposed by Glynn (1995) and Szechtman and Glynn (2002), which we
briefly recall here.

To be more specific, let us introduce more notations. Let, for any $t>0$,
\begin{equation}
\psi _{t}(\theta ):=\int_{0}^{t}\psi _{N}(\log (e^{\theta }\bar{F}%
(t-u)+F(t-u)))du
\end{equation}%
This is the Gartner-Ellis limit (see for example Dembo and Zeitouni (1998))\
of $\bar{Q}^{\infty }(t)$ since
\begin{equation*}
\frac{1}{s}\log Ee^{\theta \bar{Q}^{\infty }(t)}=\frac{1}{s}\log E\exp
\left\{ \theta \sum_{i=1}^{N_{s}(t)}I(V_{i}>t-A_{i})\right\} \rightarrow
\int_{0}^{t}\psi _{N}(\log (e^{\theta }\bar{F}(t-u)+F(t-u)))du
\end{equation*}%
where $I(\cdot )$ is the indicator function (see Glynn (1995) for a proof.
It uses \eqref{increment property} and the definition of Riemann sum;
alternatively, see Lemma \ref{multivariate psi} in Section 3 as a
generalization of this result). Let us state the following properties of $%
\psi_t(\cdot)$ for later convenience:

\begin{lemma}
$\psi_t(\cdot)$ is defined on $\mathbb{R}$, twice continuously
differentiable, strictly convex and steep. \label{psi properties}
\end{lemma}

%\bigskip

Next let $a_{t}=1-\lambda \int_{t}^{\infty }\bar{F}(u)du$. Note that $%
a_{t}s+o(s)$ is the number of customers needed excluding the initial ones to
reach overflow at time $t$. In other words,%
\begin{equation}
P_{r}(Q^{\infty }(t)>s)=P(\bar{Q}^{\infty }(t)>a_{t}s+o(s))  \label{initial}
\end{equation}%
Now denote $\theta _{t}$ as the unique positive solution of the equation $%
\psi _{t}^{\prime }(\theta )=a_{t}$. Such solution exists because $\psi
_{t}(\cdot )$ is steep and that $a_{t}=1-\lambda\int_t^\infty\bar{F}%
(u)du>\lambda \int_{0}^{t}\bar{F}(u)du=\psi _{t}^{\prime }(0)$. Then under
our current assumptions Gartner-Ellis Theorem concludes that $(1/s)\log
P_{r}(Q^{\infty }(t)>s)\rightarrow -I_{t}$ where
\begin{equation}
I_{t}=\sup_{\theta \in \mathbb{R}}\{\theta a_{t}-\lambda _{t}(\theta
)\}=\theta _{t}a_{t}-\psi _{t}\left( \theta _{t}\right)  \label{I_t}
\end{equation}%
$I_{t}$ is the so-called rate function of $\bar{Q}^{\infty }(t)$ evaluated
at $a_{t}$.

At this point let us note the following properties of $\theta_t$ and $I_t$
when regarded as functions of $t$:

\begin{lemma}
$\theta_t$ satisfies the following:

\begin{enumerate}
\item $\theta _{t}>0$ is non-increasing in $t$ for all $t>0$

\item $\lim_{t\rightarrow 0}\theta _{t}=\infty $

\item $\lim_{t\to\infty}\theta_t=\theta_\infty$ where $\theta_\infty$ is the
unique positive root of the equation $\psi_\infty^{\prime }(\theta)=1$, and
\begin{equation}
\psi_\infty(\theta)=\int_0^\infty\psi_N(\log(e^\theta\bar{F}(u)+F(u)))du
\label{psi infinity}
\end{equation}
\end{enumerate}

\label{theta}
\end{lemma}

%\bigskip

\begin{lemma}
$I_t$ satisfies the following:

\begin{enumerate}
\item $I_{t}$ is non-increasing in $t$ for $t>0$.

\item $\lim_{t\to\infty}I_t=\inf_{t>0}I_t=I^*$ where
\begin{equation}
I^*=\theta_\infty-\psi_\infty(\theta_\infty)  \label{I optimal}
\end{equation}

\item If $V$ has bounded support over $[0,M]$, then $I^*=I_t$ for any $t\geq
M$.
\end{enumerate}

\label{rate}
\end{lemma}

%\bigskip

To construct an implementable efficient importance sampling scheme, one can
look at the derivative of $\psi _{t}\left( \theta \right) $:
\begin{equation*}
\psi _{t}^{\prime }(\theta )=\int_{0}^{t}\psi _{N}^{\prime }(\log (e^{\theta
}\bar{F}(t-u)+F(t-u)))\frac{e^{\theta }\bar{F}(t-u)}{e^{\theta }\bar{F}%
(t-u)+F(t-u)}du
\end{equation*}%
which is the mean of $\bar{Q}^{\infty }(t)$ under the exponential change of
measure with parameter $\theta $. When $\theta =0$, $\psi _{t}^{\prime
}(0)=\int_{0}^{t}\psi _{N}^{\prime }(0)\bar{F}(t-u)du=\lambda \int_{0}^{t}%
\bar{F}(t-u)du$. Comparing with $\psi _{t}^{\prime }(\theta _{t})$ suggests
a build-up of the system by accelerating the arrival rate from $\lambda $ to
$\psi _{N}^{\prime }(\log (e^{\theta _{t}}\bar{F}(t-u)+F(t-u)))$ at time $u$
and changing the service time distributions such that the probability for an
arrival at time $u$ to stay in the system at time $t$ is given by $e^{\theta
_{t}}\bar{F}(t-u)/(e^{\theta _{t}}\bar{F}(t-u)+F(t-u))$. Denote $\tilde{P}%
^{t}(\cdot )$ and $\tilde{E}^{t}[\cdot ]$ as the probability measure and
expectation under importance sampling. The above changes can be achieved by
setting an exponential tilting of the $i$-th interarrival time $U_{i}$ by
\begin{eqnarray*}
&&\tilde{P}^{t}(U_{i}\in dy) \\
&=&\exp \{\kappa _{s}^{-1}(-\log (e^{\theta _{t}}\bar{F}%
(t-A_{i})+F(t-A_{i})))y-\kappa _{s}(\kappa _{s}^{-1}(-\log (e^{\theta _{t}}%
\bar{F}(t-A_{i})+F(t-A_{i}))))\}{} \\
&&{}P(U_{i}\in dy) \\
&=&e^{-s\psi _{N}(\log (e^{\theta _{t}}\bar{F}(t-A_{i})+F(t-A_{i})))y}(e^{%
\theta _{t}}\bar{F}(t-A_{i})+F(t-A_{i}))P(U_{i}\in dy)
\end{eqnarray*}%
given the $i$-th arrival time $A_{i}$ (recall the convention $%
U_{i}=A_{i+1}-A_{i}$), and for an arrival at $A_{i}$ its tilted service time
distribution follows
\begin{equation*}
\tilde{P}^{t}(V_{i}\in dy)=\left\{
\begin{array}{ll}
\frac{f(y)}{e^{\theta _{t}}\bar{F}(t-A_{i})+F(t-A_{i})} & \text{\ for\ }%
0\leq y\leq t-A_{i} \\
\frac{e^{\theta _{t}}f(y)}{e^{\theta _{t}}\bar{F}(t-A_{i})+F(t-A_{i})} &
\text{\ for\ }y>t-A_{i}%
\end{array}%
\right.
\end{equation*}%
The contribution to likelihood ratio $P(\cdot )/\tilde{P}^{t}(\cdot )$ by
each arrival and service time assignment is accordingly (using slight abuse
of notation)
\begin{equation}
\frac{P(U_{i})}{\tilde{P}^{t}(U_{i})}=\frac{e^{s\psi _{N}(\log (e^{\theta
_{t}}\bar{F}(t-A_{i})+F(t-A_{i})))U_{i}}}{e^{\theta _{t}}\bar{F}%
(t-A_{i})+F(t-A_{i})}  \label{arrival tilt}
\end{equation}%
and
\begin{equation}
\frac{P(V_{i})}{\tilde{P}^{t}(V_{i})}=\frac{e^{\theta _{t}}\bar{F}%
(t-A_{i})+F(t-A_{i})}{e^{\theta _{t}I(V_{i}>t-A_{i})}}  \label{service tilt}
\end{equation}%
We tilt the process using \eqref{arrival tilt} and \eqref{service tilt}
until the time that we know overflow will happen at time $t$ i.e. $t\wedge
\tau _{s}[t]$ where $\tau _{s}[t]=\inf
\{u>0:r(t)+\sum_{i=1}^{N_{s}(u)}I(V_{i}>t-A_{i})>s\}$. The overall
likelihood ratio on the set $Q^{\infty }(t)>s$ will be

\begin{eqnarray}
L &=&\prod_{i=1}^{N_{s}(\tau _{s}[t])-1}\frac{e^{s\psi _{N}(\log
(e^{\theta_t}\bar{F}(t-A_{i})+F(t-A_{i})))}}{e^{\theta_t}\bar{F}%
(t-A_{i})+F(t-A_{i})}\prod_{i=1}^{N_{s}(\tau _{s}[t])}\frac{e^{\theta_t}\bar{%
F}(t-A_{i})+F(t-A_{i})}{e^{\theta_tI(V_{i}>t-A_{i})}}  \notag \\
&=&\exp \left\{ s\sum_{i=1}^{N_{s}(\tau _{s}[t])-1}\psi _{N}(\log
(e^{\theta_t}\bar{F}(t-A_{i})+F(t-A_{i})))U_{i}-\theta_t\sum_{i=1}^{N_{s}(%
\tau _{s}[t])}I(V_{i}>t-A_{i})\right\} {}  \notag \\
&&{}(e^{\theta_t}\bar{F}(t-A_{\tau _{s}[t]})+F(t-A_{\tau _{s}[t]}))
\label{L interim}
\end{eqnarray}%
This estimator $LI(Q^{\infty }(t)>s)$ can be shown to be asymptotically
optimal in estimating $P_{r}(Q^{\infty }(t)>s)$:

\begin{proposition}
\begin{equation*}
\limsup_{s\to\infty}\frac{1}{s}\log\tilde{E}^t_r[L^2;Q^\infty(t)>s]\leq-2I_t
\end{equation*}
\end{proposition}

%\bigskip

\begin{proof}
The proof follows from Szechtman and Glynn (2002), but for completeness (and also due to our introduction of $\tau_s[t]$ that simplifies the argument in their paper slightly) we shall present it here.

Note that $\sum_{i=1}^{N_{s}(\tau
_{s}[t])}I(V_{i}>t-A_{i})=s+1-r(t)=a_ts+o(s)$ by the definition of $\tau _{s}[t]$ and $r(t)$. Also,  $e^{\theta_t}\bar{F}(t-A_{\tau _{s}[t]})+F(t-A_{\tau
_{s}[t]})\leq e^{\theta_t}$ since $\theta_t>0$.

Since $\psi_N$ is continuous, $\sum_{i=1}^{N_{s}(\tau _{s}[t])-1}\psi _{N}(\log (e^{\theta_t}%
\bar{F}(t-A_{i})+F(t-A_{i})))U_{i}$ is an approximation to the Riemann
integral $\int_{0}^{\tau _{s}[t]}\psi _{N}(\log (e^{\theta_t}\bar{F}%
(t-u)+F(t-u)))du$, with intervals defined by $0=A_0<A_{1}<A_{2}<\ldots
<A_{N_{s}(\tau _{s}[t])}$ and within each interval the leftmost function
value is used as approximation (with the last interval truncated). Since $\psi _{N}(\log (e^{\theta_t }\bar{F}%
(t-u)+F(t-u)))$ is non-decreasing in $u$ when $\theta_t >0$, and $\tau
_{s}[t]\leq t$ on $Q^\infty(t)>s$, we have%
\begin{eqnarray*}
&&\sum_{i=1}^{N_{s}(\tau _{s}[t])-1}\psi _{N}(\log (e^{\theta_t}\bar{F%
}(t-A_{i})+F(t-A_{i})))U_{i} \\
&\leq &\int_{0}^{\tau _{s}[t]}\psi _{N}(\log (e^{\theta_t}\bar{F}%
(t-u)+F(t-u)))du \\
&\leq &\int_{0}^{t}\psi _{N}(\log (e^{\theta_t}\bar{F}(t-u)+F(t-u)))du
\\
&=&\psi _{t}(\theta_t)
\end{eqnarray*}%
on $Q^{\infty }(t)>s$. Hence \eqref{L interim} gives
\begin{equation*}
L^2\leq e^{2s\psi _{t}(\theta_t)-2\theta_t(a_ts+o(s))}
\end{equation*}%
which yields the proposition.
\end{proof}

%\bigskip

\subsubsection{Distribution of Random Horizon}

Denote $\tau $ as our randomized time horizon. We propose a discrete
power-law distribution for $\tau $ independent of the process:%
\begin{equation}
P(\tau =T+k\delta )=\frac{1}{(k+1)^{2}}-\frac{1}{(k+2)^{2}}\text{\ \ for\ \ }%
k=0,1,2\ldots  \label{tau}
\end{equation}%
where $\delta =\delta (s)=$ $c/s$ for some constant $c>0$. The power-law
distribution of $\tau $ is to avoid exponential contribution from the
mixture probability to the likelihood ratio that may disturb algorithmic
efficiency. Notice that we use a power law of order 2, and in fact we can
choose any power law distribution (with finite mean so that it does not take
long time to generate the process up to $\tau $).

$T$ is a constant to avoid tilting the process on a time horizon too close
to 0, otherwise likelihood ratio would blow up for paths that hit overflow
very early (because of the fact that $\lim_{t\rightarrow 0}\theta
_{t}=\infty $ in Lemma \ref{rate} Part 1; see also Section 3). A good choice of $T$ is the following. Let $\tilde{I}%
_{t}=\sup_{\theta \in \mathbb{R}}\{\theta (1-\lambda EV)-\psi _{N}(\theta
)t\}=\tilde{\theta}_{t}(1-\lambda EV)-\psi _{N}(\tilde{\theta}_{t})t$ where $%
\tilde{\theta}_{t}$ is the solution to the equation $\psi _{N}^{\prime
}(\theta )t=1-\lambda EV$ (which exists by the steepness assumption for
small enough $t$). This is the rate function of $N_{s}(t)$ evaluated at $%
1-\lambda EV$.

We choose $0<T<\infty $ that satisfies%
\begin{equation}
~\tilde{I}_{T}>2I^{\ast }  \label{T}
\end{equation}%
which always exists by the following lemma:

\begin{lemma}
$\tilde{I}_t$ satisfies the following:

\begin{enumerate}
\item $\tilde{I}_t$ is non-increasing in $t$ for $t<\eta$ for some small $%
\eta>0$.

\item $\tilde{I}_t\to\infty$ as $t\searrow0$.
\end{enumerate}

\label{tilde I}
\end{lemma}

%\bigskip

\begin{remark}
In fact by looking at the arguments in the next section, one can see that $%
\delta$ being merely $o(1)$ leads to asymptotic optimality. However, the
coarser the $\delta$, the larger is the subexponential factor beside the
exponential decay component in the variance, with the
extreme that when $\delta$ is order 1, asymptotic optimality no longer
holds. The choice of $\delta=c/s$ is found to perform well empirically, as
illustrated in Section 6.
\end{remark}

%\bigskip

\subsubsection{Likelihood Ratio}

After sampling the randomized time horizon, we accelerate the process using
the sequential tilting scheme \eqref{arrival tilt} and \eqref{service
tilt} with a realized $\tau =t$. But since we are now interested in the
first passage probability, we tilt the process until $t\wedge \tau
_{s}\wedge \tau _{A}$ (rather than $\tau _{s}[t]$ defined above). If $%
t\wedge \tau _{s}<\tau _{A}$, we continue the $GI/G/s$ system under the
original measure. Also, to prevent a blow-up of likelihood ratio close to $%
t=0$, we use the original measure throughout the whole process whenever $%
\tau =T$ (see the proof of efficiency next section). Now denote $\tilde{E}%
[\cdot ]$ and $\tilde{P}\left( \cdot \right) $ as the importance sampling
measure.\ We have
\begin{equation*}
\tilde{P}(W_{u},0\leq u\leq \tau _{s}\wedge \tau _{A})=\sum_{k=0}^{\infty
}P(\tau =T+k\delta )\tilde{P}^{T+k\delta }(W_{u},0\leq u\leq \tau _{s}\wedge
\tau _{A})
\end{equation*}%
(with $\tilde{P}^{T}(\cdot )=P(\cdot )$). So the overall likelihood ratio $%
L=L(W_{\cdot })$ on the set $\tau _{s}<\tau _{A}$ is given by
\begin{align}
L& =\frac{dP}{d\tilde{P}}=\frac{P(W_{u},0\leq u\leq \tau _{s})}{%
\sum_{k=0}^{\infty }P(\tau =T+k\delta )\tilde{P}^{T+k\delta }(W_{u},0\leq
u\leq \tau _{s})}  \notag \\
& =\frac{1}{\sum_{k=0}^{\infty }P(\tau =T+k\delta )L_{T+k\delta }^{-1}}
\label{L}
\end{align}%
where $L_{t}=L_{t}(W_{\cdot })$ is the individual likelihood ratio as a
sequential product of \eqref{arrival tilt} and \eqref{service tilt} up to $%
t\wedge \tau _{s}$ i.e.%
\begin{equation}
L_t=\left\{
\begin{array}{l}
\exp \left\{ s\sum_{i=1}^{N_{s}(\tau _{s})-1}\psi _{N}(\log (e^{\theta _{t}}%
\bar{F}(t-A_{i})+F(t-A_{i})))U_{i}-\theta _{t}\sum_{i=1}^{N_{s}(\tau
_{s})-1}I(V_{i}>t-A_{i})\right\} \\
\text{\ \ \ \ \ \ \ \ \ \ \ \ \ \ \ \ \ \ \ \ \ \ \ \ \ \ \ \ \ \ \ \ \ \ \
\ \ \ \ \ \ \ \ \ \ \ \ \ \ \ \ \ \ \ \ \ \ \ \ \ \ \ \ \ \ \ \ \ \ \ \ \ \
\ \ \ \ \ \ \ \ \ \ \ for\ }t\geq \tau _{s} \\
\exp \left\{ s\sum_{i=1}^{N_{s}(t)-1}\psi _{N}(\log (e^{\theta _{t}}\bar{F}%
(t-A_{i})+F(t-A_{i})))U_{i}-\theta
_{t}\sum_{i=1}^{N_{s}(t)-1}I(V_{i}>t-A_{i})\right\} \\
\text{\ \ \ \ \ \ \ \ \ \ \ \ \ \ \ \ \ \ \ \ \ \ \ \ \ \ \ \ \ \ \ \ \ \ \
\ \ \ \ \ \ \ \ \ \ \ \ \ \ \ \ \ \ \ \ \ \ \ \ \ \ \ \ \ \ \ \ \ \ \ \ \ \
\ \ \ \ \ \ \ \ \ \ \ for\ }t<\tau _{s}%
\end{array}%
\right.  \label{L_t}
\end{equation}%
for $t>T$ and is 1 for $t=T$.

%\bigskip

\subsubsection{The Algorithm}

We now state our algorithm. Assuming we start from $r(\cdot )\in J(\cdot )$
with a given initial age $B(0)$, do the following: \bigskip

\noindent\textbf{Algorithm 2}

\begin{enumerate}
\item Set $A_{0}=0$. Also initialize $N_{A}\leftarrow 0$, $L\leftarrow 0$
and $\tau _{s}\leftarrow \infty $.

\item Sample $\tau$ according to \eqref{tau}. Say we get a realization $%
\tau=t$.

\item Simulate $U_{0}$ according to the initial age $B(0)$. Set $A_{1}=U_{0}$%
. Check if $\tau _{A}$ is reached, in which case go to Step 7.

\item Starting from $i=1$, repeat the following (setting $\theta _{t}$ as
the one in \eqref{I_t} for $t>T$ and 0 for $t=T$):

\begin{enumerate}
\item Generate $V_{i}$ according to
\begin{equation*}
\tilde{P}^{t}(V_{i}\in dy):=\left\{
\begin{array}{ll}
\frac{f(y)}{e^{\theta _{t}}\bar{F}(t-A_{i})+F(t-A_{i})} & \text{\ for\ }%
0\leq y\leq t-A_{i} \\
\frac{e^{\theta _{t}}f(y)}{e^{\theta _{t}}\bar{F}(t-A_{i})+F(t-A_{i})} &
\text{\ for\ }y>t-A_{i}%
\end{array}%
\right.
\end{equation*}

\item Generate $U_{i}$ according to
\begin{equation*}
\tilde{P}^{t}(U_{i}\in dy):=e^{-s\psi _{N}(\log (e^{\theta _{t}}\bar{F}%
(t-A_{i})+F(t-A_{i})))y}(e^{\theta _{t}}\bar{F}(t-A_{i})+F(t-A_{i}))P(U_{i}%
\in dy)
\end{equation*}

\item Set $A_{i+1}=U_{i}+A_{i}$.

\item If $\tau _{A}$ is reached in $[A_{i},A_{i+1})$, go to Step 7.

\item Compute $Q^{\infty }(A_{i+1})$. If $Q^{\infty }(A_{i+1})>s$ then set $%
\tau _{s}\leftarrow A_{i+1}$, remove the new arrival at $A_{i+1}$, update $%
N_{A}\leftarrow N_{A}+1$, and go to Step 5.

\item If $A_{i+1}\geq t$, go to Step 5.

\item Update $i\leftarrow i+1$.
\end{enumerate}

\item Repeat the following:

\begin{enumerate}
\item Generate $V_{i}$ and $U_{i}$ under the original measure. Set $%
A_{i+1}=U_{i}+A_{i}$.

\item If $\tau _{A}$ is reached in $[A_{i},A_{i+1})$, go to Step 6.

\item Compute $Q(A_{i+1})$. This includes the removal of new arrival $%
A_{i+1} $ from the system in case it is a loss; in such case update $%
N_{A}\leftarrow N_{A}+1$, and set $\tau _{s}\leftarrow A_{i+1}$ if in
addition that $\tau _{s}=\infty $.

\item Update $i\leftarrow i+1$.
\end{enumerate}

\item Compute $LI(\tau _{s}<\tau _{A})$ using \eqref{L} and \eqref{L_t}.

\item Output $N_{A}LI(\tau _{s}<\tau _{A})$.
\end{enumerate}

\bigskip

\section{Algorithmic Efficiency}

In this section we will prove asymptotic optimality of the estimator
outputted by Algorithm 2. To be more precise, we will identify $I^{\ast }$
defined in \eqref{I optimal} as the exponential decay rate of $E_{A}N_{A}$.
The key result is the following:

\begin{thm}
The second moment of the estimator in Algorithm 2 satisfies
\begin{equation*}
\limsup_{s\rightarrow \infty }\frac{1}{s}\log \tilde{E}_{r}[N_{A}^{2}L^{2};%
\tau _{s}<\tau _{A}]\leq -2I^{\ast }
\end{equation*}%
for any $r(\cdot )\in J(\cdot )$. \label{upper bound copy(1)}
\end{thm}

%\bigskip

This result, together with Theorem \ref{lower bound copy(1)} in the sequel,
will expose a loop of inequality that leads to asymptotic optimality and
large deviations asymptotic simultaneously. The main technicality of this
result is an estimate of the continuity of the likelihood ratio, or
intuitively the \textquotedblleft overshoot" at the time of loss. It draws
upon a two-dimensional point process description of the system, in which the
geometry of the process plays an important role in estimating this
\textquotedblleft overshoot".

\begin{proof} Denote $\lceil x\rceil =\min \{T+k\delta ,\ k=0,1,\ldots
:x\leq T+k\delta \}$. Also recall the definition $a_{t}=1-\lambda
\int_{t}^{\infty }\bar{F}(u)du$.

Consider the likelihood ratio in \eqref{L}:
\begin{eqnarray*}
&&LI(\tau _{s}<\tau _{A})=\frac{1}{\sum_{k=0}^{\infty }P(\tau =T+k\delta
)L_{T+k\delta }^{-1}}I(\tau _{s}<\tau _{A})\leq \frac{L_{\lceil \tau _{s}\rceil }}{P(\tau =\lceil \tau _{s}\rceil )}%
I(\tau _{s}<\tau _{A}) \\
&=&P(\tau =T)^{-1}I(\tau _{s}\leq T;\tau _{s}<\tau _{A})+P(\tau =\lceil \tau
_{s}\rceil )^{-1}\exp \Bigg\{s\sum_{i=1}^{N_{s}(\tau _{s})-1}\psi _{N}(\log
(e^{\theta _{\lceil \tau _{s}\rceil }}\bar{F}(\lceil \tau _{s}\rceil
-A_{i}){} \\
&&{}+F(\lceil \tau _{s}\rceil -A_{i})))U_{i}-\theta _{\lceil \tau _{s}\rceil
}\sum_{i=1}^{N_{s}(\tau _{s})-1}I(V_{i}>\lceil \tau _{s}\rceil -A_{i})\Bigg\}%
I(\tau _{s}>T;\tau _{s}<\tau _{A}) \\
&\leq &C_{1}I(\tau _{s}\leq T;\tau _{s}<\tau _{A})+\frac{C_{2}\tau _{s}^{3}}{%
\delta ^{3}}\exp \left\{ s\psi _{\lceil \tau _{s}\rceil }(\theta _{\lceil
\tau _{s}\rceil })-\theta _{\lceil \tau _{s}\rceil }(\bar{Q}^{\infty }(\tau
_{s},\lceil \tau _{s}\rceil -\tau _{s})-1)\right\}{}\\
&&{} I(\tau _{s}>T;\tau
_{s}<\tau _{A}) \\
&\leq &C_{1}I(\tau _{s}\leq T;\tau _{s}<\tau _{A})+\frac{C_{2}\tau _{s}^{3}}{%
\delta ^{3}}\exp \Bigg\{-sI^{\ast }+\theta _{\lceil \tau _{s}\rceil }\Bigg(%
sa_{\lceil \tau _{s}\rceil }+1-\bar{Q}^{\infty }(\tau _{s},\lceil \tau
_{s}\rceil -\tau _{s})\Bigg)\Bigg\}{}\\
&&{}I(\tau _{s}>T;\tau _{s}<\tau _{A})
\end{eqnarray*}%
where $C_{1}$ and $C_{2}$ are positive constants. Note that the second inequality
comes from the fact that $\sum_{i=1}^{N_{s}(\tau _{s})-1}\psi _{N}(\log
(e^{\theta _{\lceil \tau _{s}\rceil }}\bar{F}(\lceil \tau _{s}\rceil
-A_{i})+F(\lceil \tau _{s}\rceil -A_{i})))U_{i}$
is a Riemann sum of the
integral $\psi _{\lceil \tau _{s}\rceil }(\theta _{\lceil \tau _{s}\rceil
})=\int_{0}^{\lceil \tau _{s}\rceil }\psi _{N}(\log (e^{\theta _{\lceil \tau
_{s}\rceil }}\bar{F}(\lceil \tau _{s}\rceil -u)+F(\lceil \tau _{s}\rceil
-u)))du$ (excluding the intervals at the two ends) and that $\psi _{N}(\log (e_{\lceil \tau _{s}\rceil }^{\theta }\bar{%
F}(\lceil \tau _{s}\rceil -u)+F(\lceil \tau _{s}\rceil -u)))$ is a
non-decreasing function in $u$. Also note that $\sum_{i=1}^{N_{s}(\tau
_{s})}I(V_{i}>\lceil \tau _{s}\rceil -A_{i}){}=\bar{Q}^{\infty }(\tau
_{s},\lceil \tau _{s}\rceil -\tau _{s})$ is the number of customers who
arrive before $\tau _{s}$ and leave after $\lceil \tau _{s}\rceil $. The
last inequality follows from the definition of $I_{\lceil \tau _{s}\rceil }$
and Lemma \ref{rate} Part 2. Now we have
\begin{eqnarray}
&&\tilde{E}_{r}[N_{A}^{2}L^{2};\tau _{s}<\tau _{A}]=E_{r}[N_{A}^{2}L;\tau
_{s}<\tau _{A}]  \notag \\
&\leq &C_{1}E_{r}[N_{A}^{2};\tau _{s}\leq T;\tau _{s}<\tau _{A}]+\frac{C_{2}%
}{\delta ^{3}}e^{-sI^{\ast }}E_{r}\Bigg[N_{A}^{2}\tau _{s}^{3}\exp \left\{
\theta _{\lceil \tau _{s}\rceil }\left( sa_{\lceil \tau _{s}\rceil }+1-\bar{Q%
}^{\infty }(\tau _{s},\lceil \tau _{s}\rceil -\tau _{s})\right) \right\};{} \notag\\
&&{}\tau _{s}>T;\tau _{s}<\tau _{A}\Bigg]  \label{eq4}
\end{eqnarray}%
Consider the first summand. By Holder's inequality $
E_{r}[N_{A}^{2};\tau _{s}\leq T;\tau _{s}<\tau _{A}]\leq
(E_{r}[N_{A}^{2p}])^{1/p}(P_{r}(\tau _{s}\leq T))^{1/q}$
for $1/p+1/q=1$. Also, $P_{r}(\tau _{s}\leq T)\leq P(N_{s}(T)>s-r(T))\leq P(N_{s}(T)>s(1-\lambda
EV)+o(s))$ and a straightforward invocation of Gartner-Ellis Theorem yields
$\lim_{s\rightarrow \infty }\frac{1}{s}\log P(N_{s}(T)>s(1-\lambda EV)+o(s))=-%
\tilde{I}_{T}<-2I^{\ast }$ by our choice of $T$ in \eqref{T}. Combining these observations, and using Lemma \ref{asymptotic},
we get
\begin{equation*}
\limsup_{s\rightarrow \infty }\frac{1}{s}\log E_{r}[N_{A}^{2};\tau _{s}\leq
T;\tau _{s}<\tau _{A}]\leq \limsup_{s\rightarrow \infty }\frac{1}{sp}\log
E_{r}[N_{A}^{2p}]+\limsup_{s\rightarrow \infty }\frac{1}{sq}\log P_{r}(\tau
_{s}\leq T)\leq -2I^{\ast }
\end{equation*}%
for $q$ close enough to 1.

In view of \eqref{eq4} and Dembo and Zeitouni (1998) Lemma 1.2.15, the proof will be complete once we can prove that
\begin{equation}
\limsup_{s\rightarrow \infty }\frac{1}{s}\log E_{r}\Bigg[N_{A}^{2}\tau
_{s}{}^{3}\exp \left\{ \theta _{\lceil \tau _{s}\rceil }\left( sa_{\lceil
\tau _{s}\rceil }+1-\bar{Q}^{\infty }(\tau _{s},\lceil \tau _{s}\rceil -\tau
_{s})\right) \right\} ;\tau _{s}>T;\tau _{s}<\tau _{A}\Bigg]\leq -I^{\ast }
\label{eq5}
\end{equation}%
To this end, we write
\begin{eqnarray}
&&E_{r}\Bigg[N_{A}^{2}\tau _{s}{}^{3}\exp \left\{ \theta _{\lceil \tau
_{s}\rceil }\left( sa_{\lceil \tau _{s}\rceil }+1-\bar{Q}^{\infty }(\tau
_{s},\lceil \tau _{s}\rceil -\tau _{s})\right) \right\} ;\tau _{s}>T;\tau
_{s}<\tau _{A}\Bigg]  \notag \\
&=&E_{r}\Bigg[N_{A}^{2}\tau _{s}{}^{3}\exp \left\{ \theta _{\lceil \tau
_{s}\rceil }\left( s+1-\lambda s\int_{\lceil \tau _{s}\rceil }^{\infty }\bar{%
F}(u)du-\bar{Q}^{\infty }(\tau _{s},\lceil \tau _{s}\rceil -\tau
_{s})\right) \right\} ;\tau _{s}>T;\tau _{s}<\tau _{A}\Bigg]  \notag \\
&\leq &e^{C\theta _{T}\sqrt{s}}E_{r}\Bigg[N_{A}^{2}\tau _{s}^{3}\exp \left\{
\theta _{\lceil \tau _{s}\rceil }\left( s+1-r(\lceil \tau _{s}\rceil )-\bar{Q%
}^{\infty }(\tau _{s},\lceil \tau _{s}\rceil -\tau _{s})\right) \right\}
;\tau _{s}>T;\tau _{s}<\tau _{A}\Bigg]  \notag \\
&=&e^{C\theta _{T}\sqrt{s}}\sum_{k=1}^{\infty }E_{r}\Bigg[N_{A}^{2}\tau
_{s}^{3}\exp \left\{ \theta _{\lceil \tau _{s}\rceil }\left( s+1-r(\lceil
\tau _{s}\rceil )-\bar{Q}^{\infty }(\tau _{s},\lceil \tau _{s}\rceil -\tau
_{s})\right) \right\} ;\lceil \tau _{s}\rceil =T+k\delta ;{} \notag\\
&&{}\tau
_{A}>T+(k-1)\delta \Bigg]  \notag \\
&\leq &e^{C\theta _{T}\sqrt{s}}\sum_{k=1}^{\infty
}(E_{r}N_{A}^{2p})^{1/p}(E_{r}\tau _{A}{}^{3q})^{1/q}(P_{r}(\tau
_{A}>T+(k-1)\delta ))^{1/h}{}  \notag \\
&&{}\left( E_{r}\left[ \exp \left\{ l\theta _{T+k\delta }\left(
s+1-r(T+k\delta )-\bar{Q}^{\infty }(\tau _{s},T+k\delta -\tau _{s})\right)
\right\} ;T+(k-1)\delta <\tau _{s}\leq T+k\delta \right] \right) ^{1/l}
\notag  \label{main} \\
&=&e^{O(\sqrt{s})}\sum_{k=1}^{\infty }(E_{r}N_{A}^{2p})^{1/p}(E_{r}\tau
_{A}{}^{3q})^{1/q}(P_{r}(\tau _{A}>T+(k-1)\delta ))^{1/h}{}  \notag \\
&&\left( E_{r}\left[ \exp \left\{ l\theta _{T+k\delta }\left( s+1-r(\tau
_{s})-\bar{Q}^{\infty }(\tau _{s},T+k\delta -\tau _{s})\right) \right\}
;T+(k-1)\delta <\tau _{s}\leq T+k\delta \right] \right) ^{1/l}  \label{eq9}
\end{eqnarray}%
where $C$ is a positive constant and $1/p+1/q+1/h+1/l=1$. The first
inequality follows from the fact that $r(\cdot )\in J(\cdot )$ and Lemma \ref%
{rate} Part 1 while the second inequality follows from generalized Holder's
inequality. The last equality holds because $r(\tau _{s})-r(T+k\delta )=o(s)$, again since $r(\cdot)\in J(\cdot)$, for $T+(k-1)\delta <\tau _{s}\leq T+k\delta $.

We now analyze
\begin{equation}
E_{r}\left[ \exp \left\{ l\theta _{T+k\delta }\left( s+1-r(\tau _{s})-\bar{Q}%
^{\infty }(\tau _{s},T+k\delta -\tau _{s})\right) \right\} ;T+(k-1)\delta
<\tau _{s}\leq T+k\delta \right]  \label{eq10}
\end{equation}%
We plot the arrivals on a two-dimensional plane, with $x$-axis indicating
the time of arrival and $y$-axis indicating the assigned service time at the
time of arrival. Such plot has been used in the study of $M/G/\infty$ system (see for example Foley (1982)). In this representation it is easy to see that the departure
time of an arriving customer is the $45^{\circ }$ projection of the point
onto the $x$-axis. As a result, $\bar{Q}^{\infty }(t)$ for example, will be
the number of all the points inside the triangular simplex created by a
vertical line and a downward $45^{\circ }$ line joining at the point $(t,0)$%
. See Figure 1.

\begin{figure}[ht]
\centering
\subfigure{
\includegraphics[scale=.3]{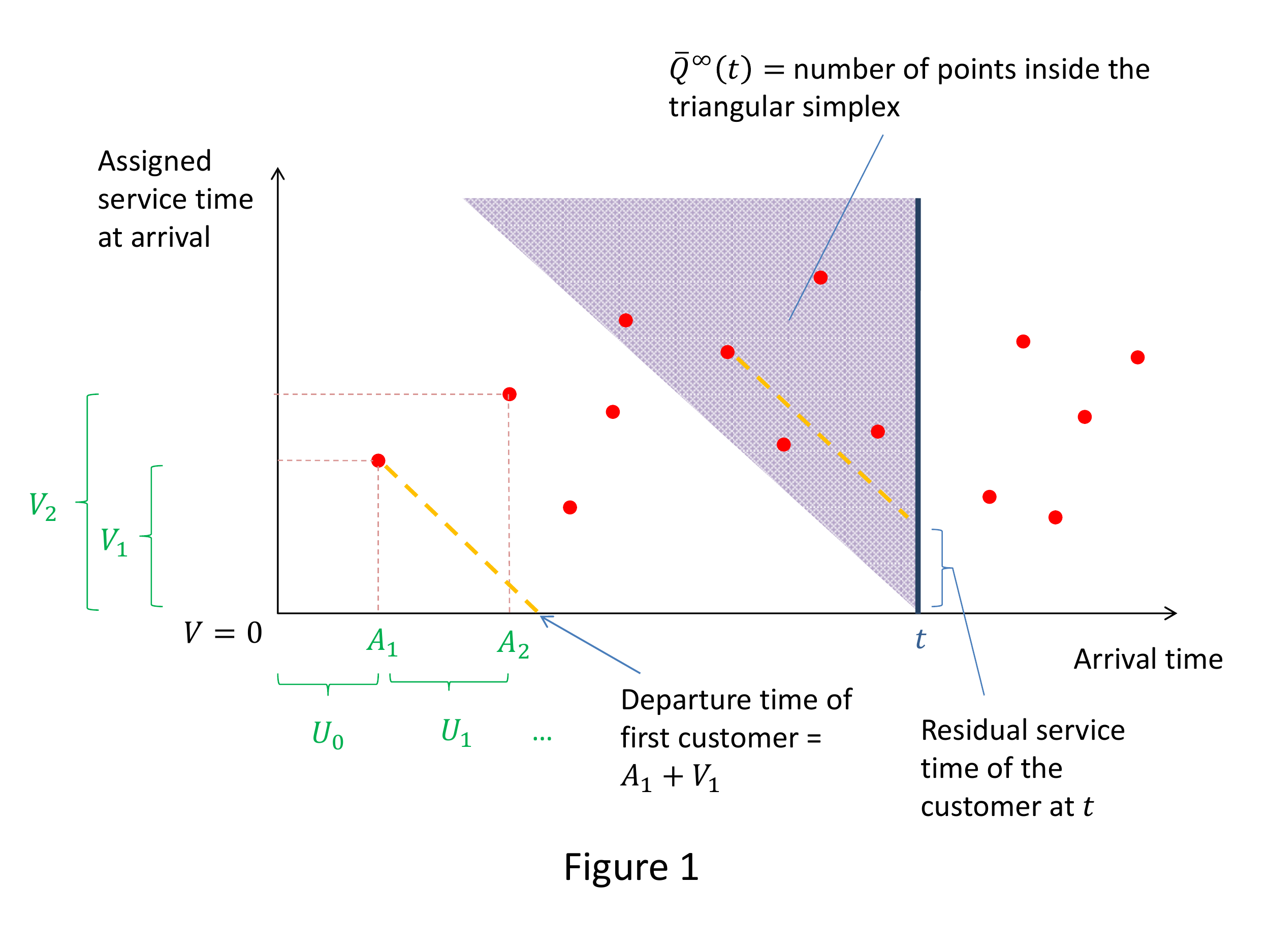}
} \subfigure{
\includegraphics[scale=.3]{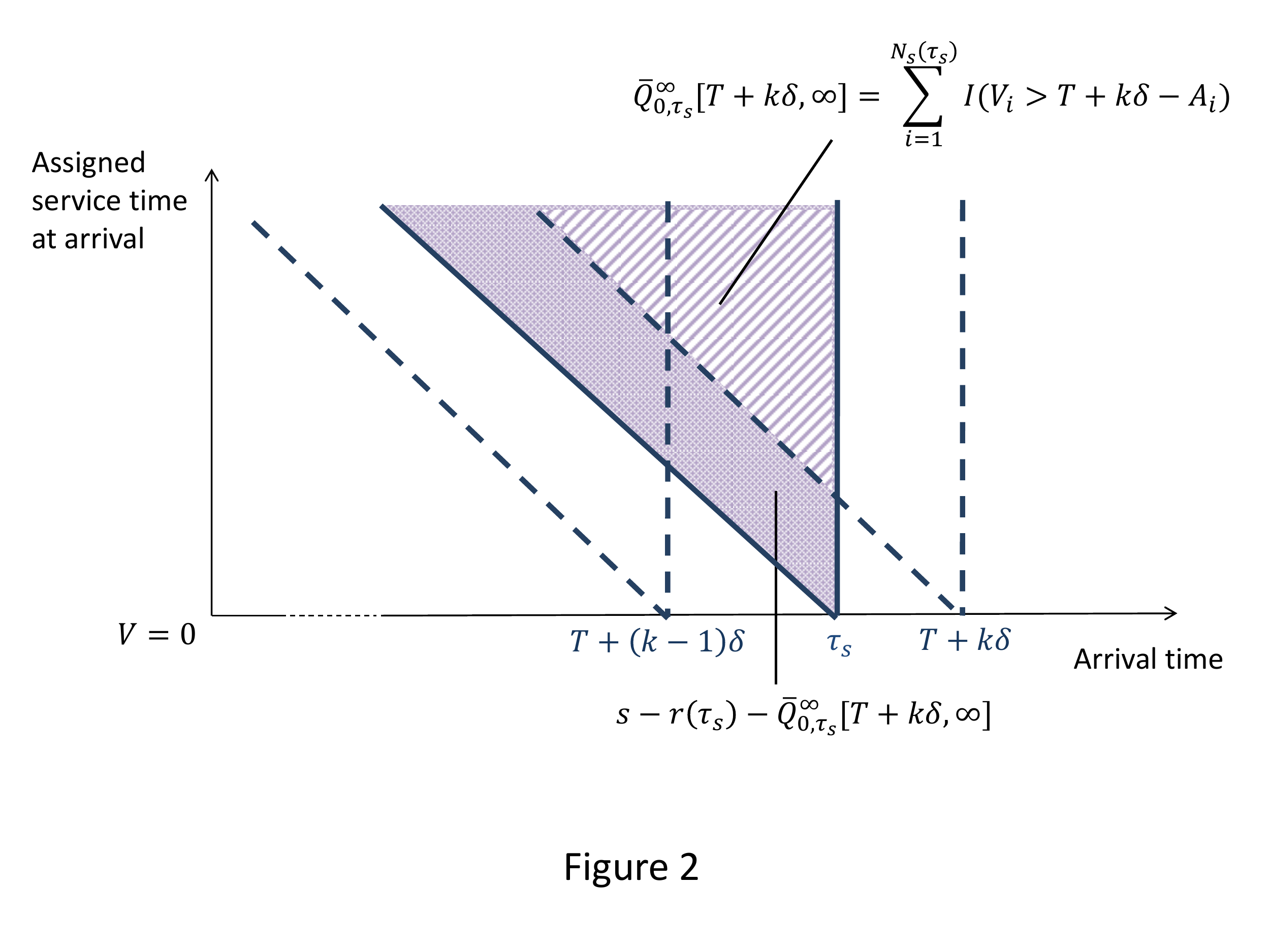}
} \subfigure{
\includegraphics[scale=.3]{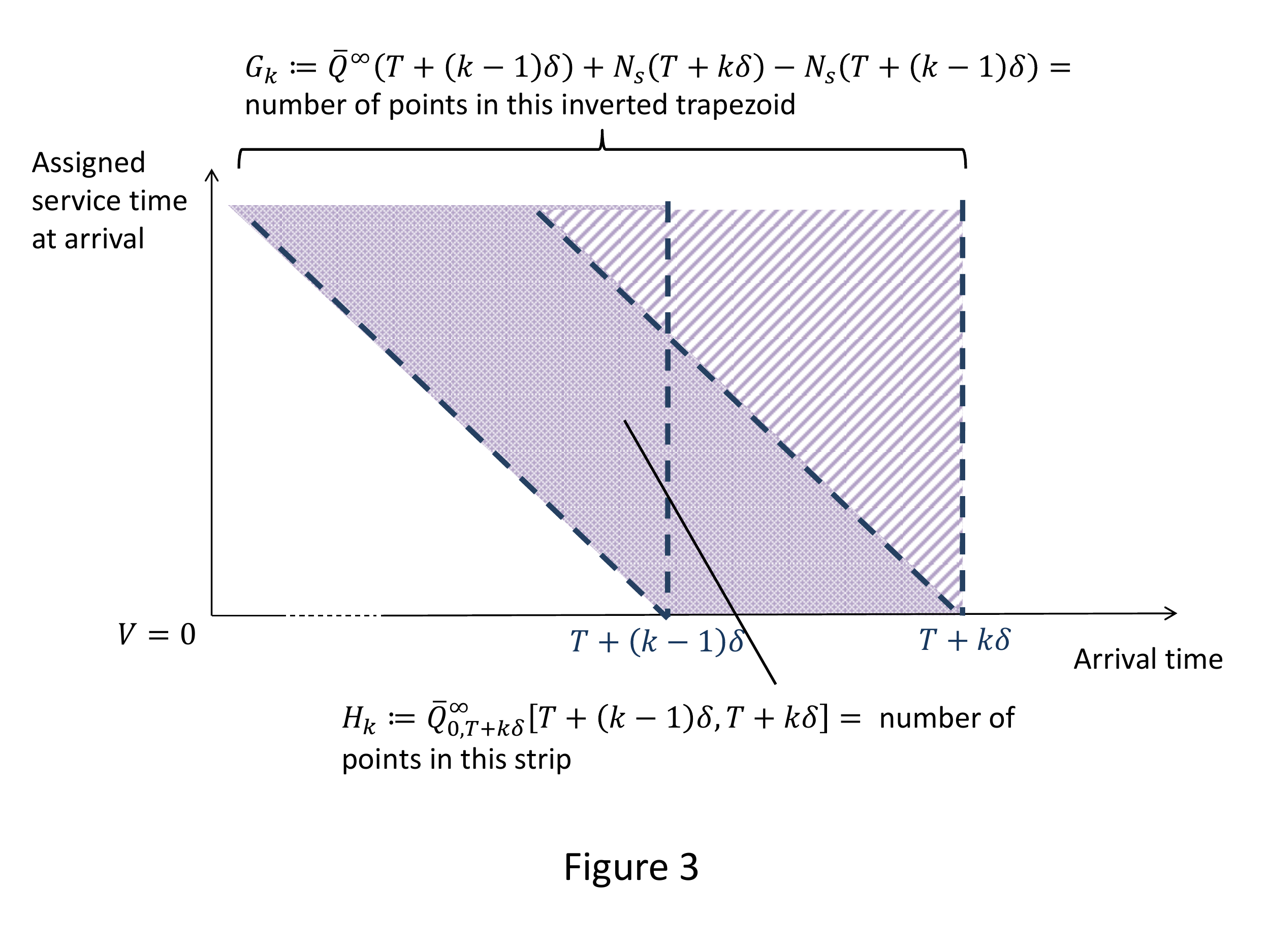}
} \subfigure{
\includegraphics[scale=.3]{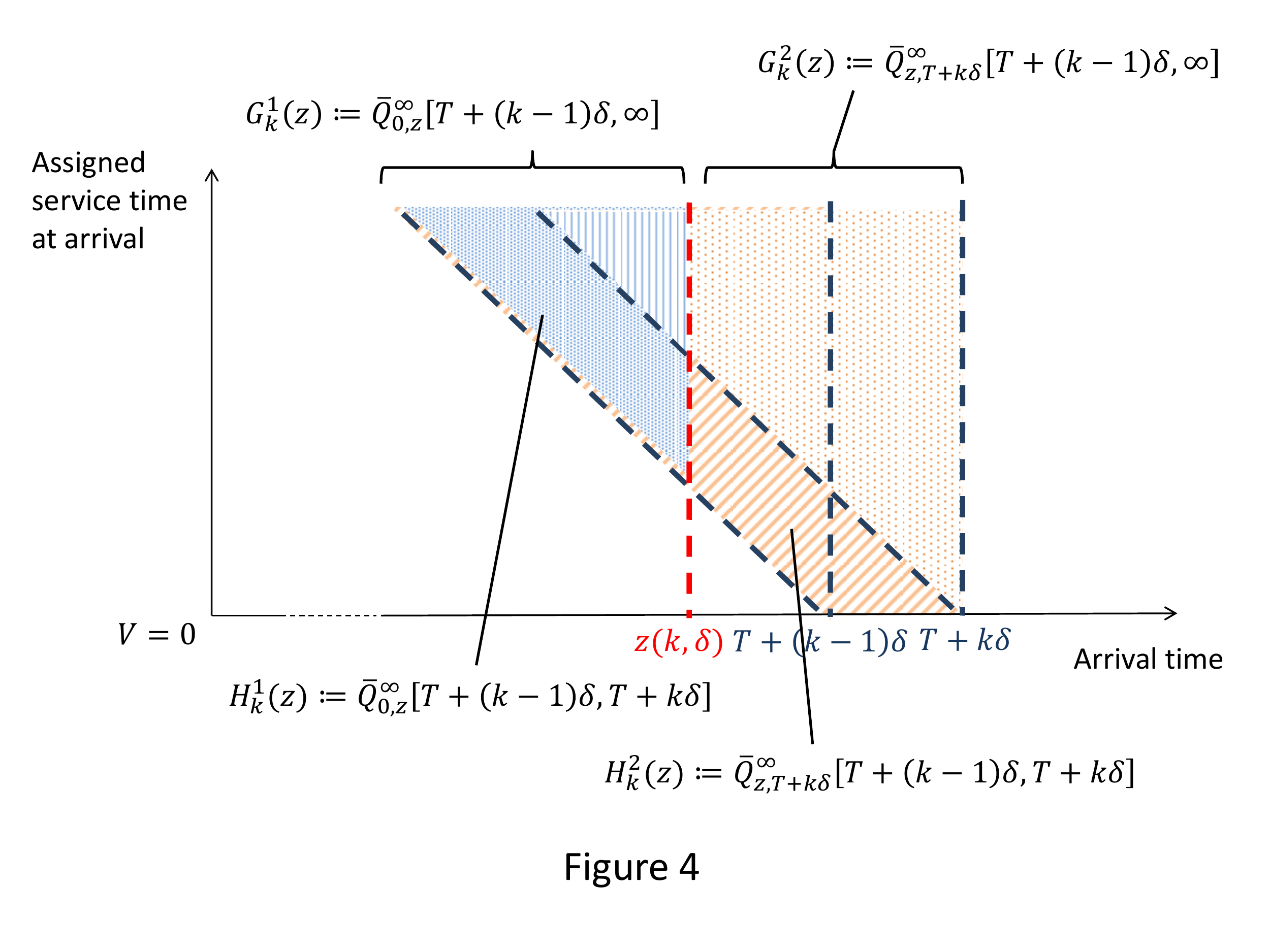}
}
\end{figure}

For notational convenience we denote $\bar{Q}_{t_{1},t_{2}}^{\infty
}[t_{3},t_{4}]:=\sum_{i=N_{s}(t_{1})+1}^{N_{s}(t_{2})}I(t_{3}-A_{i}<V_{i}%
\leq t_{4}-A_{i})$ as the number of customers in the $GI/G/\infty $ system
who arrive sometime in $(t_{1},t_{2}]$ and leave the system sometime in $%
(t_{3},t_{4}]$. It is easy to see, for example, that $\bar{Q}^{\infty }(\tau
_{s},T+k\delta -\tau _{s})=\bar{Q}_{0,\tau _{s}}^{\infty }[T+k\delta ,\infty
]$ for $T+k\delta \geq \tau _{s}$.

Figure 2 shows the region filled in by $\bar{Q}^{\infty }(\tau
_{s},T+k\delta -\tau _{s})=\bar{Q}_{0,\tau _{s}}^{\infty }[T+k\delta ,\infty
]$ as a shifted simplex starting from the point $(\tau _{s},T+k\delta -\tau
_{s})$. Note that by definition $\bar{Q}^{\infty }(\tau _{s})=s+1-r(\tau
_{s})$, and so $s+1-r(\tau _{s})-\bar{Q}_{0,\tau _{s}}^{\infty }[T+k\delta
,\infty ]$ corresponds to the downward strip ending at $(\tau _{s},0)$ and $%
(\tau _{s},T+k\delta -\tau _{s})$, which is obviously smaller than the
region represented by $H_k:=\bar{Q}_{0,T+k\delta }^{\infty }[T+(k-1)\delta
,T+k\delta ]$ in Figure 3.

Define $G_{k}=\bar{Q}^{\infty }(T+(k-1)\delta )+N_{s}(T+k\delta
)-N_{s}(T+(k-1)\delta )$, which is represented by the trapezoidal area
depicted in Figure 3. Observe that $T+(k-1)\delta <\tau _{s}\leq T+k\delta $
implies that one of the triangular simplex corresponding to $\bar{Q}^{\infty
}(t),$ for $T+(k-1)\delta <t\leq T+k\delta ,$ has number of points larger
than $s-r(T+(k-1)\delta )$. This in turn implies that the region represented
by $G_{k}$ has more than $s-r(T+(k-1)\delta )$ number of points.

The above observations lead to
\begin{eqnarray}
&&E_{r}[\exp \{l\theta _{T+k\delta }(s+1-r(\tau _{s})-\bar{Q}_{0,\tau
_{s}}^{\infty }[T+k\delta ,\infty ])\}{};T+(k-1)\delta <\tau _{s}\leq
T+k\delta ]  \notag \\
&\leq &E_{r}[e^{l\theta _{T+k\delta }H_{k}};G_{k}>s-r(T+(k-1)\delta )]
\label{eq6}
\end{eqnarray}

From now on we focus on the case when service time has unbounded support
(the bounded support case is simpler and will be presented later in the proof). We introduce a time point $z=z(k,s)$ and consider the
divisions of areas represented by $H_{k}$ and $G_{k}$ in Figure 4:
\begin{equation*}
\begin{array}{ll}
H_{k}^{1}(z):=\bar{Q}_{0,z}^{\infty }[T+(k-1)\delta ,T+k\delta ] & \subset \
\ G_{k}^{1}(z):=\bar{Q}_{0,z}^{\infty }[T+(k-1)\delta ,\infty ] \\
H_{k}^{2}(z):=\bar{Q}_{z,T+k\delta }^{\infty }[T+(k-1)\delta ,T+k\delta ] &
\subset \ \ G_{k}^{2}(z):=\bar{Q}_{z,T+k\delta }^{\infty }[T+(k-1)\delta
,\infty ]%
\end{array}%
\end{equation*}%
Note that $H_{k}=H_{k}^{1}(z)+H_{k}^{2}(z)$ and $%
G_{k}=G_{k}^{1}(z)+G_{k}^{2}(z)$.

Moreover, define $A_{i}^{k},i=1,\ldots ,G_{k}$ to be the arrival
times of all the customers that $G_{k}$ is counting. Note that given the
arrival times $A_{i}^{k},i=1,\ldots ,G_{k}$, the events whether each
of these customers falls into $H_{k}$ are independent Bernoulli random
variables with probability
\begin{equation}
p_{i}^{k}:=\frac{\bar{F}(T+(k-1)\delta -A_{i}^{k})-\bar{F}(T+k\delta
-A_{i}^{k})}{\bar{F}(T+(k-1)\delta -A_{i}^{k})}  \label{prob}
\end{equation}%
Hence we can write \eqref{eq6} as
\begin{eqnarray}
&&E_{r}[e^{l\theta _{T+k\delta
}(H_{k}^{1}(z)+H_{k}^{2}(z))};G_{k}>s-r(T+(k-1)\delta )]  \notag \\
&=&E_{r}[E_{r}[e^{l\theta _{T+k\delta }(H_{k}^{1}(z)+H_{k}^{2}(z))}|A%
_{i}^{k},i=1,\ldots ,G_{k}];G_{k}>s-r(T+(k-1)\delta )]  \notag \\
&=&E_{r}[E_{r}[e^{l\theta _{T+k\delta }H_{k}^{1}(z)}|A%
_{i}^{k},i=1,\ldots ,G_{k}^{1}(z)]E_{r}[e^{l\theta _{T+k\delta
}H_{k}^{2}(z)}|A_{i}^{k},i=G_{k}^{1}(z)+1,\ldots
,G_{k}^{1}(z)+G_{k}^{2}(z)];{} \notag\\
&&{}G_{k}^{1}(z)+G_{k}^{2}(z)>s-r(T+(k-1)\delta )]
\notag \\
&\leq &E_{r}\left[ e^{l\theta _{T+k\delta
}G_{k}^{1}(z)}\prod_{i=G_{k}^{1}(z)+1}^{G_{k}^{1}(z)+G_{k}^{2}(z)}(1+(e^{l%
\theta _{T+k\delta
}}-1)p_{i}^{k});G_{k}^{1}(z)+G_{k}^{2}(z)>s-r(T+(k-1)\delta )\right]
\label{eq7}
\end{eqnarray}%
Let
\begin{equation}
p_{k}(z):=\sup_{A_{i}^{k}>z}p_{i}^{k}\leq \frac{C\delta }{\bar{F}%
(T+k\delta -z)}  \label{max prob}
\end{equation}%
for some constant $C>0$, where the inequality follows from \eqref{prob}. Also let
\begin{align*}
\psi _{s,z,k}^{1}(\theta )& :=\log Ee^{\theta
G_{k}^{1}(z)}=s\int_{0}^{z}\psi _{N}(\log (e^{\theta }\bar{F}(T+(k-1)\delta
-u)+F(T+(k-1)\delta -u)))du+o(s) \\
\psi _{s,z,k}^{2}(\theta )& :=\log Ee^{\theta
G_{k}^{2}(z)}=s\int_{z}^{T+k\delta }\psi _{N}(\log (e^{\theta }\bar{F}%
(T+(k-1)\delta -u)+F(T+(k-1)\delta -u)))du+o(s)
\end{align*}%
where $o(s)$ is uniform in $\theta$, $k$ and $z$. This is due to the following lemma, whose proof will be deferred to the appendix:

\begin{lemma}
We have
$$\frac{1}{s}\log Ee^{\theta\bar{Q}_{w,z}^\infty[t,\infty]}\to\int_{w}^{z}\psi _{N}(\log (e^{\theta }\bar{F}(t-u)+F(t-u)))du$$
uniformly over $\theta\in[\theta_\infty,\theta_T]$, $t\geq T$ and $0\leq w\leq z\leq t+\eta$ for any $\eta>0$. \label{uniformity}
\end{lemma}

%\bigskip

When $p_{k}(z)$ is small
enough, \eqref{eq7} is less than or equal to
\begin{eqnarray}
&&E_{r}[e^{l\theta _{T+k\delta }G_{k}^{1}(z)}(1+(e^{l\theta _{T+k\delta
}}-1)p_{k}(z))^{G_{k}^{2}(z)};G_{k}^{1}(z)+G_{k}^{2}(z)>s-r(T+(k-1)\delta )]
\notag \\
&=&E_{r}[E_{r}[e^{l\theta _{T+k\delta }G_{k}^{1}(z)+\log (1+(e^{l\theta
_{T+k\delta }}-1)p_{k}(z))G_{k}^{2}(z)};G_{k}^{2}(z)>s-r(T+(k-1)\delta
)-G_{k}^{1}(z)|G_{k}^{1}(z),B(z)]]  \notag \\
&\leq &E_{r}[\exp \{l\theta _{T+k\delta }G_{k}^{1}(z)-\theta _{T+(k-1)\delta
}(s-r(T+(k-1)\delta )-G_{k}^{1}(z)){}  \notag \\
&&{}+\psi _{s,z,k}^{2}(\log (1+(e^{l\theta _{T+k\delta }}-1)p_{k}(z))+\theta
_{T+(k-1)\delta })\}]  \notag \\
&=&\exp \Bigg\{\psi _{s,z,k}^{1}(l\theta _{T+k\delta }+\theta
_{T+(k-1)\delta })-\theta _{T+(k-1)\delta }(s-r(T+(k-1)\delta )){}  \notag \\
&&{}+\psi _{s,z,k}^{2}(\log (1+(e^{l\theta _{T+k\delta }}-1)p_{k}(z))+\theta
_{T+(k-1)\delta })\Bigg\}  \notag \\
&=&\exp \Bigg\{s\int_{0}^{z}\psi _{N}(\log (e^{l\theta _{T+k\delta }+\theta
_{T+(k-1)\delta }}\bar{F}(T+(k-1)\delta -u)+F(T+(k-1)\delta -u)))du{}  \notag \\
&&{}-s\int_{0}^{z}\psi _{N}(\log (e^{\log (1+(e^{l\theta _{T+k\delta
}}-1)p_{k}(z))+\theta _{T+(k-1)\delta }}\bar{F}(T+(k-1)\delta -u)+F(T+(k-1)\delta
-u)))du{}  \notag \\
&&{}-\theta _{T+(k-1)\delta }(s-r(T+(k-1)\delta ))+s\psi _{T+(k-1)\delta }(\log
(1+(e^{l\theta _{T+k\delta }}-1)p_{k}(z))+\theta _{T+(k-1)\delta }){} \notag\\
&&{}+o(s)%
\Bigg\}  \label{eq8}
\end{eqnarray}%
where the inequality follows by Chernoff's inequality, and the last equality
follows from
\begin{equation*}
\psi _{s,z,k}^{2}(\theta )=s\psi _{T+(k-1)\delta }(\theta )-s\int_{0}^{z}\psi
_{N}(\log (e^{\theta }\bar{F}(T+(k-1)\delta -u)+F(T+(k-1)\delta -u)))du+o(s)
\end{equation*}%
uniformly, by Lemma \ref{uniformity}.

Now let $\rho _{s}\nearrow \infty $ be a sequence satisfying $s\bar{F}(\rho
_{s})\nearrow \infty $, whose existence is guaranteed by the unbounded support assumption. We divide into two cases: For $T+(k-1)\delta \leq \rho
_{s}$, we put $z=0$ and so by \eqref{max prob} and we have $p_{k}(0)\searrow
0$ as $s\nearrow\infty$ (recall $\delta =O(1/s)$). Consequently \eqref{eq8} becomes
\begin{equation*}
\exp \{-\theta _{T+(k-1)\delta }(s-r(T+(k-1)\delta ))+s\psi _{T+(k-1)\delta
}(\log (1+(e^{l\theta _{T+k\delta }}-1)p_{k}(z))+\theta _{T+(k-1)\delta
})+o(s)\}=e^{-sI_{T+(k-1)\delta }+o(s)}
\end{equation*}%
For $T+(k-1)\delta >\rho _{s}$, we put $z=T+(k-1)\delta -\rho _{s}$ so that $%
T+(k-1)\delta -z=\rho _{s}$. Hence again $p_{k}(z)\searrow 0$. Also,
\begin{eqnarray*}
&&\int_{0}^{z}\psi _{N}(\log (e^{l\theta _{T+k\delta }+\theta
_{T+(k-1)\delta }}\bar{F}(T+(k-1)\delta -u)+F(T+(k-1)\delta -u)))du \\
&=&\int_{T+(k-1)\delta -z}^{T+(k-1)\delta }\psi _{N}(\log (e^{l\theta _{T+k\delta
}+\theta _{T+(k-1)\delta }}\bar{F}(u)+F(u)))du \\
&\leq &\int_{T+(k-1)\delta -z}^{\infty }C_1\lambda (e^{l\theta _{T+k\delta
}+\theta _{T+(k-1)\delta }}-1)\bar{F}(u)du \\
&=&C_2\lambda \int_{\rho_s}^{\infty }\bar{F}(u)du=o(1)
\end{eqnarray*}%
for large enough $T+(k-1)\delta -z=\rho_s$ and some constants $C_1,C_2>0$, due to the fact
that $\log (1+x)\leq x$ for $x>0$ and that $\psi _{N}^{\prime }(0)=\lambda $%
. It is now obvious that \eqref{eq8} also becomes $e^{-sI_{T+(k-1)\delta
}+o(s)}$ in this case.

Hence \eqref{eq9} is less than or equal to
\begin{eqnarray*}
&&e^{-sI^{\ast }/l+o(s)}\sum_{k=1}^{\infty
}(E_{r}N_{A}^{2p})^{1/p}(E_{r}\tau _{A}{}^{3q})^{1/q}(P_{r}(\tau
_{A}>T+(k-1)\delta ))^{1/h} \\
&\leq &e^{-sI^{\ast }/l+o(s)}(E_{r}N_{A}^{2p})^{1/p}(E_{r}\tau
_{A}{}^{3q})^{1/q}\left( (P_{r}(\tau _{A}>T))^{1/h}+\frac{1}{\delta }%
\int_{T}^{\infty }(P_{r}(\tau _{A}>u))^{1/h}du\right)
\end{eqnarray*}%
From this, and using Lemma \ref{asymptotic}, we get
\begin{equation*}
\limsup_{s\rightarrow \infty }\frac{1}{s}\log E_{r}\Bigg[N_{A}^{2}\tau
_{s}{}^{2}\exp \left\{ \theta _{\lceil \tau _{s}\rceil }\left( sa_{\lceil
\tau _{s}\rceil }+1-\bar{Q}^{\infty }(\tau _{s},\lceil \tau _{s}\rceil -\tau
_{s})\right) \right\} ;\tau _{s}>T;\tau _{s}<\tau _{A}\Bigg]\leq -\frac{%
I^{\ast }}{l}
\end{equation*}%
Since $l$ is arbitrarily close to 1, we have proved \eqref{eq5}.

Finally, we consider the case when $V$ has bounded support over $[0,M]$. Pick a small constant $a>0$, and consider the set of customers $\tilde{G}_k=\bar{Q}_{(T+(k-1)\delta-M)\vee0,T+k\delta}[T+(k-1)\delta-a,\infty]$ that consists of $G_k$ and a trapezoidal strip of width $a$ running through $(T+(k-1)\delta-a,0)$, $(T+(k-1)\delta,0)$, $((T+(k-1)\delta-M)\vee0,M\wedge(T+(k-1)\delta))$ and $((T+(k-1)\delta-M)\vee0,M\wedge(T+(k-1)\delta)-a)$. See Figure 5.

\begin{figure}[ht]
\centering
\subfigure{
\includegraphics[scale=.35]{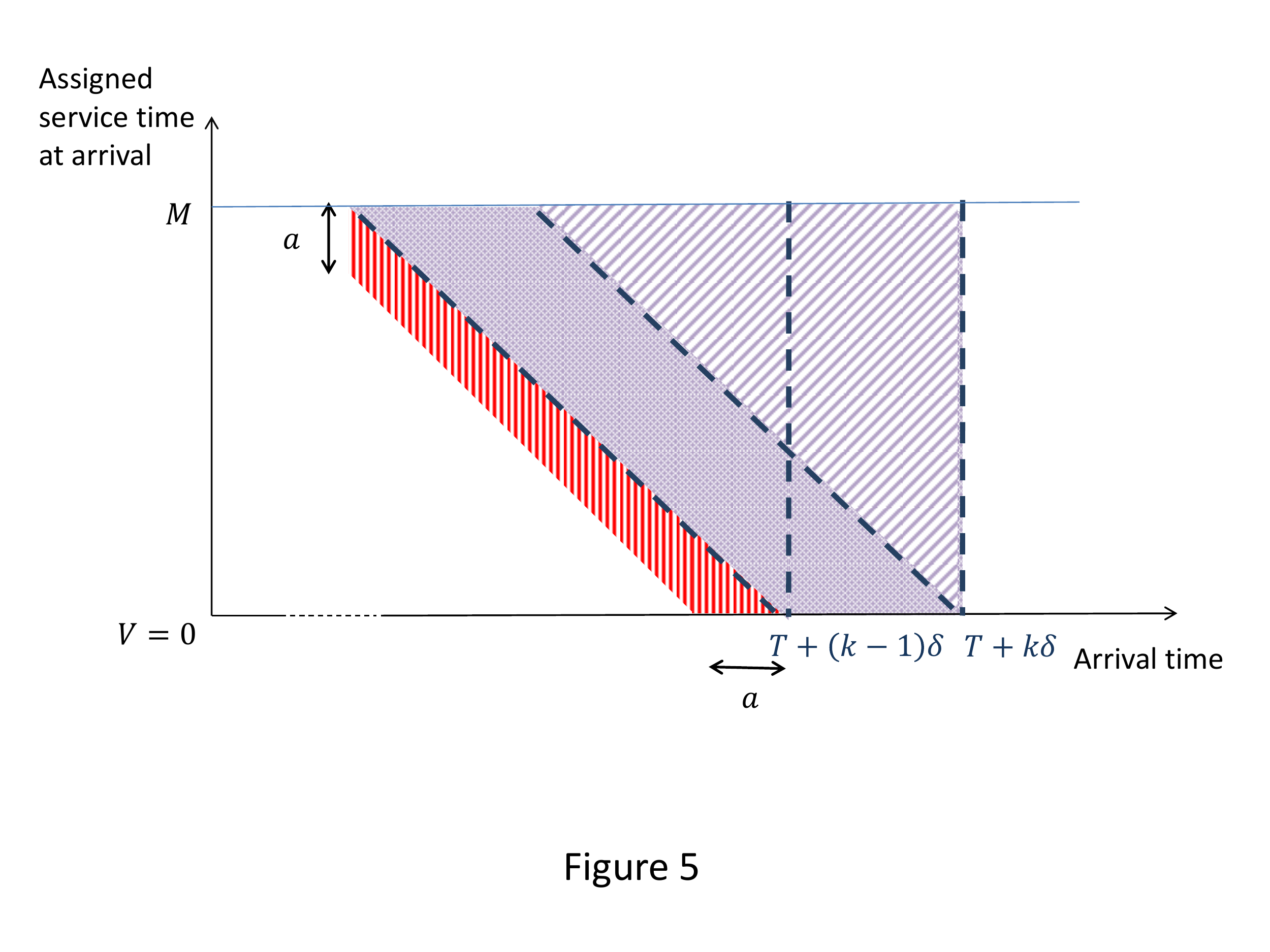}
}
\end{figure}

Denote $\tilde{A}_i^k,i=1,\ldots,\tilde{G}_k$ as the arrival times of customers falling in $\tilde{G}_k$. Then we have
\begin{eqnarray}
&&E_r[e^{l\theta_{T+k\delta}H_k};G_k>s-r(T+(k-1)\delta)] \notag\\
&\leq&E_r[e^{l\theta_{T+k\delta}H_k};\tilde{G}_k>s-r(T+(k-1)\delta)] \notag\\
&=&E_r[E_r[e^{l\theta_{T+k\delta}H_k}|\tilde{A}_i^k,i=1,\ldots,\tilde{G}_k];\tilde{G}_k>s-r(T+(k-1)\delta)] \notag\\
&=&E_r\left[\prod_{i=1}^{\tilde{G}_k}(1+(e^{l\theta_{T+k\delta}})\tilde{p}_i^k);\tilde{G}_k>s-r(T+(k-1)\delta)\right] \label{interim bounded1}
\end{eqnarray}
where
$$\tilde{p}_i^k=\frac{\bar{F}(T+(k-1)\delta-\tilde{A}_i^k)-\bar{F}(T+k\delta-\tilde{A}_i^k)}{\bar{F}(T+(k-1)\delta-a-\tilde{A}_i^k)}\leq \tilde{p}_k:=\sup_{i=1,\ldots,\tilde{G}_k}\tilde{p}_i^k\leq\frac{C\delta}{\bar{F}(M-a)}$$
Hence \eqref{interim bounded1} is less than or equal to
\begin{eqnarray}
&&E_r[e^{\log(1+(e^{l\theta_{T+k\delta}})\tilde{p}_k)\tilde{G}_k};\tilde{G}_k>s-r(T+(k-1)\delta)] \notag\\
&\leq&e^{-\theta_{T+(k-1)\delta}(s-r(T+(k-1)\delta))+\tilde{\psi}_k(\log(1+(e^{l\theta_{T+k\delta}}-1)\tilde{p})+\theta_{T+(k-1)\delta})} \label{interim bounded2}
\end{eqnarray}
where $\tilde{\psi}_k(\theta):=\log Ee^{\theta\tilde{G}_k}$, by Chernoff's inequality. Now note that by Lemma \ref{uniformity} we have
\begin{align*}
\tilde{\psi}_k(\theta)&=s\int_{(T+(k-1)\delta-M)\vee0}^{T+k\delta}\psi_N(\log(e^\theta\bar{F}(T+(k-1)\delta-a-u)+F(T+(k-1)\delta-a-u)))du+o(s)\\
&=s\int_0^{(M-a)\wedge(T+(k-1)\delta-a)}\psi_N(\log(e^\theta\bar{F}(u)+F(u)))du+s\psi_N(\theta)(a+\delta)+o(s)\\
&\leq s\psi_{T+(k-1)\delta}(\theta)+saC+o(s)
\end{align*}
for some constant $C>0$, uniformly in $\theta$ and $k$. Hence \eqref{interim bounded2} is less than or equal to
\begin{eqnarray*}
&&e^{-\theta_{T+(k-1)\delta}(s-r(T+(k-1)\delta))+s\psi_{T+(k-1)\delta}(\theta_{T+(k-1)\delta})+saC+o(s)}\\
&=&e^{-sI_{T+(k-1)\delta}+saC+o(s)}
\end{eqnarray*}
Thus \eqref{eq9} is less than or equal to
$$e^{-sI^*/l+saC/l+o(s)}\sum_{k=1}^{\infty }(E_{r}N_{A}^{2p})^{1/p}(E_{r}\tau
_{A}{}^{3q})^{1/q}(P_{r}(\tau _{A}>T+(k-1)\delta ))^{1/h}$$
This gives
$$\limsup_{s\rightarrow \infty }\frac{1}{s}\log E_{r}\Bigg[N_{A}^{2}\tau
_{s}{}^{3}\exp \left\{ \theta _{\lceil \tau _{s}\rceil }\left( sa_{\lceil
\tau _{s}\rceil }+1-\bar{Q}^{\infty }(\tau _{s},\lceil \tau _{s}\rceil -\tau
_{s})\right) \right\} ;\tau _{s}>T;\tau _{s}<\tau _{A}\Bigg]\leq -\frac{I^*}{l}+\frac{aC}{l}$$
Since $l$ and $a$ can be chosen arbitrarily close to 1 and 0 respectively, \eqref{eq5} holds and conclusion follows.

\end{proof}

%\bigskip

\begin{remark}
The proof can be simplified in the case of $M/G/s$ system. In particular,
there is no need to condition on $A_{i}^k$ nor introduce the constant $a$ in
the case of bounded support $V$. Since arrival is Poisson, the
two-dimensional description of arrivals via the arrival time and the
required service time at the time of arrival leads to a Poisson random
measure. Hence all the points in $G_k$ are independently sampled, each with
probability of falling into $H_k$ being
\begin{equation*}
p_k:=\frac{\int_{0}^{T+k\delta }(\bar{F}(T+(k-1)\delta -u)-\bar{F}(T+k\delta
-u))du}{\int_{0}^{T+k\delta }\bar{F}(T+(k-1)\delta -u)du}\leq\frac{C\delta
(M+\delta )}{\int_{0}^{T+(k-1)\delta }\bar{F}(u)du+N_{s}((k-1)\delta
,k\delta )}=O(\delta)
\end{equation*}%
for some constant $C>0$. Then \eqref{eq10} immediately becomes
\begin{eqnarray*}
&&E_{r}[(p_ke^{l\theta _{T+k\delta }}+1-p_k)^{G_k};G_k>s-r(T+(k-1)\delta)] \\
&=&E_{r}[e^{O(\delta )G_k};G_k>s-r(T+(k-1)\delta)]
\end{eqnarray*}%
The rest follows similarly as in the proof.
\end{remark}

%\bigskip

\begin{remark}
Note that the result coincides with Erlang's loss formula in the case of $%
M/G/s$ (see for example Asmussen (2003)), which states that the loss
probability is exactly given by
\begin{equation*}
P_{\pi }(\text{loss})=\frac{(\lambda sEV)^{s}/s!}{1+\lambda sEV+\cdots
+(\lambda sEV)^{s}/s!}
\end{equation*}%
Simple calculation reveals that $(1/s)\log P_{\pi }(\text{loss})\rightarrow
\log (\lambda EV)+1-\lambda EV=-I^{\ast }$.
\end{remark}

%\bigskip

The next result we will discuss is the lower bound:

\begin{thm}
For any $r(\cdot )\in J(\cdot )$, we have
\begin{equation*}
\liminf_{s\rightarrow \infty }\frac{1}{s}\log P_{r}(\tau _{s}<\tau _{A})\geq
-I^{\ast }
\end{equation*}%
\label{lower bound copy(1)}
\end{thm}

%\bigskip

It suffices to prove that $\liminf_{s\rightarrow \infty }(1/s)\log
P_{r}(\tau _{s}<\tau _{A})\geq -I_{t_{n}}$ for a sequence $t_{n}\nearrow
\infty $ thanks to Lemma \ref{rate} Part 1 and 2. In fact we will take $%
t_{n}=n\Delta $. In the case of bounded support $V$, it suffices to only
consider $n\Delta=\lceil M\rceil$ because of Lemma \ref{rate} Part 3. For
each $n\Delta $, the idea then is to identify a so-called optimal sample
path (or more precisely a neighborhood of such path) that possesses a rate
function $I_{n\Delta }$ and has the property $\tau _{s}<\tau _{A}$. Note
that the probability in consideration is the same for $GI/G/s$ and $%
GI/G/\infty$ systems. Henceforth we would consider paths in $GI/G/\infty$.

The way we define $A$ in \eqref{A} implies that it suffices to focus on the
process on the time-grid $\{0,\Delta ,2\Delta ,\ldots \}$ for checking the
condition $\tau _{s}<\tau _{A}$. For a path to reach $s$ at time $n\Delta $,
the form of $\psi _{n\Delta }^{\prime }(\theta _{n\Delta })$ hints that $E[%
\bar{Q}_{(k-1)\Delta ,k\Delta }^{\infty }[(j-1)\Delta ,j\Delta
]|Q^\infty(n\Delta)>s]=s\alpha_{kj}+o(s)$ and $E[\bar{Q}_{(k-1)\Delta
,k\Delta }^{\infty }[n\Delta ,\infty ]|Q^\infty(n\Delta)>s]=s\beta_k+o(s)$
where
\begin{equation*}
\alpha _{kj}:=\int_{(k-1)\Delta }^{k\Delta }\psi _{N}^{\prime }(\log
(e^{\theta _{n\Delta }}\bar{F}(n\Delta -u)+F(n\Delta -u)))\frac{F(j\Delta
-u)-F((j-1)\Delta -u)}{e^{\theta _{n\Delta }}\bar{F}(n\Delta -u)+F(n\Delta
-u)}du
\end{equation*}
and
\begin{equation*}
\beta _{k}:=\int_{(k-1)\Delta }^{k\Delta }\psi _{N}^{\prime }(\log
(e^{\theta _{n\Delta }}\bar{F}(n\Delta -u)+F(n\Delta -u)))\frac{e^{\theta
_{n\Delta }}\bar{F}(n\Delta -u)}{e^{\theta _{n\Delta }}\bar{F}(n\Delta
-u)+F(n\Delta -u)}du
\end{equation*}
for $k=1,\ldots ,n$, $j=k,\ldots ,n$. Our goal is to rigorously justify that
such a path is the optimal sample path discussed above.

We now state two useful lemmas. The first is a generalization of Glynn
(1995), whose proof resembles this earlier work and is deferred to the
appendix. The second one argues that the path we identified indeed satisfies
$\tau_s<\tau_A$:

\begin{lemma}
Let $\mathbf{\Theta }=(\theta _{kj},\theta _{k\cdot })_{k=1,\ldots
,n,j=k,\ldots ,n}\in \mathbb{R}^{n(n+1)/2+n}$, and define
\begin{equation*}
\bar{\psi}(\mathbf{\Theta })=\sum_{k=1}^{n}\int_{(k-1)\Delta }^{k\Delta
}\psi _{N}\left( \log \left( \sum_{j=k}^{n}e^{\theta _{kj}}P((j-1)\Delta
-u<V\leq j\Delta -u)+e^{\theta _{k\cdot }}\bar{F}(n\Delta-u)\right) \right)
du
\end{equation*}%
We have
\begin{equation*}
\frac{1}{s}\log E\exp \left\{ \sum_{k=1}^{n}\left( \sum_{j=k}^{n}\theta _{kj}%
\bar{Q}_{(k-1)\Delta ,k\Delta }^{\infty }[(j-1)\Delta ,j\Delta ]+\theta
_{k\cdot }\bar{Q}_{(k-1)\Delta ,k\Delta }^{\infty }[n\Delta,\infty]\right)
\right\} \rightarrow \bar{\psi}(\mathbf{\Theta })
\end{equation*}
\label{multivariate psi}
\end{lemma}

%\bigskip

\begin{lemma}
Starting with any $r(\cdot)\in J(\cdot)$, the sample path with $%
Q_{(k-1)\Delta ,k\Delta }^{\infty }[(j-1)\Delta ,j\Delta ]\in ((\alpha
_{kj}+\gamma_{kj})s,(\alpha _{kj}+\epsilon )s)$, $Q_{(k-1)\Delta ,k\Delta
}^{\infty }[n\Delta,\infty]\in ((\beta _{k}+\gamma_k)s,(\beta _{k}+\epsilon
)s)$ for all $k=1,\ldots ,n$ and $j=k,\ldots ,n$ satisfies $\tau _{s}<\tau
_{A}$. Here $\gamma_{kj},\gamma_k>0$, $\sum_{\substack{ k=1,\ldots,n  \\ %
j=k,\ldots,n}}\gamma_{kj}+\sum_{k=1,\ldots,n}\gamma_k=\gamma<\infty$ and $%
\epsilon >\gamma_{kj},\epsilon>\gamma_k$. \label{path property}
\end{lemma}

%\bigskip

\begin{proof} For $l=1,\ldots ,n$, consider
\begin{eqnarray*}
\bar{Q}^{\infty }(l\Delta ) &=&\sum_{k=1}^{l}Q_{(k-1)\Delta ,k\Delta
}^{\infty }[l\Delta ,\infty] \\
&>&\sum_{k=1}^{l}\left( \sum_{j=l+1}^{n}a_{kj}s+b_{k}s\right)+\sum_{k=1}^{l}\left(\sum_{j=l+1}^{n}\gamma_{kj}s+\gamma_ks\right)\\
&=&s\sum_{k=1}^{l}\Bigg(\sum_{j=l+1}^{n}\int_{(k-1)\Delta }^{k\Delta }\psi
_{N}^{\prime }(\log (e^{\theta _{n\Delta}}\bar{F}(n\Delta-u)+F(n\Delta-u)))\frac{F(j\Delta
-u)-F((j-1)\Delta -u)}{e^{\theta _{n\Delta}}\bar{F}(n\Delta-u)+F(n\Delta-u)}du{} \\
&&{}+\int_{(k-1)\Delta }^{k\Delta }\psi _{N}^{\prime }(\log (e^{\theta _{n\Delta}}%
\bar{F}(n\Delta-u)+F(n\Delta-u)))\frac{e^{\theta _{n\Delta}}\bar{F}(n\Delta-u)}{e^{\theta _{n\Delta}}\bar{F%
}(n\Delta-u)+F(n\Delta-u)}du\Bigg){} \\
&&{}+s\sum_{k=1}^{l}\left(\sum_{j=l+1}^{n}\gamma_{kj}+\gamma_k\right)\\
&=&s\int_{0}^{l\Delta }\psi _{N}^{\prime }(\log (e^{\theta _{n\Delta}}\bar{F}%
(n\Delta-u)+F(n\Delta-u)))\frac{e^{\theta _{n\Delta}}\bar{F}(n\Delta-u)+F(n\Delta-u)-F(l\Delta -u)}{%
e^{\theta _{n\Delta}}\bar{F}(n\Delta-u)+F(n\Delta-u)}du{}\\
&&{}+s\sum_{k=1}^{l}\left(\sum_{j=l+1}^{n}\gamma_{kj}+\gamma_k\right)\\
&>&\lambda s\int_{0}^{l\Delta }\bar{F}(l\Delta -u)du+C_1\sqrt{s}
\end{eqnarray*}%
for any given constant $C_1$, when $s$ is large enough. The last inequality follows from the monotonicity of $\psi_N'$. Note that we then have $%
Q^{\infty }(l\Delta )=\bar{Q}^{\infty }(l\Delta )+r(l\Delta )>\lambda
s+C_{2}\sqrt{s}$ for any given constant $C_{2}$ and large enough $s$. Hence $%
\tau _{A}$ is not reached in time $n\Delta$ when $s$ is large.

On the other hand,
\begin{eqnarray*}
\bar{Q}^{\infty }(n\Delta)& =&\sum_{k=1}^{n}Q_{(k-1)\Delta ,k\Delta }^{\infty }[n\Delta,\infty]
\\
& >&\sum_{k=1}^{n}\beta _{k}s+\sum_{k=1}^n\gamma_ks\\
& =&s\sum_{k=1}^{m}\int_{(k-1)\Delta }^{k\Delta }\psi _{N}^{\prime }(\log
(e^{\theta _{n\Delta}}\bar{F}(n\Delta-u)+F(n\Delta-u)))\frac{e^{\theta _{n\Delta}}\bar{F}(n\Delta-u)}{%
e^{\theta _{n\Delta}}\bar{F}(n\Delta-u)+F(n\Delta-u)}du +s\sum_{k=1}^n\gamma_k{}\\
& =&s\int_{0}^{n\Delta}\psi _{N}^{\prime }(\log (e^{\theta _{n\Delta}}\bar{F}%
(n\Delta-u)+F(n\Delta-u)))\frac{e^{\theta _{n\Delta}}\bar{F}(n\Delta-u)}{e^{\theta _{n\Delta}}\bar{F}%
(n\Delta-u)+F(n\Delta-u)}du+s\sum_{k=1}^n\gamma_k\\
& =&s\psi_{n\Delta}'(\theta_{n\Delta})+s\sum_{k=1}^n\gamma_k
\end{eqnarray*}%
where the last equality follows from the definition of $\theta _{n\Delta}$. So $Q^\infty(n\Delta)=\bar{Q}^\infty(n\Delta)+r(n\Delta)>s$ when $s$ is large enough. This concludes our proof.

\end{proof}

%\bigskip

We now prove Theorem \ref{lower bound copy(1)}:

\begin{proof}[Proof of Theorem \ref{lower bound copy(1)}]
Note that by Lemma \ref%
{path property}, for any $r(\cdot)\in J(\cdot)$ and $s$ large enough,
\begin{eqnarray}
&&P_{r}(\tau _{s}<\tau _{A})  \notag \\
&\geq &P_{r}(Q_{(k-1)\Delta ,k\Delta }^{\infty }[(j-1)\Delta ,j\Delta ]\in
((\alpha _{kj}+\gamma_{kj})s,(\alpha _{kj}+\epsilon )s),\ Q_{(k-1)\Delta ,k\Delta
}^{\infty }[n\Delta,\infty]\in ((\beta _{k}+\gamma_k)s,(\beta _{k}+\epsilon )s),{}  \notag \\
&&{}k=1,\ldots ,n,\ j=k,\ldots ,n)  \label{confined prob}
\end{eqnarray}%
for large enough $s$ given arbitrary $\gamma_{kj}$, $\gamma_k$ and $\epsilon$ satisfying conditions in Lemma \ref{path property}. Denote $\mathbf{\Gamma}=(\gamma_{kj},\gamma_k)_{k=1,\ldots,n,\ j=k,\ldots,n}$. Let
\begin{equation*}
S_{\mathbf{\Gamma}}=\prod_{k=1}^{n}\prod_{j=k}^{n}(\alpha _{kj}+\gamma_{kj},\alpha _{kj}+\epsilon )\times
\prod_{k=1}^{n}(\beta _{k}+\gamma_k,\beta _{k}+\epsilon )\subset \mathbb{R}%
^{n(n+1)/2+n}
\end{equation*}%
Using Gartner-Ellis Theorem for \eqref{confined prob} and Lemma \ref{multivariate psi}, we have
\begin{eqnarray}
&&\frac{1}{s}\log P_{r}(Q_{(k-1)\Delta ,k\Delta }^{\infty }[(j-1)\Delta
,j\Delta ]\in ((\alpha _{kj}+\gamma_{kj})s,(\alpha _{kj}+\epsilon )s),{} \notag\\
&&Q_{(k-1)\Delta,k\Delta }^{\infty }[n\Delta,\infty]\in ((\beta _{k}+\gamma_k)s,(\beta _{k}+\epsilon )s),\ k=1,\ldots ,n,\ j=k,\ldots ,n) \notag\\
&\rightarrow &-I_{\mathbf{\Gamma}} \label{lower Gartner-Ellis}
\end{eqnarray}%
where $I_{\mathbf{\Gamma}}=\inf_{\mathbf{x}\in S_{\mathbf{\Gamma}}}I(\mathbf{x})$ and
\begin{equation*}
I(\mathbf{x})=\sup_{\mathbf{\Theta }\in \mathbb{R}^{n(n+1)/2+n}}\{\langle
\mathbf{\Theta },\mathbf{x}\rangle -\bar{\psi}(\mathbf{\Theta })\}
\end{equation*}%
with $\bar{\psi}(\mathbf{\Theta })$ defined in Lemma \ref{multivariate psi}. But note that for $k=1,\ldots ,n$, $j=k,\ldots ,n$,
\begin{eqnarray}
\frac{\partial }{\partial \theta _{kj}}(\langle \mathbf{\Theta },\mathbf{x}%
\rangle -\bar{\psi}(\mathbf{\Theta })) &=&x_{kj}-\int_{(k-1)\Delta
}^{k\Delta }\psi _{N}^{\prime }\left( \log \left( \sum_{j=k}^{n}e^{\theta
_{kj}}P((j-1)\Delta -u<V\leq j\Delta -u)+e^{\theta _{k\cdot }}\bar{F}%
(n\Delta-u)\right) \right) {}  \notag \\
&&{}\frac{e^{\theta _{kj}}P((j-1)\Delta -u<V\leq j\Delta -u)}{%
\sum_{j=k}^{m}e^{\theta _{kj}}P((j-1)\Delta -u<V\leq j\Delta -u)+e^{\theta
_{k\cdot }}\bar{F}(n\Delta-u)}du  \label{FOC1} \\
\frac{\partial }{\partial \theta _{k}}(\langle \mathbf{\Theta },\mathbf{x}%
\rangle -\bar{\psi}(\mathbf{\Theta })) &=&x_{k}-\int_{(k-1)\Delta }^{k\Delta
}\psi _{N}^{\prime }\left( \log \left( \sum_{j=k}^{n}e^{\theta
_{kj}}P((j-1)\Delta -u<V\leq j\Delta -u)+e^{\theta _{k\cdot }}\bar{F}%
(n\Delta-u)\right) \right) {}  \notag \\
&&{}\frac{e^{\theta _{k}}\bar{F}(n\Delta-u)}{\sum_{j=k}^{m}e^{\theta
_{kj}}P((j-1)\Delta -u<V\leq j\Delta -u)+e^{\theta _{k\cdot }}\bar{F}(n\Delta-u)}du
\label{FOC2}
\end{eqnarray}%
Define $\mathbf{x}^{\ast }=(\alpha _{kj},\beta _{k})_{k=1,\ldots ,n,\
j=k,\ldots ,n}$. For $\mathbf{x}=\mathbf{x}^{\ast }$, it is straightforward
to verify that $\Theta ^{\ast }=(\theta _{kj}^{\ast },\theta _{k\cdot
}^{\ast })$ where $\theta _{kj}^{\ast }=0,\theta _{k\cdot }^{\ast }=\theta
_{n\Delta}$ for $k=1,\ldots ,n$, $j=k,\ldots ,n$ satisfies \eqref{FOC1} and %
\eqref{FOC2}. Since $\langle \Theta ,\mathbf{x}\rangle -\bar{\psi}(\Theta )$
is concave in $\mathbf{\Theta }$, we have
\begin{eqnarray*}
I(\mathbf{x}^{\ast }) &=&\langle \Theta ^{\ast },\mathbf{x}^{\ast }\rangle -%
\bar{\psi}(\Theta ^{\ast }) \\
&=&\theta _{n\Delta}\sum_{k=1}^{n}\beta _{k}-\sum_{k=1}^{n}\int_{(k-1)\Delta
}^{k\Delta }\psi _{N}(\log (F(n\Delta-u)-F((k-1)\Delta -u)+e^{\theta _{n\Delta}}\bar{F}%
(n\Delta-u)))du \\
&=&\theta _{n\Delta}\psi _{n\Delta}^{\prime }(\theta _{n\Delta})-\psi _{n\Delta}\left( \theta
_{n\Delta}\right) \\
&=&I^{\ast }
\end{eqnarray*}%
Now since $\langle \Theta ,\mathbf{x}\rangle -\bar{\psi}(\Theta )$ is
continuously differentiable in $\Theta $ and $\mathbf{x}$, by Implicit
Function Theorem, $I(\mathbf{x})$ is continuous in $\mathbf{x}$. This implies that
$$I_{\mathbf{\Gamma}}\leq I(\mathbf{x}^*+\mathbf{\Gamma})\to I(\mathbf{x}^*)=I^*$$
as $\mathbf{\Gamma}\to0$. Together with \eqref{confined prob}
and \eqref{lower Gartner-Ellis} gives the conclusion.

\end{proof}

%\bigskip

Theorems \ref{upper bound copy(1)} and \ref{lower bound copy(1)} together
imply both the asymptotic optimality of Algorithm 2 and the large deviations
of the loss probability:

\begin{proof}[Proof of Theorem \ref{main thm}]
Note that by Jensen's inequality
\begin{equation*}
P_{r}(\tau _{s}<\tau _{A})^{2}\leq (E_{r}N_{A})^{2}\leq \tilde{E}%
_{r}[N_{A}^{2}L^{2}]
\end{equation*}%
Hence using Theorems \ref{upper bound copy(1)} and \ref{lower bound copy(1)}
yields
\begin{equation*}
-2I^{\ast }\leq \lim_{s\rightarrow \infty }\frac{1}{s}\log P_{r}(\tau
_{s}<\tau _{A})^{2}\leq \lim_{s\rightarrow \infty }\frac{1}{s}\log
(E_{r}N_{A})^{2}\leq \lim_{s\rightarrow \infty }\frac{1}{s}\log \tilde{E}%
_{r}[N_{A}^{2}L^{2}]\leq -2I^{\ast }
\end{equation*}%
Combining Proposition \ref{asymptotic}, we conclude that the steady-state loss
probability given by \eqref{Kac} decays exponentially with rate $I^{\ast }$
and that Algorithm 2 is asymptotically optimal.
\end{proof}

\bigskip

\section{Logarithmic Estimate of Return Time}

In this section we will lay out the argument for Proposition \ref{asymptotic}%
. The first step is to reduce the problem to a $GI/G/\infty $ calculation.
Define $x(t):=\sup \{y:Q^{\infty }(t,y)>0\}$ as the maximum residual service
times among all customers present at time $t$.

\begin{lemma}
We have $\tau _{A}\leq \tau _{A}^{\prime }$ where
\begin{equation*}
\tau _{A}^{\prime }=\inf \{t\in \{\Delta ,2\Delta ,\ldots \}:x(t-u)\leq l,\
Q^{\infty }(w)<s\text{\ for\ }w\in \lbrack t-u,t]\text{\ for some\ }u>l,\
Q^{\infty }(t,\cdot )\in J(\cdot )\}
\end{equation*}%
\label{coupling} for any $l>0$.
\end{lemma}

%\bigskip

\begin{proof} The way we couple the $GI/G/\infty $ system implies that at any
point of time the number of customers in the $GI/G/s$ system is at most that
of the coupled $GI/G/\infty $ system (in fact the served customers in the $GI/G/s$
system is a subset of those in $GI/G/\infty $). Suppose at time $t-u$ we
have $Q^{\infty }(t-u)<s$ and $x(t-u)<l$. Then $Q^{\infty }(w)<s$ for $w\in
\lbrack t-u,t]$ means that all the arrivals in this interval are not lost
i.e. they all get served in both the $GI/G/\infty$ and the $GI/G/s$ system. Since $x(t-u)\leq l$, all
the customers present at time $t$ come from arrivals after time $t-u$. This
implies that $Q(t,\cdot )\equiv Q^{\infty }(t,\cdot )$. Hence the result of
the lemma.
\end{proof}
%\bigskip

The next step is to find a mechanism to identify the instant $t-u$ and set
an appropriate value for $l$ so that $\tau_A^{\prime }$ is small. We use a
geometric trial argument. Divide the time frame into blocks separated at $%
T_{0}=0,T_{1},T_{2},\ldots $ in such a way that (1) a \textquotedblleft
success" in the block would mean $\tau _{A}^{\prime }$ is reached before the
end of the block (2) $\{W_{u},T_{i}<u\leq T_{i+1}\},i=0,1,\ldots $ are
roughly independent. We then estimate the probability of \textquotedblleft
success" in a block and also the length of a block to obtain a bound for $%
\tau _{A}^{\prime }$.

At this point let us also introduce a fixed constant $t_{0}$ and state the
following result:

\begin{lemma}
For any fixed $t_{0}>0$.
\begin{equation}
P\left( \bar{Q}^{\infty }(t,y)\in \left( \lambda s\int_{y}^{t+y}\bar{F}%
(u)du\pm \sqrt{s}C_{1}\nu(y)\right) \text{\ for all\ }t\in \lbrack
0,t_{0}],\ y\in \lbrack 0,\infty )\Bigg|B(0)\right) \geq C_{2}>0 \label{CLT2}
\end{equation}
and%
\begin{equation}
P\left( \bar{Q}^{\infty }(t,y)\notin \left( \lambda s\int_{y}^{t+y}\bar{F}%
(u)du\pm \sqrt{s}C_{1}\nu(y)\right) \text{\ for some\ }t\in \lbrack
0,t_{0}],\ y\in \lbrack 0,\infty )\Bigg|B(0)\right) \geq C_{3}>0 \label{CLT3}
\end{equation}
for large enough $C_{1}>0$ and some constants $C_{2}$ and $C_{3}$, all
independent of $s$, uniformly for all initial age $B(0)$. $\nu(y)$ is
defined in \eqref{nu}. \label{CLT1}
\end{lemma}

%\bigskip

To prove this lemma, the main idea is to consider the diffusion limit of $%
Q^{\infty }(t,y)$ as a two-dimensional Gaussian field and then invoke
Borell-TIS\ inequality (Adler\ (1990)). By Pang and Whitt (2009) we know
\begin{equation*}
\frac{Q^{\infty }(t,y)-\lambda s\int_{y}^{t+y}\bar{F}(u)du}{\sqrt{s}}%
\Rightarrow R(t,y)
\end{equation*}%
in the space $D_{D[0,\infty )}[0,\infty )$, where
\begin{equation}
R(t,y)=R_1(t,y)+R_2(t,y)  \label{R}
\end{equation}
is a two-dimensional Gaussian field given by
\begin{equation}
R_1(t,y)=\lambda \int_{0}^{t}\int_{0}^{\infty }I(u+x>t+y)dK(u,x)  \label{R_1}
\end{equation}
and
\begin{equation}
R_2(t,y)=\lambda c_{a}^{2}\int_{0}^{t}\bar{F}(t+y-u)dW(u)  \label{R_2}
\end{equation}
where $W(\cdot )$ is a standard Brownian motion, and $K(u,x)=W(\lambda
u,F(x))-F(x)W(\lambda u,1)$ in which $W(\cdot ,\cdot )$ is a standard
Brownian sheet on $[0,\infty )\times \lbrack 0,1]$. $W(\cdot )$ and $K(\cdot
,\cdot )$ are independent processes.\ $c_{a}$ is the coefficient of
variation i.e. ratio of standard deviation to mean of the interarrival times.

The key step is then to show an estimate of this limiting Gaussian process:

\begin{lemma}
Fix $t_0>0$. For $i=1,2$, we have
\begin{equation*}
P(|R(t,y)|\leq C_*\nu(y)\text{\ for all\ }t\in[0,t_0],\ y\in[0,\infty))>0
\end{equation*}
for well-chosen constant $C_*>0$, where $R(\cdot,\cdot)$ and $\nu(\cdot)$
are defined in \eqref{R}, \eqref{R_1}, \eqref{R_2} and \eqref{nu}. \label%
{limiting process}
\end{lemma}

%\bigskip

This lemma relies on an invocation of Borell-TIS inequality on the Gaussian
process $R_i(t,y)$ for $i=1,2$. The verification of the conditions for such
invocation is tedious but routine, and hence will be deferred to the appendix. Here we
provide a brief outline of the arguments: For $i=1,2$,

\begin{itemize}
\item[Step 1:] Define a $d$-metric (in fact a pseudo-metric)
\begin{equation*}
d_i((t,y),(t^{\prime },y^{\prime }))=E(\tilde{R}_i(t,y)-\tilde{R}_i(t,y))^2
\end{equation*}
where $\tilde{R}_i(t,y)=R_i(t,y)/\nu(y)$. Show that the domain $[0,t_0]\times%
[0,\infty]$ can be compactified under this (pseudo) metric.

\item[Step 2:] Use an entropy argument (see for example Adler (1990)) to
show that $E\sup_S\tilde{R}_i(t,y)<\infty$. In particular, $\tilde{R}_i(t,y)$
is a.s. bounded over $S$.

\item[Step 3:] Invoke Borell-TIS inequality i.e. for $x\geq E\sup_S\tilde{R}%
_i(t,y)$,
\begin{equation*}
P\left(\sup_S\tilde{R}_i(t,y)\geq x\right)\leq\exp\left\{-\frac{1}{%
2\sigma_i^2}\left(x-E\sup_S\tilde{R}_i(t,y)\right)^2\right\}
\end{equation*}
where
\begin{equation*}
\sigma_i^2=\sup_SE\tilde{R}_i(t,y)^2
\end{equation*}
\end{itemize}

From these steps, it is straightforward to conclude Lemma \ref{limiting
process}. The rest of the proof of Lemma \ref{CLT1} is to show the
uniformity over $U_0$ in the weak limit of $\bar{Q}^\infty$ to $R$. This is
done by restricting to the set $U_0\leq x$ for $x=O(1/s)$ and using the
light tail property of $U_0$. Again, the derivation is tedious but
straightforward; the details are provided in the appendix.

We need one more lemma:

\begin{lemma}
Let $V_k$ be r.v. with distribution function $F(\cdot)$ satisfying the
light-tail assumption in \eqref{light-tail assumption}. For any $p>0$, we
have
\begin{equation*}
E\left(\max_{k=1,\ldots,n}V_k\right)^p=O(l_p(n)^p)=o(n^\epsilon)
\end{equation*}
where
\begin{equation}
l_p(n)=\inf\{y:np\int_y^\infty u^{p-1}\bar{F}(u)du<\eta\}  \label{l_p}
\end{equation}
for a constant $\eta>0$ and $\epsilon$ is any positive number. \label{max}
\end{lemma}

%\bigskip

\begin{proof}
Let $\bar{F}_n(x)=P(\max_{k=1,\ldots,n}V_k>x)$. Note that
$$E\left(\max_{k=1,\ldots,n}V_k\right)^p=p\int_0^\infty u^{p-1}\bar{F}_n(u)du\leq y^p+np\int_y^\infty u^{p-1}\bar{F}(u)du$$
for any $y\geq0$. Pick $y=l_p(n)$. Then
$$E\left(\max_{k=1,\ldots,n}V_k\right)^p=O(l_p(n)^p)$$
Using \eqref{light tail} we have $O(l_p(n)^p)=O(n^\epsilon)$ for any $\epsilon>0$.
\end{proof}
%\bigskip

We are now ready to prove Proposition \ref{asymptotic}, which we need the
following construction. Pick $\gamma=1/t_0$ where $\gamma$ is introduced in %
\eqref{nu} and $\xi(y)$ is defined in \eqref{xi}. Recall $C_1$ as in Lemma %
\ref{CLT1}. Define $T_{i},i=0,1,2,\ldots $ as follows: Given $T_{i-1}$,
define
\begin{eqnarray*}
v(s)&=&\inf \left\{y: \sqrt{s}C_1\xi(y)<\frac{1}{2}\right\} \\
z &=&\inf \left\{ kt_{0}:k=1,2,\ldots :kt_{0}\geq v(s)+\Delta \right\} \\
x_{i} &=&x(T_{i-1}) \\
w_{i} &=&\inf \{kt_{0},k=1,2,\ldots :kt_{0}\geq x_{i}\} \\
d_{i} &=&A_{N_{s}(T_{i-1}+S_{i})+1}-(T_{i-1}+S_{i})\text{ i.e. }d_{i}\text{
is the time of first arrival after }T_{i-1}+S_{i} \\
T_{i} &=&T_{i-1}+w_{i}+d_{i}+z
\end{eqnarray*}%
Note that $w_{i}$ and $z $ are multiples of $t_{0}$. For convenience define,
for $u<t$, $\bar{Q}_{u}^{\infty }(t,y):=\bar{Q}^{\infty }(u+t,y)-\bar{Q}%
^{\infty }(u,t+y)$ as the number of arrivals after time $u$ that have
residual service time larger than $y$ at time $u+t$. We define a
\textquotedblleft success" in block $i$ to be the event $\zeta _{i}$ that
all of the following occurs: 1) $\bar{Q}_{T_{i-1}+(k-1)t_{0}}^{\infty
}(t,y)\in \left( \lambda s\int_{y}^{t+y}\bar{F}(u)du\pm \sqrt{s}%
C_{1}\nu(y)\right)$ for all $t\in \lbrack 0,t_{0}]$, for every $k=1,2,\ldots
,w_{i}/t_{0}$. 2) $d_{i}\leq c /s$ for a small constant $c >0$. 3) $%
Q_{T_{i-1}+w_{i}+d_{i}+(k-1)t_{0}}^{\infty }(t,y)\in \left( \lambda
s\int_{y}^{t+y}\bar{F}(u)du\pm \sqrt{s}C_{1}\nu(y)\right)$ for all $t\in
\lbrack 0,t_{0}]$, for every $k=1,2,\ldots ,z /t_{0}$.

Roughly speaking, $\zeta _{i}$ occurs when the $GI/G/\infty $ system behaves
\textquotedblleft normally\textquotedblright\ for a long enough period so
that $Q^{\infty }(t)$ keeps within capacity for that period and the
steady-state confidence band $J(\cdot )$ is reached at the end (see the
discussion preceding Proposition \ref{asymptotic}). More precisely, starting
from $T_{i-1}$ and given $x(T_{i-1})$, $T_{i-1}+w_{i}$ is the time when all
customers in the previous block have left. Adjusting for the age at time $%
T_{i-1}+w_{i}$, starting from $T_{i-1}+w_{i}+d_{i}$, $z $ is a long enough
time so that the system would fall into $J(\cdot )$ if it behaves normally
in each steps of size $t_{0}$ throughout the period. It can be seen by
summing up the interval boundaries that the occurrence of $\zeta _{i}$
ensures $\tau _{A}^{\prime }$ is reached during the last $\Delta $ units of
time before $T_{i}$.

\begin{proof}[Proof of Proposition \ref{asymptotic}]
We first check that the occurrence of event $\zeta_i$ implies that $\tau_A'$ is reached during the last $\Delta $ units of time before $T_{i}$. As discussed above, since $w_i\geq x_i$, all the customers at time $T_{i-1}+w_i$ will be those arrive after time $T_{i-1}$. Hence the occurrence of $\zeta_i$ implies that
\begin{eqnarray}
&&Q^\infty(T_{i-1}+w_i,y) \notag\\
&\in&\left(\lambda s\sum_{k=1}^{w_i/t_0}\int_{(k-1)t_0+y}^{kt_0+y}\bar{F}(u)du\pm\sqrt{s}C_1\sum_{k=1}^{w_i/t_0}\nu((k-1)t_0+y)\right) \notag\\
&\subset&\left(\lambda s\int_y^{w_i+y}\bar{F}(u)du\pm\sqrt{s}C_1\left[\nu(y)+\frac{1}{t_0}\int_y^\infty\nu(u)du\right]\right) \notag\\
&\subset&\left(\lambda s\int_y^{w_i+y}\bar{F}(u)du\pm\sqrt{s}C_1\xi(y)\right) \label{interval}
\end{eqnarray}
and
$$Q^\infty(T_{i-1}+w_i+d_i,y)\in\left(\lambda s\int_{d_i+y}^{w_i+d_i+y}\bar{F}(u)du\pm\sqrt{s}C_1\xi(d_i+y)\right)$$

For each $t\in((k-1)t_0,kt_0]$, denote $[t]=t-(k-1)t_0$, for $k=1,\ldots,z/t_0$. Then
\begin{eqnarray}
&&Q^\infty(T_{i-1}+w_i+d_i+t,y) \notag\\
&\in&\Bigg(\lambda s\int_y^{t+y}\bar{F}(u)du+\lambda s\int_{d_i+y+t}^{w_i+d_i+y+t}\bar{F}(u)du{} \notag\\
&&{}\pm\sqrt{s} C_1\left[\sum_{j=1}^{w_i/t_0}\nu((j-1)t_0+d_i+(k-1)t_0+[t]+y)+\nu(y)+\sum_{j=2}^k\nu((j-2)t_0+[t]+y)I(k>1)\right]\Bigg) \notag\\
&\subset&\left(\lambda s\int_y^{t+y}\bar{F}(u)du+\lambda s\int_{d_i+y+t}^{w_i+d_i+y+t}\bar{F}(u)du\pm\sqrt{s}C_1\left[\sum_{j=1}^{w_i/t_0+k-1}\nu((j-1)t_0+[t]+y)+\nu(y)\right]\right) \notag\\
&\subset&\left(\lambda s\int_y^{t+y}\bar{F}(u)du+\lambda s\int_{d_i+y+t}^{w_i+d_i+y+t}\bar{F}(u)du\pm\sqrt{s}C_1\left[2\nu(y)+\frac{1}{t_0}\int_y^\infty\nu(u)du\right]\right) \notag\\
&\subset&\left(\lambda s\int_y^{t+y}\bar{F}(u)du+\lambda s\int_{d_i+y+t}^{w_i+d_i+y+t}\bar{F}(u)du\pm\sqrt{s}C'\xi(y)\right) \label{interval1}
\end{eqnarray}
where $C'=2C_1$ (which depends on $\gamma$).

It is now obvious that $\zeta_i$ implies $Q^\infty(t)<s$ for $[T_{i-1}+w_i,T_i]$. By the definition of $v(s)$, \eqref{interval} and the fact that $\lambda s\int_y^\infty\bar{F}(u)du$ is smaller and decays faster than $\sqrt{s}C_1\xi(y)$ for $y\geq v(s)$ when $s$ is large, we get $x(T_{i-1}+w_i)\leq v(s)\leq z$. Let $\tilde{T}_i=\sup\{k\Delta:k\Delta\leq T_i\}$ be the largest time before $T_i$ such that $A$ can possibly be hit i.e. in the $\Delta$-skeleton. It remains to show that $Q^\infty(\tilde{T}_i,y)\in J(y)$ in order to conclude that $\zeta_i$ implies a hit on $\tau_A'$.

From \eqref{interval1}, for $t\in[T_{i-1}+w_i+d_i,T_i]$,
$$Q^\infty(t,y)\in\left(\lambda s\int_y^{t-T_{i-1}+y}\bar{F}(u)du-\lambda s\int_{t-T_{i-1}-w_i-d_i+y}^{t-T_{i-1}-w_i+y}\bar{F}(u)du\pm\sqrt{s}C'\xi(y)\right)$$
In particular,
\begin{align}
Q^\infty(\tilde{T}_i,y)&\in\left(\lambda s\int_y^{\tilde{T}_i-T_{i-1}+y}\bar{F}(u)du-\lambda s\int_{\tilde{T}_i-T_{i-1}-w_i-d_i+y}^{\tilde{T}_i-T_{i-1}-w_i+y}\bar{F}(u)du\pm\sqrt{s}C'\xi(y)\right) \notag\\
&=\left(\lambda s\int_y^\infty\bar{F}(u)du-\lambda s\int_{\tilde{T}_i-T_{i-1}+y}^\infty\bar{F}(u)du-\lambda s\int_{\tilde{T}_i-T_{i-1}-w_i-d_i+y}^{\tilde{T}_i-T_{i-1}-w_i+y}\bar{F}(u)du\pm\sqrt{s}C'\xi(y)\right) \label{interval2}
\end{align}

Now note that
$$\lambda s\int_{\tilde{T}_i-T_{i-1}+y}^\infty\bar{F}(u)du+\lambda s\int_{\tilde{T}_i-T_{i-1}-w_i-d_i+y}^{\tilde{T}_i-T_{i-1}-w_i+y}\bar{F}(u)du\leq2\lambda s\int_{v(s)+y}^\infty\bar{F}(u)du$$
and we claim that it is further bounded from above by $\sqrt{s}C\xi(y)$ for arbitrary constant $C$ when $s$ is large enough, uniformly over $y\in[0,\infty)$. In fact, we have $v(s)\geq\inf\{y:s\int_y^\infty\bar{F}(u)\leq\alpha\}$ for any $\alpha>0$ when $s$ is large enough. Now when $\sqrt{s}C\xi(y)<\alpha/(2\lambda)$, $s\int_{v(s)+y}^\infty\bar{F}(u)du\leq s\int_y^\infty\bar{F}(u)du$ which is smaller and decays faster than $\sqrt{s}C\xi(y)$ when $s$ is large. When $\sqrt{s}C\xi(y)\geq\alpha/(2\lambda)$, we have $s\int_{v(s)+y}^\infty\bar{F}(u)du\leq s\int_{v(s)}^\infty\bar{F}(u)du\leq\alpha/(2\lambda)$. Picking $C^*=C'+C$ where $C^*$ is defined in \eqref{J}, we conclude that $\zeta_i$ implies $\tau_A'$ is reached at $\tilde{T}_i$.

Now let $N=\inf\{i:\zeta_i\text{\ occurs\ }\}$. Consider (suppressing the initial conditions), for any $p>0$,
\begin{eqnarray}
&&E(\tau_A')^p \nonumber\\
&=&E\left[\sum_{i=1}^N(w_i+d_i+z)\right]^p \nonumber\\
&=&E\left[\sum_{i=1}^\infty(w_i+d_i+z)I(N\geq i)\right]^p \nonumber\\
&\leq&\left(\sum_{i=1}^\infty(E[(w_i+d_i+z)^p;N\geq i])^{1/p}\right)^p \nonumber\\
&\leq&\left(\sum_{i=1}^\infty(E(w_i+d_i+z)^{pq})^{1/(pq)}(P(N\geq i))^{1/(pr)}\right)^p \label{intermediate1}
\end{eqnarray}
where $q,r>0$ and $1/q+1/r=1$, by using Minkowski's inequality and Holder's inequality in the first and second inequality respectively.

For $i=2,3,\ldots$, we have
\begin{equation}
E(w_i+d_i+z)^{pq}\leq[(Ew_i^{pq})^{1/(pq)}+(Ed_i^{pq})^{1/(pq)}+z]^{pq} \label{decomposition}
\end{equation}
by Minkowski's inequality again.

We now analyze $E(w_i+d_i+z)^p$ for any $p>0$. From now on $C$ denotes constant, not necessarily the same every time it appears. First note that
\begin{equation}
(Ed_i^p)^{1/p}\leq d^{(p)}:=\sup_{b\geq0}(E[d_i^p|B(T_{i-1}+w_i)=b])^{1/p}=\frac{1}{s}\sup_{b\geq0}(E[(U^0-b)^p|B^0(0)=b])^{1/p}=O\left(\frac{1}{s}\right) \label{d p}
\end{equation}
and $z\leq v(s)+\Delta+t_0=o(s^\epsilon)$ for any $\epsilon>0$. The last equality of \eqref{d p} comes from the light-tail assumption on $U^0$. Indeed, since $U^0$ is light-tailed, we have
$$\exp\left\{-\int_0^xh_U(u)du\right\}=\bar{F}_U(x)\leq e^{-cx}$$
for some $c>0$, where $h_U(\cdot)$ and $\bar{F}(x)$ are the hazard rate function and tail distribution function of $U^0$ respectively. This implies that $h(x)\geq c$ for all $x\geq0$. Then
$$\sup_{b\geq0}P(U^0-b>x|U^0>b)=\sup_{b\geq0}\exp\left\{-\int_b^{x+b}h(u)du\right\}\leq e^{-cx}$$
and so
$$\sup_{b\geq0}E[(U^0-b)^p|B^0(0)=b]=\sup_{b\geq0}p\int_0^\infty x^{p-1}P(U^0-b>x|U^0>b)dx\leq p\int_0^\infty x^{p-1}e^{-cx}dx<\infty$$

For $i=1$, $w_1\leq l(s)+t_0=o(s^\epsilon)$ where $l(s)$ is defined in \eqref{l}. Hence $E(w_1+d_1+z)^p\leq[(Ew_1^p)^{1/p}+(Ed_1^p)^{1/p}+z]^p=o(s^\epsilon)$ for any $\epsilon>0$.

Now
\begin{align}
Ew_i^p&\leq E\left[\left(\max_{i=1,\ldots,N_s(T_{i-1})-N_s(T_{i-2})}V_i\right)^p\right] \notag\\
&=E\left[E\left[\left(\max_{i=1,\ldots,N_s(T_{i-1})-N_s(T_{i-2})}V_i\right)^p\Bigg|N_s(T_{i-1})-N_s(T_{i-2})\right]\right] \notag\\
&\leq CE[l_p(N_s(T_{i-1})-N_s(T_{i-2}))^p]\text{\ \ for some constant $C=C(p)$ and $l_p(\cdot)$ defined in \eqref{l_p}} \notag\\
&\leq CE[(N_s(T_{i-1})-N_s(T_{i-2}))^\epsilon]\text{\ \ for constant $C=C(p,\epsilon)$} \label{intermediate11}
\end{align}
for any $\epsilon>0$, by Lemma \ref{max}. Pick $\epsilon<1$. By Jensen's inequality and elementary renewal theorem, \eqref{intermediate11} is less than or equal to
\begin{eqnarray}
&&C(E[N_s(T_{i-1})-N_s(T_{i-2})])^\epsilon \notag\\
&=&C(E[N_s(T_{i-1})-N_s(T_{i-2})|T_{i-1}-T_{i-2}])^\epsilon \notag\\
&\leq&C(E[\tilde{\lambda}s(T_{i-1}-T_{i-2})])^\epsilon\text{\ \ for some $\tilde{\lambda}>\lambda$} \notag\\
&=&C\tilde{\lambda}^\epsilon s^\epsilon(E[T_{i-1}-T_{i-2}])^\epsilon \notag\\
&=&C\tilde{\lambda}^\epsilon s^\epsilon(E[w_{i-1}+d_{i-1}+z])^\epsilon \label{intermediate21}
\end{eqnarray}

Let $y_i=E[w_i+d_i+z]$. We then have
$$y_i=Cs^\epsilon y_{i-1}^\epsilon+d^{(1)}+z$$
By construction $y_i\geq t_0$, and since $v(s)=o(s^\epsilon)$ for any $\epsilon>0$ we have
$$d^{(1)}+z\leq Cs^\epsilon t_0^\epsilon\leq Cs^\epsilon y_i^\epsilon$$
for large enough $s$, uniformly over $i$. Hence
$$y_i\leq Cs^\epsilon y_{i-1}^\epsilon+d^{(1)}+z\leq Cs^\epsilon y_{i-1}^\epsilon$$
Now we can write
\begin{align}
y_i&\leq Cs^\epsilon y_{i-1}^\epsilon\leq Cs^\epsilon(Cs^\epsilon y_{i-2}^\epsilon)^\epsilon=C^{1+\epsilon}s^{\epsilon+\epsilon^2}y_{i-2}^{\epsilon^2}{} \notag\\
&{}\cdots\leq(C^{1/(1-\epsilon)}\vee1)s^{\epsilon/(1-\epsilon)}y_1^{\epsilon^{i-1}}=o(s^\rho) \label{intermediate31}
\end{align}
for any $\rho>0$ by choosing $\epsilon$, uniformly over $i$.

Therefore from \eqref{decomposition}, \eqref{intermediate21} and \eqref{intermediate31}, we get
\begin{equation}
E(w_i+d_i+z)^{pq}=o(s^\epsilon) \label{intermediate51}
\end{equation}
for any $\epsilon>0$ uniformly over $i$.

Now consider
\begin{eqnarray}
P(N\geq1)&=&P(\zeta_1^c)=1-P(\zeta_1) \notag\\
&\leq&1-P\left(d_1\leq\frac{c}{s}\right)C_2^{(w_1+z)/t_0}{} \notag\\
&&{}\text{\ where $C_2$ is defined in Lemma \ref{CLT1} and $c$ is defined in the discussion of $\zeta_i$} \notag\\
&\leq&1-be^{-a(w_1+z)} \notag\\
&=&1-be^{-o(s^\epsilon)} \label{intermediate61}
\end{eqnarray}
for some constants $a>0$ and $0<b<1$ and any $\epsilon>0$. Moreover, for $i=2,3,\ldots$,
\begin{align}
P(N\geq i)&=P(N\geq i-1)P(\zeta_{i-1}^c|N\geq i-1) \nonumber\\
&\leq P(N\geq i-1)E[1-be^{-a(w_{i-1}+z)}|N\geq i-1] \nonumber\\
&\leq P(N\geq i-1)(1-be^{-a(E[w_{i-1}|N\geq i-1]+z)}) \label{intermediate bound}
\end{align}
by Jensen's inequality and that the function $1-be^{-a(\cdot+z)}$ is concave.

Consider $E[w_i|N\geq i]$ for any $i=2,3,\ldots$. We have
\begin{equation}
E[w_i|N\geq i]=E[E[w_i|\zeta_{i-1}^c,\ w_{i-1}+d_{i-1}+z]|N\geq i] \label{intermediate41}
\end{equation}
Now by singling out failure in the first trial of $t_0$ (see the discussion on $\zeta_i$), we get
$$P(\zeta_{i-1}^c|w_{i-1}+d_{i-1}+z)\geq C_3$$
where $C_3$ is defined in Lemma \ref{CLT1}, uniformly over $w_{i-1}+d_{i-1}+z$. Hence
\begin{align*}
C_3E[w_i|\zeta_{i-1}^c,w_{i-1}+d_{i-1}+z]&\leq\int P(\zeta_{i-1}^c|w_{i-1}+d_{i-1}+z)E[w_i|\zeta_{i-1}^c,w_{i-1}+d_{i-1}+z]P(w_{i-1}+d_{i-1}+z\in dx)\\
&\leq Ew_i
\end{align*}
which gives
$$E[w_i|\zeta_{i-1}^c,\ w_{i-1}+d_{i-1}+z]\leq\frac{Ew_i}{C_3}$$
uniformly over $w_{i-1}+d_{i-1}+z$. Therefore \eqref{intermediate41} is bounded from above by $Ew_i/C_3$.

From \eqref{intermediate21} and \eqref{intermediate31} we know that $Ew_i=o(s^\epsilon)$ for any $\epsilon>0$. So \eqref{intermediate bound} is less than or equal to
\begin{equation}
P(N\geq i-1)(1-be^{-a(Ew_{i-1}/C_3+z)})=P(N\geq i-1)(1-be^{-o(s^\epsilon)}) \label{asymptotic2}
\end{equation}
for any $\epsilon>0$ uniformly over $i$.

By \eqref{intermediate1}, \eqref{intermediate61}, \eqref{intermediate51} and \eqref{asymptotic2} we get
\begin{align*}
E\tau^p&\leq o(s^\epsilon)\left(\sum_{i=1}^\infty(P(N\geq i))^{1/(pr)}\right)^p\\
&\leq o(s^\epsilon)\left(\sum_{i=1}^\infty(1-be^{-o(s^\epsilon)})^{i/(pr)}\right)^p\\
&\leq o(s^\epsilon)\frac{1}{[1-(1-be^{-o(s^\epsilon)})^{1/(pr)}]^p}\\
&\leq o(s^\epsilon)e^{o(s^\epsilon)}
\end{align*}
Hence
$$\frac{1}{s}\log E\tau^p\leq\frac{\epsilon}{s}+\frac{o(s^\epsilon)}{s}\to0$$
as $s\to\infty$. On the other hand, we pick $A$ such that $\tau_A\geq\Delta$ and so
$$\frac{1}{s}\log E\tau_A^p\geq\frac{1}{s}\log\Delta^p\to0$$
Conclusion follows for \eqref{tau limit}.

For \eqref{N limit}, note that $N_A\leq N_s(\tau_A)\leq N_s(\tau_A')$ and $EN_s(t)^p=O(st)$ since $(1/s)\log Ee^{\theta N_s(t)}\to-\psi_N(\theta)t$. Hence
$$EN_s(\tau_A')^p\leq O(s^p)E(\tau_A')^p$$
and the result follows from \eqref{tau limit}.
\end{proof}
%\bigskip

\begin{remark}
The proof of Proposition \ref{asymptotic} can be simplified when the service
time has bounded support, say on $[0,M]$. In this case the $GI/G/\infty $
system is \textquotedblleft $M+U_{0}$-independent" i.e. $W_{t}^{\infty }$,
the state of the system at time $t$ and $W_{A_{N_{s}(t)+1}+M}^{\infty }$,
the state of the system at $M$ time units after the first arrival since time
$t$ are independent. As a result we can merely set $v(s)=M$ and $x_i=M$ for
any $i$, and the same argument as above will apply.
\end{remark}

\bigskip

\section{Numerical Example}

We close this paper by a numerical example for $GI/G/s$. We set the
interarrival times in the base system to be $\text{Gamma}(1/2,1/2)$ so $%
\lambda =1$. For illustrative convenience we set the service times as $\text{%
Uniform}(0,1)$. Hence traffic intensity is $1/2$. In this case, we can
simply set $C^{\ast }=1$ and $\xi(y)=\text{sd}(R(\infty ,y))\vee C_{1}=\sqrt{%
\lambda \int_{y}^{\infty }F(u)\bar{F}(u)du+\lambda c_{a}^{2}\int_{y}^{\infty
}\bar{F}(u)^{2}du}\vee C_{1}$ with $C_{1}=1.1$ (note that $\eta=0$ and we
use a truncated $\xi(y)$; the validity of this simpler choice than the one
displayed in Section 2.1 can be verified from the arguments in Section 4
specialized to the case of bounded service time). Also we choose $\Delta =1$%
. To test the numerical efficiency of our importance sampling algorithm, we
compare it with\ crude Monte Carlo scheme using increasing values of $s$,
namely $s=10,\ 30,\ 60,\ 80,\ 100$ and $120$.

As discussed in Section 2, since we run our importance sampler everytime we
hit set $A$, the initial positions of the importance samplers are dependent.
To get an unbiased estimate of standard error we group the samples into
batches and obtain statistics based on these batch samples (see Asmussen and
Glynn (2007)). To make the estimates and statistics comparable, for each
experiment we run the computer for roughly 120 seconds CPU time and always
use 20 batches. In the tables below, we output the estimates of loss
probability, the relative errors (ratios of sample standard
deviation to sample mean) and 95\% confidence intervals for both crude Monte
Carlo scheme and importance sampler under different values of $s$.

When $s$ is small we see that crude Monte Carlo performs slightly better
than our importance sampler. However, when $s$ is over 80, importance
sampler starts to perform better. When $s$ is above 100, crude Monte Carlo
totally breaks down while our importance sampler still gives estimates that
have encouragingly small relative error.

\begin{table}[ht]
\begin{minipage}[b]{0.5\linewidth}\centering
{\scriptsize\begin{tabular}{l|lll}
\multicolumn{4}{c}{Crude Monte Carlo}\\
\multicolumn{4}{c}{}\\
$s$& Estimate & R.E. & C.I. \\ \hline
$10$ & $0.05318$ & $0.0265$ & $(0.05252,0.05384)$ \\
$30$ & $0.003174$ & $0.111$ & $(0.003009,0.003338)$ \\
$60$ & $7.0922\times 10^{-5}$ & $1.388$ & $(2.4847\times
10^{-5},1.1700\times 10^{-4})$ \\
$80$ & $6.9444\times 10^{-7}$ & $4.472$ & $(-7.5904\times
10^{-7},2.1479\times 10^{-6})$ \\
$100$ & $0$ & $N/A$ & $N/A$ \\
$120$ & $0$ & $N/A$ & $N/A$%
\end{tabular}}
\end{minipage}
\hspace{0.5cm}
\begin{minipage}[b]{0.5\linewidth}
\centering
{\scriptsize \begin{tabular}{|lll}
\multicolumn{3}{c}{Importance Sampler}\\
\multicolumn{3}{c}{}\\
Estimate & R.E. & C.I. \\ \hline
$0.05412$ & $0.130$ & $(0.05084,0.05740)$ \\
$0.003204$ & $0.570$ & $(0.002349,0.004060)$ \\
$6.2585\times 10^{-5}$ & $2.258$ & $(-3.5529\times
10^{-6},1.2872\times 10^{-4})$ \\
$4.5001\times 10^{-8}$ & $1.879$ & $(5.4365\times
10^{-9},8.4565\times 10^{-8})$ \\
$8.1178\times 10^{-10}$ & $2.296$ & $(-6.0511\times
10^{-11},1.6841\times 10^{-9})$ \\
$1.3025\times 10^{-10}$ & $4.472$ & $(-1.4237\times
10^{-10},4.0286\times 10^{-10})$%
\end{tabular}}
\end{minipage}
\end{table}

We can also analyze the graphical depiction of the sample paths. Figures 6
and 7 are two sample paths run by Algorithm 2, initialized at the mean of $%
Q(t,y)$ i.e. $\lambda s\int_y^\infty\bar{F}(u)du$. Figure 6 is a contour
plot of $Q(t,y)$, whereas Figure 7 is a three-dimensional plot of another $%
Q(t,y)$. As we can see, the number of customers (the color at the $t$-axis)
increases from time 0 to around 0.95 when it hits overflow in the contour
plot. Similar trajectory appears in the three-dimensional plot. These plots
are potentially useful for operations manager to judge the possibility of
overflow over a finite horizon given the current state.

\begin{figure}[ht]
\centering
\subfigure{
\includegraphics[scale=.5]{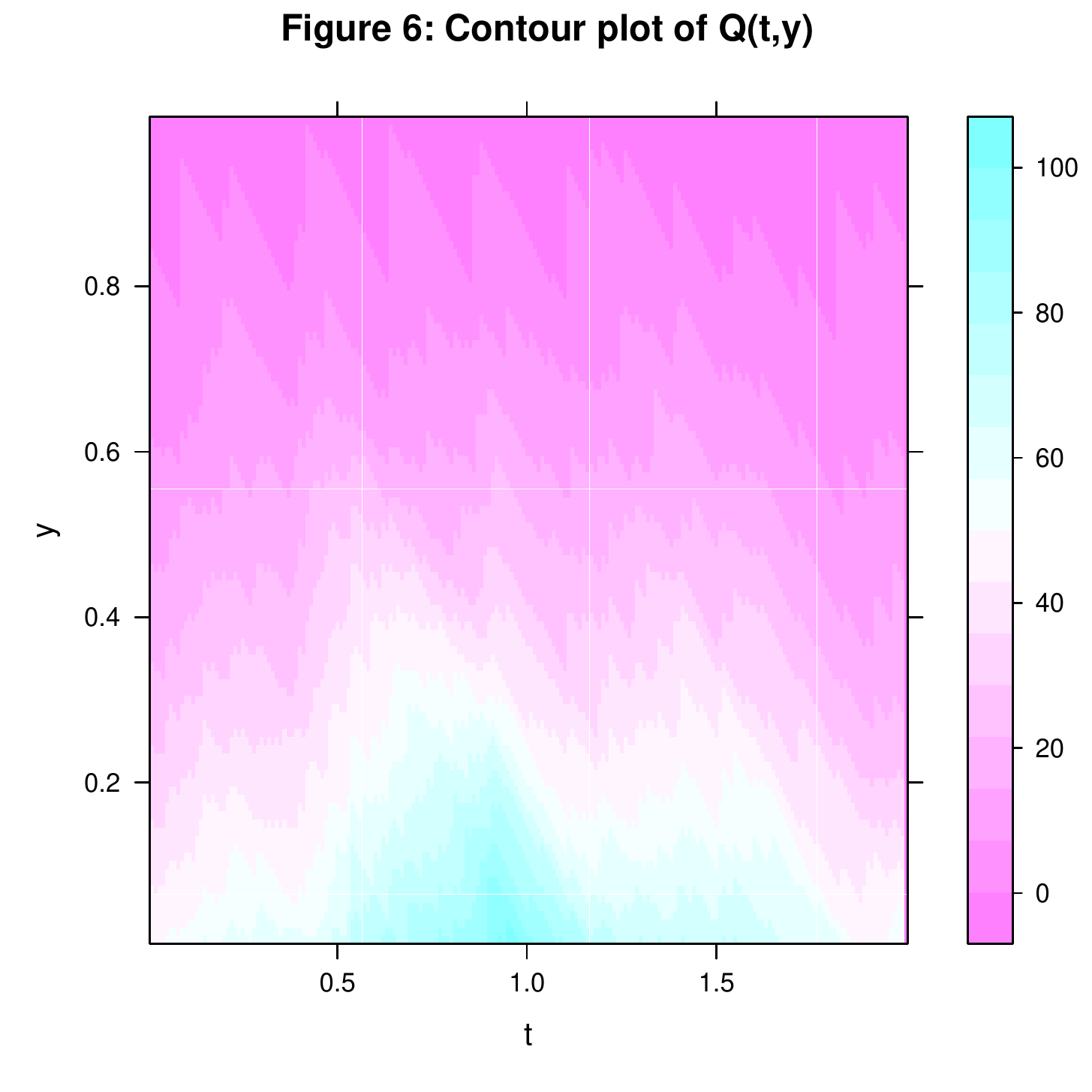}
} \subfigure{
\includegraphics[scale=.5]{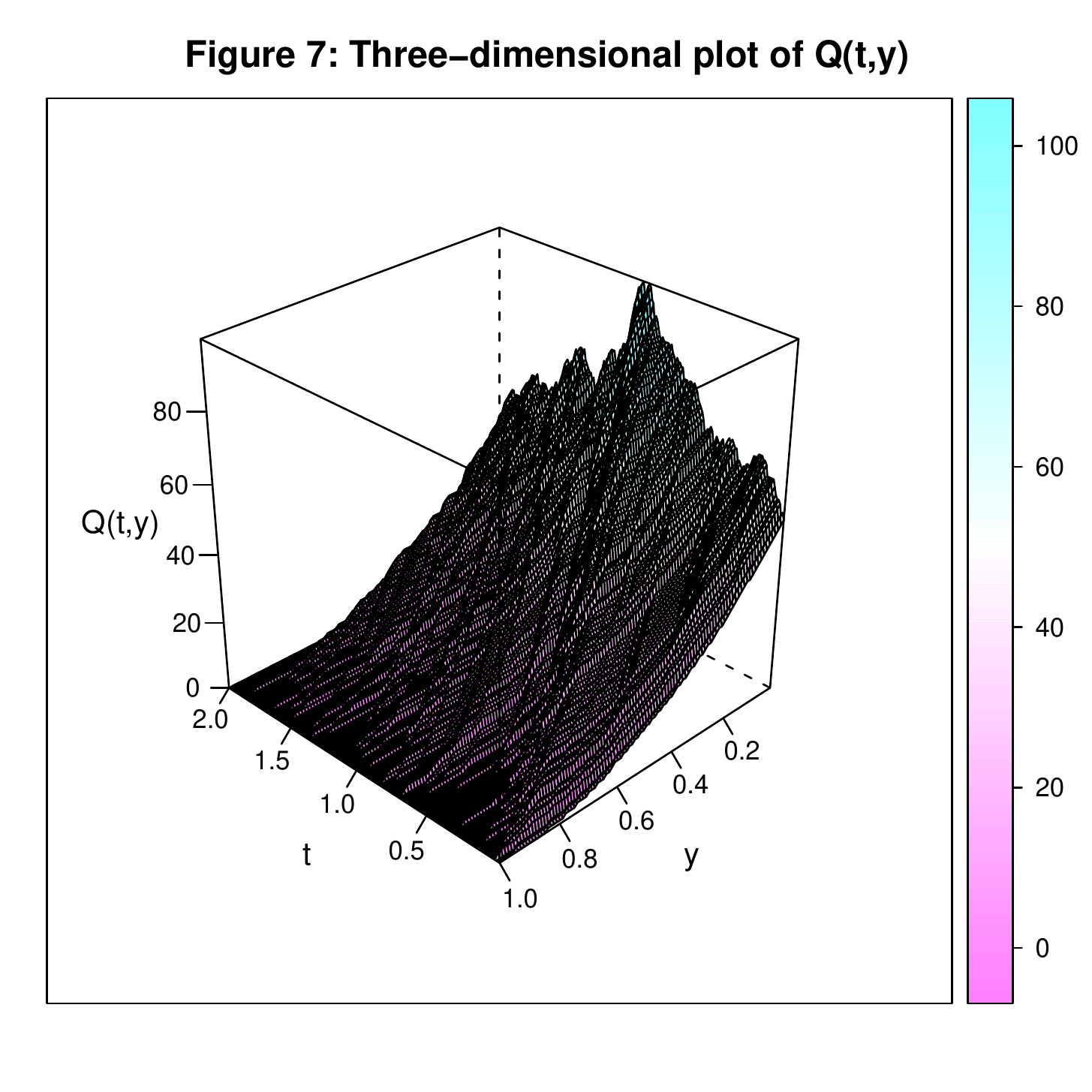}
}
\end{figure}

%\bigskip

\appendix

\section{Technical Proofs}

\subsection{Proof of Lemma \protect\ref{psi properties}}

The domain of $\psi_t(\cdot)$ is easily seen to inherit from $\psi_N(\cdot)$%
. Write
\begin{equation*}
\psi_t(\theta)=\int_0^t\psi_N(\log(e^\theta\bar{F}(u)+F(u)))du
\end{equation*}
Note that
\begin{equation*}
\frac{\partial}{\partial\theta}\psi_N(\log(e^\theta\bar{F}(u)+F(u)))=\psi
_{N}^{\prime }(\log (e^{\theta }\bar{F}(u)+F(u)))\frac{e^{\theta }\bar{F}(u)%
}{e^{\theta }\bar{F}(u)+F(u)}
\end{equation*}
is continuous in $u$ and $\theta$. Hence
\begin{equation*}
\psi _{t}^{\prime }(\theta )=\int_{0}^{t}\psi _{N}^{\prime }(\log (e^{\theta
}\bar{F}(u)+F(u)))\frac{e^{\theta }\bar{F}(u)}{e^{\theta }\bar{F}(u)+F(u)}du
\end{equation*}%
(see Rudin (1976), p. 236 Theorem 9.42). Moreover, $\psi _{N}^{\prime }(\log
(e^{\theta }\bar{F}(u)+F(u)))e^{\theta }\bar{F}(u)/(e^{\theta }\bar{F}%
(u)+F(u))$ is uniformly continuous in $u$ and a neighborhood of $\theta$,
for any $\theta\in\mathbb{R}$. Hence $\psi _{t}^{\prime }(\theta )$ is
continuous in $\theta$. Also the strict monotonicity of $\psi_N^{\prime
}(\cdot)$ implies that $\psi_t^{\prime }(\theta)$ too is strictly increasing
for any $\theta>0$.

Following the same argument, we have
\begin{equation*}
\psi_t^{\prime \prime }(\theta)=\int_{0}^{t}\Bigg[\psi _{N}^{\prime
\prime}(\log(e^ \theta\bar{F}(u)+F(u)))\left( \frac{e^{\theta }\bar{F}(u)}{%
e^{\theta }\bar{F}(u)+F(u)}\right) ^{2}+\psi _{N}^{\prime }(\log(e^\theta
\bar{F}(u)+F(u)))\frac{F(u)\bar{F}(u)e^{\theta }}{(e^{\theta }\bar{F}%
(u)+F(u))^{2}}\Bigg]du
\end{equation*}
which is continuous in $\theta$.

Finally, note that as $\theta\nearrow\infty$, $\psi _{N}^{\prime }(\log
(e^{\theta }\bar{F}(u)+F(u)))e^{\theta }\bar{F}(u)/(e^{\theta }\bar{F}%
(u)+F(u))\nearrow\infty$ for any $u\in\text{supp}\ \bar{F}$ since $%
\psi_N(\cdot)$ is steep. By monotone convergence theorem we conclude that $%
\psi_t(\cdot)$ is steep.

%\bigskip

\subsection{Proof of Lemma \protect\ref{theta}}

1) Denote $\theta(t)=\theta_t$ for convenience. Since $\psi_t^{\prime
}(\cdot)$ is continuously differentiable by Lemma \ref{psi properties}, by
implicit function theorem, we can differentiate $\psi_t^{\prime
}(\theta(t))=a_t$ with respect to $t$ on both sides to get
\begin{eqnarray*}
\psi _{N}^{\prime}(\log(e^{\theta (t)}\bar{F}(t)+F(t)))\frac{e^{\theta (t)}%
\bar{F}(t)}{e^{\theta (t)}\bar{F}(t)+F(t)}+\int_{0}^{t}\Bigg[\psi
_{N}^{\prime \prime} (\log(e^{\theta (t)}\bar{F}(u)+F(u)))\left( \frac{%
e^{\theta (t)}\bar{F}(u)}{e^{\theta (t)}\bar{F}(u)+F(u)}\right) ^{2}{} && \\
{}+\psi _{N}^{\prime}(\log(e^{\theta (t)}\bar{F}(u)+F(u)))\frac{F(u)\bar{F}%
(u)e^{\theta (t)}}{(e^{\theta (t)}\bar{F}(u)+F(u))^{2}}\Bigg]du\theta
^{\prime }(t) &=&\lambda \bar{F}(t)
\end{eqnarray*}%
which gives
\begin{eqnarray*}
&&\theta ^{\prime }(t) \\
&=&\frac{\lambda \bar{F}(t)-\psi _{N}^{\prime }(\log (e^{\theta(t) }\bar{F}%
(t)+F(t)))e^{\theta(t) }\bar{F}(t)/(e^{\theta(t) }\bar{F}(t)+F(t))}{%
\int_{0}^{t}\Bigg[\psi _{N}^{\prime \prime}(\log(e^{\theta (t)}\bar{F}%
(u)+F(u)))\left( \frac{e^{\theta (t)}\bar{F}(u)}{e^{\theta (t)}\bar{F}%
(u)+F(u)}\right) ^{2}+\psi _{N}^{\prime} (\log(e^{\theta (t)}\bar{F}%
(u)+F(u)))\frac{F(u)\bar{F}(u)e^{\theta (t)}}{(e^{\theta (t)}\bar{F}%
(u)+F(u))^{2}}\Bigg]du} \\
&\leq &0
\end{eqnarray*}%
The inequality is due to the fact that
\begin{equation}
g_{t}(\theta ):=\psi _{N}^{\prime }(\log (e^{\theta }\bar{F}(t)+F(t)))\frac{%
e^{\theta }\bar{F}(t)}{e^{\theta }\bar{F}(t)+F(t)}  \label{g}
\end{equation}%
is non-decreasing in $\theta $ and $g_{t}(0)=\lambda \bar{F}(t)$, and that $%
\psi _{N}(\cdot )$ is non-decreasing and convex. Hence $\theta (t)$ is
non-increasing.

\bigskip

\noindent 2) Since $a_t\geq1-\lambda EV$, $\theta_t\geq\bar{\theta}_t$ where
$\bar{\theta}_t$ satisfies $\psi_t^{\prime }(\bar{\theta}_t)=1-\lambda EV$,
well-defined when $t$ is small enough. Moreover, it is easy to check that $%
\psi_t^{\prime }(\theta)\leq\psi_N^{\prime }(\theta)t$ for any $\theta,t>0$
(either by the formula of $\psi_t^{\prime }$ and $\psi_N^{\prime }$ or by
definition in terms of Gartner-Ellis limit). This implies that $%
(\psi_t^{\prime -1}(y)\geq(\psi_N^{\prime -1}(y/t)$ for any $y$ in the
domain. Putting $y=1-\lambda EV$ gives $\bar{\theta}_t\geq(\psi_N^{\prime
-1}((1-\lambda EV)/t)$. By steepness of $\psi_N$ we have $(\psi_N^{\prime
-1}((1-\lambda EV)/t)\nearrow\infty$ as $t\searrow0$. So $%
\theta_t\nearrow\infty$ as $t\searrow0$.

\bigskip

\noindent 3) Consider $\psi_t^{\prime }(\theta_t)=a_t$, or $%
\theta_t=(\psi_t^{\prime -1}(a_t)$. Now from \eqref{psi infinity} we have
\begin{equation*}
\psi_\infty^{\prime }(\theta)=\int_0^\infty\psi_N^{\prime \theta}\bar{F}%
(u)+F(u)))\frac{e^\theta\bar{F}(u)}{e^\theta\bar{F}(u)+F(u)}du
\end{equation*}
and that $\psi_\infty^{\prime }(\theta)$ is increasing in $\theta$, by the
same argument as in the proof of 1). Moreover, by monotone convergence we
have $\psi_t^{\prime }\nearrow\psi_\infty^{\prime }$ as $t\nearrow\infty$.

By Billingsley (1979), p. 287, or Resnick (2008), p. 5, Proposition 0.1, we
have $(\psi_t^{\prime -1}\to(\psi_\infty^{\prime -1}$ as $t\nearrow\infty$.
Moreover, since $(\psi_t^{\prime -1}$ is increasing over the compact
interval $[\lambda EV,1]$, the convergence is uniform. By Resnick (2008), p.
2, this implies continuous convergence, and hence $(\psi_t^{\prime
-1}(a_t)\to(\psi_\infty^{\prime -1}(1)$, or $\theta_t\to\theta_\infty$.

%\bigskip

\subsection{Proof of Lemma \protect\ref{rate}}

1) As in the proof of Lemma \ref{theta} Part 1, denote $\theta(t)=\theta_t$.
Consider
\begin{align*}
\frac{d}{dt}I_{t}& =\theta (t)\lambda \bar{F}(t)+\theta ^{\prime
}(t)a_{t}-\psi _{t}^{\prime }(\theta (t))\theta ^{\prime }(t)-\psi _{N}(\log
(e^{\theta (t)}\bar{F}(t)+F(t))) \\
& =\theta (t)\lambda \bar{F}(t)-\psi _{N}(\log (e^{\theta (t)}\bar{F}%
(t)+F(t)))
\end{align*}%
since $\psi _{t}^{\prime }(\theta (t))=a_{t}$. Note that $h_{t}(\theta
):=\psi _{N}(\log (e^{\theta }\bar{F}(t)+F(t)))$ is convex in $\theta $ for
any $t\geq 0$ and so
\begin{equation*}
h_{t}(\theta (t))\geq h_{t}(0)+h_{t}^{\prime }(0)\theta (t)
\end{equation*}%
which gives
\begin{equation*}
\psi _{N}(\log (e^{\theta (t)}\bar{F}(t)+F(t)))\geq \lambda \bar{F}(t)\theta
(t)
\end{equation*}%
Hence $(d/dt)I_{t}\leq 0$ and so $I_t$ is non-increasing.

\bigskip

\noindent 2) Write $I_t=a_t\theta_t-\psi_t(\theta_t)$. By Lemma \ref{theta}
Part 3, $\theta_t\searrow\theta_\infty$ on $[\theta_\infty,\theta_T]$ for $%
t\geq T$ for some $T>0$. Since $\psi_t(\theta)$ is increasing in $\theta$,
by continuous convergence (see Resnick (2008), p. 2) we have $%
\psi_t(\theta_t)\to\psi_\infty(\theta_\infty)$. Hence $I_t\to I^*$ defined
in \eqref{I optimal}.

\bigskip

\noindent 3) Note that in case $V$ is supported on $[0,M]$, it is easy to
check that $I_t=I_M$ is the same for any $t\geq M$. Hence the conclusion.

%\bigskip

\subsection{Proof of Lemma \protect\ref{tilde I}}

1) Following the spirit of the proof of Lemma \ref{rate} Part 1, denote $%
\tilde{\theta}(t)=\tilde{\theta}_t$ for convenience and consider
\begin{equation*}
\frac{d}{dt}\tilde{I}_{t}=\tilde{\theta}^{\prime }(t)(1-\lambda EV)-\psi
_{N}^{\prime }(\tilde{\theta}(t))t\tilde{\theta} ^{\prime }(t)-\psi _{N}(%
\tilde{\theta}(t))=-\psi _{N}(\tilde{\theta}(t))\leq 0
\end{equation*}%
for small $t$, using $\psi_N^{\prime }(\tilde{\theta}_t)t=1-\lambda EV$.
Hence the conclusion.

\bigskip

\noindent 2) Consider $\tilde{\theta}_t=(\psi_N^{\prime -1}((1-\lambda
EV)/t) $, well-defined by the strict monotonicity of $\psi_N^{\prime }$. By
steepness of $\psi_N$ we have $(\psi_N^{\prime -1}((1-\lambda
EV)/t)\nearrow\infty$ as $t\searrow0$. So $\tilde{\theta}_t\nearrow\infty$
as $t\searrow0$.

Now write
\begin{equation*}
\tilde{I}_t=\tilde{\theta}_t(1-\lambda EV)-\psi_N(\tilde{\theta}%
_t)t=(1-\lambda EV)\left(\tilde{\theta}_t-\frac{\psi_N(\tilde{\theta}_t)}{%
\psi_N^{\prime }(\tilde{\theta}_t)}\right)\to\infty
\end{equation*}
where the convergence follows from \eqref{technical} and 1).

%\bigskip

\subsection{Proof of Lemma \protect\ref{uniformity}}

To prove Lemma \ref{uniformity}, we first need the following analytical
lemma:

\begin{lemma}
Let $h_m:\mathcal{D}\subset\mathbb{R}^n\to\mathbb{R}$ be a sequence of
monotone functions, in the sense that $h_m(x_1,x_2,%
\ldots,x_{i-1},y_i,x_{i+1},\ldots,x_n)$ is either non-decreasing or
non-increasing in $y_i$ fixing $x_1,\ldots,x_{i-1},x_i,\ldots,x_n$, for any $%
i=1,\ldots,n$. Moreover, suppose $\mathcal{D}$ is compact. If $h_m\to h$
pointwise, where $h$ is continuous, then the convergence is uniform over $%
\mathcal{D}$. \label{monotone}
\end{lemma}

%\bigskip

\begin{proof}
Since $\mathcal{D}$ is compact, continuity of $h$ implies uniform continuity. Therefore, given $\epsilon>0$, there exists $\delta>0$ such that $\|\mathbf{x}_1-\mathbf{x}_2\|<\delta$ implies $|h(\mathbf{x}_1)-h(\mathbf{x}_2)|<\epsilon$. Compactness of $\mathcal{D}$ implies that there is a finite collection of these $\delta$-balls to cover $\mathcal{D}$. Let $\{N_\delta(\mathbf{x})\}_{x\in\mathcal{E}}$ be such collection. Note that $h_m\to h$ uniformly over $\mathcal{E}$.

For any $\mathbf{x}=(x_1,\ldots,x_n)\in\mathcal{D}$, consider
$$|h_m(\mathbf{x}-h(\mathbf{x})|\leq|h_m(\mathbf{x})-h_m(\tilde{\mathbf{x}})|+|h_m(\tilde{\mathbf{x}})+h(\tilde{\mathbf{x}})|+|h(\tilde{\mathbf{x}})-h(\mathbf{x})|$$
where $\tilde{x}=(\tilde{x}_1,\ldots,\tilde{x}_n)$ is chosen to be the closet point to $\mathbf{x}$ in $\mathcal{E}$ that satisfies: For $i=1,\ldots,n$, $\tilde{x}_i\geq x_i$ if $h$ is non-decreasing in the $i$-th component, and $\tilde{x}_i\leq x_i$ if $h$ is non-increasing in the $i$-th component.

By construction we have $|h(\tilde{\mathbf{x}})-h(\mathbf{x})|<2\epsilon$ and $|h_m(\tilde{\mathbf{x}})-h(\tilde{\mathbf{x}})|<\epsilon$ when $m$ is large enough.

Now
\begin{eqnarray*}
&&|h_m(\mathbf{x})-h_m(\tilde{\mathbf{x}})|\\
&=&h_m(\tilde{\mathbf{x}})-h_m(\mathbf{x})\text{\ \ by our choice of $\tilde{\mathbf{x}}$ and monotone property of $h_m$}\\
&\leq&h_m(\tilde{\mathbf{x}})-h_m(\tilde{\tilde{\mathbf{x}}})\text{\ \ where $\tilde{\tilde{\mathbf{x}}}$ is chosen to be the closet point to $\mathbf{x}$ in $\mathcal{E}$ that satisfies:}\\
&&\ \ \ \ \ \ \ \ \ \ \ \ \ \ \ \ \ \ \ \ \ \ \text{For $i=1,\ldots,n$, $\tilde{\tilde{x}}_i\leq x_i$ if $h$ is non-decreasing in the $i$-th component, and}\\
&&\ \ \ \ \ \ \ \ \ \ \ \ \ \ \ \ \ \ \ \ \ \ \text{$\tilde{\tilde{x}}_i\geq x_i$ if $h$ is non-increasing in the $i$-th component.}\\
&\leq&|h_m(\tilde{\mathbf{x}})-h(\tilde{\mathbf{x}})|+|h(\tilde{\mathbf{x}})-h(\tilde{\tilde{\mathbf{x}}})|+|h_m(\tilde{\tilde{\mathbf{x}}})-h(\tilde{\tilde{\mathbf{x}}})|\\
&\leq&\epsilon+2\epsilon+\epsilon
\end{eqnarray*}
when $m$ is large enough.

Combining the above, we have $|h_m(\mathbf{x})-h(\mathbf{x})|\leq7\epsilon$ for all $x\in\mathcal{D}$. Hence the conclusion.

\end{proof}
%\bigskip

\begin{proof}[Proof of Lemma \ref{uniformity}]
For convenience write $\psi_s(\theta;w,z,t)=\log Ee^{\bar{Q}_{w,z}^\infty[t,\infty]}$ and
$$\psi(\theta;w,z,t)=\int_w^z\psi_N(\log(e^\theta\bar{F}(t-u)+F(t-u)))du$$
defined for $\theta\in[\theta_\infty,\theta_T]$, $t\geq T$ and $0\leq w\leq z\leq t+\eta$ for some $\eta>0$. We can extend the domain by putting $\psi_s(\theta;w,z,t)=\psi_s(\theta;w,t+\eta,t)$ and $\psi(\theta;w,z,t)=\psi(\theta;w,t+\eta,t)$ for $z>t+\eta$, and $\psi_s(\theta;w,z,t)=\psi(\theta;w,z,t)=0$ for $w>z$.

Note that $\psi_s(\theta;w,z,t)$ defined as such is non-decreasing in $\theta$, non-increasing in $w$, non-decreasing in $z$ and non-increasing in $t$. Also, $\psi_s(\theta;w,z,t)\to\psi(\theta;w,z,t)$ pointwise with $\psi(\theta;w,z,t)$ continuous. Hence the convergence is uniform over the compact set $\theta\in[\theta_\infty,\theta_T]$ and $(w,z,t)\in[0,K+\eta]\times[0,K+\eta]\times[0,K]$ by Lemma \ref{monotone}, for any $K>0$. By our construction we can extend the set of uniform convergence to $(w,z,t)\in[0,\infty)^2\times[0,K]$.

We now choose $K$ as follows. Given $\epsilon>0$, there exists $K>0$ such that for all $t>K$, $z\leq t-K$, we have
\begin{align*}
\psi(\theta;w,z,t)&=\int_w^z\psi_N(\log(e^\theta\bar{F}(t-u)+F(t-u)))du\\
&=\int_{t-z}^{t-w}\psi_N(\log(e^\theta\bar{F}(u)+F(u)))du\\
&\leq\int_K^\infty\psi_N(\log(e^\theta\bar{F}(u)+F(u)))du\\
&\leq C_1\lambda\int_K^\infty\log(1+(e^\theta-1)\bar{F}(u))du\\
&\leq C_2\lambda\int_K^\infty\bar{F}(u)du\\
&<\epsilon
\end{align*}
for some $C_1,C_2>0$, uniformly over $\theta\in[\theta_\infty,\theta_T]$. Hence for $z\leq t-K$, $\psi_s(\theta;w,z,t)\leq\psi_s(\theta;0,t-K,t)\to\psi(\theta;0,t-K,t)<\epsilon$ uniformly over $\theta\in[\theta_\infty,\theta_T]$ and so $|\psi_s(\theta;w,z,t)-\psi(\theta;w,z,t)|<3\epsilon$ for large enough $s$.

For $z>t-K$, we write
$$\psi_s(\theta;w,z,t)=\frac{1}{s}\log Ee^{\theta\bar{Q}_{w,t-K}^\infty[t,\infty]I(w< t-K)+\theta\bar{Q}_{(t-K)\vee w,z}^\infty[t,\infty]}$$
which is bounded from above by
\begin{eqnarray*}
&&\frac{1}{s}\log\left(Ee^{\theta\bar{Q}_{w,t-K}^\infty[t,\infty]I(w< t-K)}E_0e^{\theta\bar{Q}_{0,(z-t+K)\wedge(z-w)}^\infty[K,\infty]}\right)\\
&=&\psi_s(\theta;w,t-K,t)I(w<t-K)+\frac{1}{s}\log E_0e^{\theta\bar{Q}_{0,(z-t+K)\wedge(z-w)}^\infty[K,\infty]}
\end{eqnarray*}
and bounded from below by
\begin{eqnarray}
&&\frac{1}{s}\log\left(Ee^{\theta\bar{Q}_{0,t-K}^\infty[t,\infty]I(w<t-K)}E_{00}e^{\theta\bar{Q}_{0,(z-t+K)\wedge(z-w)}^\infty[K,\infty]}\right) \notag\\
&=&\psi_s(\theta;w,t-K,t)I(w<t-K)+\frac{1}{s}\log E_{00}e^{\theta\bar{Q}_{0,(z-t+K)\wedge(z-w)}^\infty[K,\infty]} \label{lower bound1}
\end{eqnarray}
where $E_0[\cdot]$ denotes the expectation conditioned that a customer arrives at time 0 and is counted in $\bar{Q}_{0,(z-t+K)\wedge(z-w)}^\infty[t,\infty]$, while $E_{00}[\cdot]$ denotes the expectation conditioned on delayed arrival with tail distribution (in the basic scale) given by $\sup_b P(U^0-b>x|U^0-b)$. Note that $\sup_b P(U^0-b>x|U^0>b)$ is a valid tail distribution because of the light-tail assumption on $U^0$. Indeed, it is obvious that $\sup_b P(U^0-b>0|U^0>b)=1$, and by the same argument following that of \eqref{d p}, we have $\sup_b P(U^0-b>x|U^0>b)\leq e^{-cx}\to0$ for some $c>0$. Moreover, it is obvious that $\sup_b P(U^0-b>x|U^0>b)$ is non-increasing. Now by construction this tail distribution is stochastically at most as large as $P(U^0-b>x|U^0>b)$ for any $b\geq0$, and hence \eqref{lower bound1}. Note that $\frac{1}{s}\log E_0e^{\theta\bar{Q}_{0,(z-t+K)\wedge(z-w)}^\infty[K,\infty]}$ and $\frac{1}{s}\log E_{00}e^{\theta\bar{Q}_{0,(z-t+K)\wedge(z-w)}^\infty[K,\infty]}$ both converge to $\psi(\theta;0,(z-t+K)\wedge(z-w),K)$ uniformly by the argument earlier (as a special case when $t\leq K$). Also we have shown that $\psi_s(\theta;w,t-K,t)$ converges to $\psi_s(\theta;w,t-K,t)$ uniformly for $t>K$ (as a special case when $z\leq t-K$ and $t>K$). The sandwich argument concludes the lemma.
\end{proof}

\subsection{Proof of Lemma \protect\ref{multivariate psi}}

Consider
\begin{eqnarray*}
&&\frac{1}{s}\log E\exp \left\{ \sum_{k=1}^{n}\left( \sum_{j=k}^{n}\theta
_{kj}Q_{(k-1)\Delta ,k\Delta }^{\infty }[(j-1)\Delta ,j\Delta ]+\theta
_{k\cdot }Q_{(k-1)\Delta ,k\Delta }^{\infty }[n\Delta,\infty]\right) \right\}
\\
&=&\frac{1}{s}\log E\exp \Bigg\{ \sum_{k=1}^{n}\Bigg( \sum_{j=k}^{n}\theta
_{kj}\sum_{i=N_{s}((k-1)\Delta )+1}^{N_{s}(k\Delta )}I((j-1)\Delta
<V_{i}+A_{i}\leq j\Delta ){} \\
&&{}+\theta _{k\cdot }\sum_{i=N_{s}((k-1)\Delta )+1}^{N_{s}(k\Delta
)}I(V_{i}+A_{i}>n\Delta)\Bigg) \Bigg\} \\
&=&\frac{1}{s}\log E\prod_{k=1}^{n}\prod_{i=N_{s}((k-1)\Delta
)+1}^{N_{s}(k\Delta )}\left( \sum_{j=k}^{n}e^{\theta _{kj}}P((j-1)\Delta
<V_{i}+A_{i}\leq j\Delta )+e^{\theta _{k\cdot }}\bar{F}(n\Delta-A_{i})\right)
\\
&=&\frac{1}{s}\log E\exp \left\{ \sum_{k=1}^{n}\int_{(k-1)\Delta }^{k\Delta
}h_{k}(u)dN_{s}(u)\right\}
\end{eqnarray*}%
where
\begin{equation*}
h_{k}(u)=\log \left( \sum_{j=k}^{n}e^{\theta _{kj}}P((j-1)\Delta
<V_{i}+u\leq j\Delta )+e^{\theta _{k\cdot }}\bar{F}(n\Delta-u)\right)
\end{equation*}%
Now
\begin{eqnarray*}
&&\frac{1}{s}\log E\exp \left\{ \sum_{k=1}^{n}\sum_{w=1}^{m}h_{k}(\underline{%
\zeta }_{kw})\left[ N_{s}\left( (k-1)\Delta +\frac{w\Delta }{m}\right)
-N_{s}\left( (k-1)\Delta +\frac{(w-1)\Delta }{m}\right) \right] \right\} \\
&\leq &\frac{1}{s}\log E\exp \left\{ \sum_{k=1}^{n}\int_{(k-1)\Delta
}^{k\Delta }h_{k}(u)dN_{s}(u)\right\} \\
&\leq &\frac{1}{s}\log E\exp \left\{ \sum_{k=1}^{n}\sum_{w=1}^{m}h_{k}(%
\overline{\zeta }_{kw})\left[ N_{s}\left( (k-1)\Delta +\frac{w\Delta }{m}%
\right) -N_{s}\left( (k-1)\Delta +\frac{(w-1)\Delta }{m}\right) \right]
\right\}
\end{eqnarray*}%
where $\underline{\zeta }_{kw}=\text{argmin}\{h_{k}(u):(k-1)\Delta
+(w-1)\Delta /m\leq u\leq (k-1)\Delta +w\Delta /m\}$ and $\overline{\zeta }%
_{kw}=\text{argmax}\{h_{k}(u):(k-1)\Delta +(w-1)\Delta /m\leq u\leq
(k-1)\Delta +w\Delta /m\}$. The existence of $\underline{\zeta }_{kw}$ and $%
\overline{\zeta }_{kw}$ is guaranteed by the continuity of $h_{k}(\cdot )$,
which is implied by our assumption that $V_{i}$ has density.

Letting $s\rightarrow \infty $ and by \eqref{increment property} we have
\begin{eqnarray*}
\sum_{k=1}^{n}\sum_{w=1}^{m}\psi _{N}(h_{k}(\underline{\zeta }_{kw}))\frac{%
\Delta }{m} &\leq &\liminf_{s\rightarrow \infty }\frac{1}{s}\log E\exp
\left\{ \sum_{k=1}^{n}\int_{(k-1)\Delta }^{k\Delta }h_{k}(u)dN_{s}(u)\right\}
\\
&\leq &\limsup_{s\rightarrow \infty }\frac{1}{s}\log E\exp \left\{
\sum_{k=1}^{n}\int_{(k-1)\Delta }^{k\Delta }h_{k}(u)dN_{s}(u)\right\} \\
&\leq &\sum_{k=1}^{n}\sum_{w=1}^{m}\psi _{N}(h_{k}(\overline{\zeta }_{kw}))%
\frac{\Delta }{m}
\end{eqnarray*}%
By continuity of $h_{k}\left( \cdot \right) $ and $\psi _{N}\left( \cdot
\right) $, $\psi _{N}\left( h_{k}(\cdot )\right) $ is Riemann integrable.
Letting $m\rightarrow \infty $ yields the conclusion.

%\bigskip

\subsection{Proof of Lemma \ref{CLT1} and \ref{limiting process}}

Our goal here is to prove Lemma \ref{CLT1}, via Lemma \ref{limiting process}. For convenience let $G(y)=\left(\int_y^\infty\bar{F}(u)du\right)^{1/(2+\eta)}$ where $\eta$ is defined in \eqref{nu}. Note that by L'Hospital's rule and Assumption \eqref{light-tail assumption}, we have
\begin{equation}
\lim_{y\to\infty}\frac{y\bar{F}(y)}{G(y)}=\lim_{y\to\infty}\frac{\bar{F}(y)-yf(y)}{-\bar{F}(y)}=\lim_{y\to\infty}(yh(y)-1)=\infty \label{property}
\end{equation}

As discussed before, the key step to show Lemma \ref{CLT1} is an estimate of the limiting Gaussian process given by Lemma \ref{limiting process}. The proof of this inequality takes three steps. We first consider the case when $i=1$. The first step is to define a $d$-metric (in fact a pseudo-metric)
\begin{equation}
d_1((t,y),(t^{\prime },y^{\prime }))=E(\tilde{R}_1(t,y)-\tilde{R}_1(t,y))^2
\label{d-metric}
\end{equation}
where $\tilde{R}_1(t,y)=R_1(t,y)/\nu(y)$ and show that the domain is compact under this (pseudo) metric. Then we can prove that the Gaussian process $\tilde{R}_1(t,y)$ is a.s. bounded by an entropy argument. The third step is an invocation of Borell's inequality.

For convenience let $S=[0,t_0]\times[0,\infty)$.

Before these steps, we need an estimate of the $d$-metric:

\begin{lemma}
Let $(t,y)$ and $(t^{\prime },y^{\prime })$ be two points on $%
[0,t_0]\times[0,\infty)$. Without loss of generality assume $t+y\leq
t^{\prime }+y^{\prime }$. Then
\begin{eqnarray}
&&\frac{\lambda\int_0^{t_2}(\bar{F}(t+y-u)-\bar{F}(t'+y'-u))(1+F(t+y-u)-F(t'+y'-u))du}{\nu(y)^2}{} \nonumber\\
&&{}+\lambda\int_0^{t_2}\bar{F}(t'+y'-u)F(t'+y'-u)du\cdot\left(\frac{1}{\nu(y)}-\frac{1}{\nu(y')}\right)^2{} \nonumber\\
&&{}+\frac{\lambda\int_{t_2}^{t_1}\bar{F}(t_1+y_1-u)F(t_1+y_1-u)du}{\nu(y_1)^2} \label{metric}
\end{eqnarray}
where $t_1=t\vee t^{\prime }$ and $y_1$ is
the corresponding $y$ or $y^{\prime }$. \label{metric lemma}
\end{lemma}
%\bigskip

The proof of this lemma follows the approach in Lemma 5.1 of Krichagina and
Puhalskii (1999). Hence we only sketch the proof here:
\begin{proof}(Sketch)
Recall that
$$\tilde{R}_1(t,y)=\frac{\int_0^t\int_0^\infty I(u+x>t+y)dK(u,x)}{\nu(y)}$$
For a partition $\{u_0=0,u_1,u_2,\ldots,u_k\}$ of $[0,t_0]$, define
$$I_{k,t+y}(u,x)=\sum_{i=1}^kI(u\in(u_{i-1},u_i])I(x>t+y-u_i)$$
Let
$$\tilde{R}_1^k(t,y)=\frac{\int_0^t\int_0^\infty I_{k,t+y}(u,x)dK(u,x)}{\nu(y)}$$
be a discretized  version of $\tilde{R}_1(t,y)$. One can check that $\tilde{R}_1^k(t,y)$ converges to $\tilde{R}_1(t,y)$ in mean square as the mesh of the partition goes to 0.

Now take $(t,y)$ and $(t',y')$ in $S$ such that $t+y\leq t'+y'$. Define $t_1=t\vee t'$ and $y_1$ be the corresponding $y$ or $y'$, and define $t_2=t\wedge t'$ and $y_2$ be the corresponding $y$ or $y'$. Also define $\bar{k}$ such that $u_{\bar{k}}\leq t_1$ while $u_{\bar{k}+1}>t_1$. Using (5.4) and (5.5) in Krichagina and Puhalskii (1999), we have
\begin{eqnarray*}
&&E(\tilde{R}_1^k(t,y)-\tilde{R}_1^k(t',y'))^2\\
&=&\sum_{i=1}^{\bar{k}}\frac{1}{\nu(y)^2}\lambda(u_i-u_{i-1})(F(t'+y'-u_i)-F(t+y-u_i))(1+F(t+y-u_i)-F(t'+y'-u_i)){}\\
&&{}+\sum_{i=1}^{\bar{k}}\left(\frac{1}{\nu(y)}-\frac{1}{\nu(y')}\right)^2\lambda(u_i-u_{i-1})\bar{F}(t'+y'-u_i)F(t'+y'-u_i){}\\
&&{}+\sum_{i=\bar{k}+1}^k\frac{1}{\nu(y_1)^2}\lambda(u_i-u_{i-1})\bar{F}(t_1+y_1-u_i)F(t_1+y_1-u_i)\\
&&{}+o(1)
\end{eqnarray*}
which converges to \eqref{metric} as the mesh goes to 0.
\end{proof}
%\bigskip

\begin{lemma}
We can compactify the space $[0,t_0]\times[0,\infty]$ with the $d$-metric
defined in \eqref{d-metric}. \label{compactification}
\end{lemma}
%\bigskip

\begin{proof}
Consider the mapping $(i,\tan):[0,t_0]\times[0,\pi/2]\to[0,t_0]\times[0,\infty]$, where $i$ is the identity map. Here the domain is equipped with the Euclidean metric while the image is equipped with the $d$-metric. We will show that the mapping $(i,\tan)$ is continuous and well-defined over its domain, including the points $(t,x)$ where $x=\pi/2$, and hence its image is compact.

Suppose first that $(t,x)\to(t^*,x^*)$ where $x\neq\pi/2$. Since $\tan(\cdot)$ is continuous, and $\int_{y}^{t+y}\bar{F}(u)du$ and $\nu(y)$ are continuous in $t$ and $y$ (under Euclidean metric), it is easy to see that $d_1((t,\tan x),(t^*,\tan x^*))\to0$ by using \eqref{metric}.

We now show that $d_1(\cdot,\cdot)$ is still a (pseudo) metric when including the points $(t,y)$ with $y=\infty$. Define, for $y'=\infty$, that
\begin{eqnarray*}
&&d_1((t,y),(t',y'))\\
&=&\frac{\lambda\int_0^{t_2}\bar{F}(t+y-u)(1+F(t+y-u))du}{\nu(y)^2}+\left\{\begin{array}{ll}\frac{\lambda\int_{t'}^t\bar{F}(t+y-u)F(t+y-u)du}{\nu(y)^2}&\text{\ if\ }t>t'\\
0&\text{\ if\ }t\leq t'\end{array}\right.
\end{eqnarray*}
and $d_1((t,y),(t',y'))=0$ if $y=y'=\infty$. It is straightforward to check that $d_1(\cdot,\cdot)$ is continuous at $y'=\infty$ by using \eqref{metric} (note that the second term of \eqref{metric} goes to 0 since for $y'$ large enough it is less than or equal to $\lambda\int_{y'+(t'-t_2)}^{y'+t'}\bar{F}(du)du/\nu(y')^2\leq\lambda G(y)^{1-2/(2+\eta)}\to0$). Hence both the communtativity and triangle inequality hold also at $y'=\infty$, which implies that $d_1(\cdot,\cdot)$ is a pseudo-metric on $[0,t_0]\times[0,\infty]$. Now consider $x^*=\pi/2$. It is now easy to see that $d_1((t,\tan x),(t^*,\infty))\to0$ as $(t,x)\to(t^*,\pi/2)$.
\end{proof}
%\bigskip

\begin{lemma}
$E\sup_S\tilde{R}_1(t,y)<\infty$. In particular, $\tilde{R}_1(t,y)$ is a.s.
bounded over $S$. \label{entropy}
\end{lemma}
%\bigskip

\begin{proof}
We use $C$ here to denote constants, not necessarily the same every time it appears. We carry out an entropy argument (see for example Adler (1990))
$$E\sup_S\tilde{R}_1(t,y)\leq K\int_0^\infty H^{1/2}(\epsilon)d\epsilon = K\int_0^{\text{diam}(S)/2}H^{1/2}(\epsilon)d\epsilon$$
where $K>0$ is a universal constant, $H(\epsilon)=\log N(\epsilon)$ with $N(\epsilon)$ the $\epsilon$-th order entropy of $S$ i.e. the minimum number of $\epsilon$-balls (under $d$-metric) to cover $S$, and $\text{diam}(S)$ is the diameter of $S$ given by $\sup_{(t,y),(t',y')\in S} d_1((t,y),(t',y'))$.

As in Lemma \ref{metric lemma}, let $(t,y)$ and $(t',y')$ be two points on $[0,t_0]\times[0,\infty]$ such that $t+y\leq t'+y'$, and let $t_1=t\vee t'$ with $y_1$ the corresponding $y$ or $y'$. Note that from \eqref{metric} we have
\begin{eqnarray}
d_1((t,y),(t',y'))&=&\frac{\lambda\int_0^{t_2}(\bar{F}(t+y-u)-\bar{F}(t'+y'-u))du}{\nu(y)^2}{} \notag\\
&&{}+\lambda\int_0^{t_2}\bar{F}(t'+y'-u)du\left(\frac{1}{\nu(y)}-\frac{1}{\nu(y')}\right)^2+\frac{\lambda\int_{t_2}^{t_1}\bar{F}(t_1+y_1-u)du}{\nu(y_1)^2} \notag\\
&\leq&\frac{\lambda\int_y^{t+y}\bar{F}(u)du}{\nu(y)^2}+\lambda\left(\int_y^{t_0+y}\bar{F}(u)du\wedge\int_{y'}^{t_0+y'}\bar{F}(u)du\right)\left(\frac{1}{\nu(y)}-\frac{1}{\nu(y')}\right)^2{} \notag\\
&&{}+\frac{\lambda\int_{y_1}^{y_1+|t-t'|}\bar{F}(u)du}{\nu(y_1)^2} \notag\\
&\leq&\lambda G(y)^{1-2/(2+\eta)}+\lambda(G(y)^{1-2/(2+\eta)}\vee G(y')^{1-2/(2+\eta)})+\lambda G(y_1)^{1-2/(2+\eta)} \notag\\
&\leq&C(G(y)^{\eta/(2+\eta)}\vee G(y')^{\eta/(2+\eta)}) \label{estimate}
\end{eqnarray}
which implies that $\text{diam}(S)$ is bounded.

Now pick any $\epsilon>0$. Since $G(\cdot)$ is continuous we can define $G^{-1}(\cdot)$ to be the inverse of $G(\cdot)$. From \eqref{estimate} we have $d_1((t,y),(t',y'))<\epsilon$ for $y,y'>G^{-1}((\epsilon/C)^{(2+\eta)/\eta}$ for some constant $C>0$.

Now also note that
\begin{eqnarray*}
d_1((t,y),(t',y'))&\leq&\frac{\lambda\int_0^{t_2}(\bar{F}(t+y-u)-\bar{F}(t'+y'-u))du}{\nu(y)^2}{}\\
&&{}+\lambda\left(\int_y^{t_0+y}\bar{F}(u)du\right)\wedge\left(\int_{y'}^{t_0+y'}\bar{F}(u)du\right)\left(\frac{1}{\nu(y)}-\frac{1}{\nu(y')}\right)^2+\frac{\lambda|t-t'|}{\nu(y_1)^2}\\
&\leq&\frac{C}{\nu(y)^2\wedge\nu(y')^2}(|t-t'|+|y-y'|)+C(G(y)\wedge G(y'))\frac{\bar{F}(\bar{y})^2}{G(\bar{y})^{2(1+1/(2+\eta))}}|y-y'|^2\\
&&\text{\ \ where $\bar{y}$ is between $y$ and $y'$, by mean value theorem on $1/\nu(\cdot)$}\\
&\leq&\frac{C}{G(y)^{2/(2+\eta)}\wedge G(y')^{2/(2+\eta)}}(|t-t'|+|y-y'|)+\frac{C}{G(y)^{1+1/(2+\eta)}\wedge G(y')^{1+1/(2+\eta)}}|y-y'|^2\\
&\leq&\frac{C}{G(y)^{(3+\eta)/(2+\eta)}\wedge G(y')^{(3+\eta)/(2+\eta)}}(|t-t'|+|y-y'|\vee|y-y'|^2)
\end{eqnarray*}

When at least one of $y$ and $y'$ is less than or equal to $G^{-1}((\epsilon/C)^{(2+\eta)/\eta})$, we then get
$$d_1((t,y),(t',y'))\leq\frac{C}{\epsilon^{(3+\eta)/\eta}}(|t-t'|+|y-y'|\vee|y-y'|^2)$$
Hence we can fill up the space $S$ by
$$N(\epsilon)=O\left(\frac{1}{\epsilon^2}\cdot\frac{1}{\epsilon^{(3+\eta)/\eta}}\cdot G^{-1}\left(\left(\frac{\epsilon}{C}\right)^{(2+\eta)/\eta}\right)\right)$$
number of $\epsilon$-balls. By \eqref{light tail} we get that $G(y)\leq C/y^{1/p}$ for any $p>0$, and so $G^{-1}(\epsilon)\leq C/\epsilon^{1/p}$. This gives
$$N(\epsilon)=O\left(\frac{1}{\epsilon^2}\cdot\frac{1}{\epsilon^{(3+\eta)/\eta}}\cdot\frac{1}{\epsilon^p}\right)=O\left(\frac{1}{\epsilon^{2+(3+\eta)/\eta+p}}\right)$$
and hence
$$\int_0^{\text{diam}(S)}H^{1/2}(\epsilon)d\epsilon=O\left(\int_0^C\sqrt{\log\left(\frac{1}{\epsilon}\right)}d\epsilon+C\right)<\infty$$
\end{proof}
%\bigskip

\begin{lemma}
Borell-TIS inequality holds i.e. for $x\geq E\sup_S\tilde{R}_1(t,y)$,
\begin{equation*}
P\left(\sup_S\tilde{R}_1(t,y)\geq x\right)\leq\exp\left\{-\frac{1}{2\sigma_1^2}%
\left(x-E\sup_S\tilde{R}_1(t,y)\right)^2\right\}
\end{equation*}
where
\begin{equation*}
\sigma_1^2=\sup_SE\tilde{R}_1(t,y)^2
\end{equation*}
\label{Borell}
\end{lemma}
%\bigskip

\begin{proof}
Note that
$$E\tilde{R}_1(t,y)^2=\frac{\lambda\int_0^t\bar{F}(t+y-u)F(t+y-u)du}{\nu(y)^2}\leq\frac{\lambda\int_y^{t+y}\bar{F}(u)du}{G(y)^{2/(2+\eta)}}\leq\lambda G(y)^{\eta/(2+\eta)}$$
and so
$$\sigma_1^2=\sup_SE\tilde{R}_1(t,y)\leq C$$
for some constant $C$. By Lemma \ref{entropy} $\tilde{R}_1(t,y)$ is a.s. bounded and Borell-TIS inequality holds.
\end{proof}
\bigskip

We now carry out the same scheme for $R_2(t,y)$. Let $\tilde{R}_2(t,y)=R_2(t,y)/\nu(y)$. Indeed it is straightforward to show that the $d$-metric
of $\tilde{R}_2(t,y)$ is given by
\begin{align}
d_2((t,y),(t^{\prime },y^{\prime }))&=E(\tilde{R}_2(t,y)-\tilde{R}_2(t',y'))^2 \notag\\
&=\lambda c_a^2\int_0^{t_2}\left(\frac{\bar{%
F}(t+y-u)}{\nu(y)}-\frac{\bar{F}(t^{\prime }+y^{\prime }-u)}{\nu%
(y^{\prime })}\right)^2du+\lambda c_a^2\int_{t_2}^{t_1}\left(\frac{\bar{F}%
(t_1+y_1-u)}{\nu(y_1)}\right)^2du  \label{d R_2}
\end{align}
where again $t_1=t\vee t^{\prime }$, $t_2=t\wedge t^{\prime }$ and $y_1$, $%
y_2$ are the corresponding $y$ or $y^{\prime }$.

\begin{lemma}
We can compactify the space $S$ with the $d$-metric
defined in \eqref{d R_2}. \label{compactification2}
\end{lemma}
\bigskip

\begin{proof}
For $(t,y),(t',y')$ such that $y,y'\neq\infty$, write
\begin{eqnarray*}
&&d_2((t,y),(t',y'))\\
&=&\lambda c_a^2\Bigg(\frac{\int_0^{t_2}\bar{F}(t+y-u)^2du}{\nu(y)^2}+\frac{\int_0^{t_2}\bar{F}(t'+y'-u)^2du}{\nu(y')^2}-\frac{2\int_0^{t_2}\bar{F}(t+y-u)\bar{F}(t'+y'-u)du}{\nu(y)\nu(y')}{}\\
&&{}+\frac{\int_{t_2}^{t_1}\bar{F}(t_1+y_1-u)^2du}{\nu(y_1)^2}\Bigg)
\end{eqnarray*}
and define, for $y'=\infty$, that
$$d_2((t,y),(t',y'))=\int_0^t\frac{\bar{F}(t+y-u)^2}{\nu(y)^2}du$$
and $d_2((t,y),(t',y'))=0$ if both $y,y'=\infty$.

Then $d_2((t,y),(t',y'))$ is continuous at $y'=\infty$ since
$$\frac{\int_0^{t_2}\bar{F}(t'+y'-u)du}{\nu(y')^2}\leq\frac{\int_{y'}^{t_0+y'}\bar{F}(u)du}{\nu(y')^2}=G(y')^{\eta/(2+\eta)}\to0$$
and
\begin{align*}
\frac{\int_0^{t_2}\bar{F}(t+y-u)\bar{F}(t'+y'-u)du}{\nu(y)\nu(y')}&\leq\frac{\sqrt{\int_0^{t_2}\bar{F}(t+y-u)^2du\int_0^{t_2}\bar{F}(t'+y'-u)^2du}}{\nu(y)\nu(y')}\\
&\leq\sqrt{\frac{\int_y^{t_0+y}\bar{F}(u)du}{\nu(y)^2}}\cdot\sqrt{\frac{\int_{y'}^{t_0+y'}\bar{F}(u)du}{\nu(y')^2}}\\
&\leq G(y)^{\eta/(2(2+\eta))}G(y')^{\eta/(2(2+\eta))}\\
&\to0
\end{align*}
If $t'>t$, then
$$\frac{\int_t^{t'}\bar{F}(t'+y'-u)^2du}{\nu(y')^2}\leq\frac{\int_{y'}^{t_0+y'}\bar{F}(u)du}{\nu(y')^2}\leq G(y')^{\eta/(2+\eta)}\to0$$
Hence $d_2(\cdot,\cdot)$ is continuous at $y'=\infty$. The rest follows as in the proof of Lemma \ref{compactification}.
\end{proof}
\bigskip

\begin{lemma}
$E\sup_S\tilde{R}_2(t,y)<\infty$. In particular, $\tilde{R}_2(t,y)$ is a.s.
bounded over $S$. \label{entropy2}
\end{lemma}
\bigskip

\begin{proof}
From \eqref{d R_2} we have the estimate
\begin{eqnarray}
&&d_2((t,y),(t',y')) \notag\\
&\leq&2\lambda c_a^2\left(\int_0^t\left(\frac{\bar{F}(t+y-u)}{\nu(y)}\right)^2du\vee\int_0^{t'}\left(\frac{\bar{F}(t'+y'-u)}{\nu(y')}\right)^2du\right)+\lambda c_a^2\int_{t_1}^{t_2}\left(\frac{\bar{F}(t_1+y_1-u)}{\nu(y_1)}\right)^2du \notag\\
&\leq&2\lambda c_a^2(G(y)^{\eta/(2+\eta)}\vee G(y')^{\eta/(2+\eta)})+\lambda c_a^2G(y_1)^{\eta/(2+\eta)} \label{estimate12}
\end{eqnarray}
On the other hand, using multivariate Taylor series expansion,
\begin{eqnarray*}
&&\frac{\bar{F}(t+y-u)}{\nu(y)}-\frac{\bar{F}(t'+y'-u)}{\nu(y')}\\
&\leq&\sup_{t,y}\left|\frac{f(t+y-u)}{\nu(y)}\right||t-t'|+\sup_{t,y}\left|\frac{1}{2+\eta}\frac{\bar{F}(t+y-u)\bar{F}(y)}{G(y)^{1+1/(2+\eta)}}-\frac{f(y)}{G(y)^{1/(2+\eta)}}\right||y-y'|\\
&\leq&\frac{C}{G(y)^{(3+\eta)/(2+\eta)}}(|t-t'|+|y-y'|)
\end{eqnarray*}
and hence
\begin{equation}
d_2((t,y),(t',y'))\leq\frac{C}{G(y)^{(3+\eta)/(2+\eta)}}(|t-t'|+|y-y'|) \label{estimate22}
\end{equation}
where $C$ are constants not necessarily the same every time they appear. With \eqref{estimate12} and \eqref{estimate22}, the rest follows as in the proof of Lemma \ref{entropy}.
\end{proof}
\bigskip

\begin{lemma}
Borell-TIS inequality holds i.e. for $x\geq E\sup_S\tilde{R}_2(t,y)$,
\begin{equation*}
P\left(\sup_S\tilde{R}_2(t,y)\geq x\right)\leq\exp\left\{-\frac{1}{2\sigma_2^2}%
(x-E\sup_S\tilde{R}_2(t,y))^2\right\}
\end{equation*}
where
\begin{equation*}
\sigma_2^2=\sup_SE\tilde{R}_2(t,y)^2
\end{equation*}
\label{Borell2}
\end{lemma}
\bigskip

\begin{proof}
Note that
$$E\tilde{R}_2(t,y)^2=\frac{\lambda c_a^2\int_0^t\bar{F}(t+y-u)^2du}{\nu(y)^2}\leq\frac{\lambda c_a^2\int_y^{t+y}\bar{F}(u)du}{G(y)^{2/(2+\eta)}}\leq\lambda c_a^2G(y)^{\eta/(2+\eta)}$$
The rest follows as in the proof of Lemma \ref{Borell}.
\end{proof}
\bigskip

Lemma \ref{limiting process} is now an immediate corollary of Lemma \ref{Borell} and \ref{Borell2}:

\begin{proof}[Proof of Lemma \ref{limiting process}]
\begin{eqnarray*}
&&P(|R(t,y)|\leq C_*\nu(y)\text{\ for all\ }t\in[0,t_0],\ y\in[0,\infty))\\
&\geq&P\left(\sup_S|\tilde{R}_1(t,y)|+\sup_S|\tilde{R}_2(t,y)|\leq C_*\right) \\
&\geq&P\left(\sup_S|\tilde{R}_1(t,y)|\leq\frac{C_*}{2}\right)P\left(\sup_S|%
\tilde{R}_2(t,y)|\leq\frac{C_*}{2}\right) \\
&>&0
\end{eqnarray*}
when $C_*$ is large enough, by the independence of $\tilde{R}_1(\cdot,\cdot)$ and $\tilde{R}_2(\cdot,\cdot)$ in the
second inequality.
\end{proof}
\bigskip

With Lemma \ref{limiting process}, we now prove Lemma \ref{CLT1}.

\begin{proof}[Proof of Lemma \ref{CLT1}]
First consider \eqref{CLT2}. Take $C_1=3C_*$ where $C_*$ is the constant in Lemma \ref{limiting process}. We have
\begin{eqnarray}
&&P\left(\bar{Q}^\infty(t,y)\in\left(\lambda s\int_y^{t+y}\bar{F}(u)du\pm\sqrt{s}C_1\nu(y)\right)\text{\ for all\ }t\in[0,t_0],\ y\in[0,\infty)\Bigg|B(0)\right) \nonumber\\
&\geq&P\Bigg(U_0\leq x,\ 0\in\left(\lambda s\int_y^{t+y}\bar{F}(u)du\pm\sqrt{s}C_1\nu(y)\right)\text{\ for\ }t\in[0,U_0],\ y\in[0,\infty),{} \nonumber\\
&&{}\bar{Q}^\infty(t,y)\in\left(\lambda s\int_y^{t+y}\bar{F}(u)du\pm\sqrt{s}C_1\nu(y)\right)\text{\ for all\ }t\in[U_0,t_0],\ y\in[0,\infty)\Bigg|B(0)\Bigg) \label{intermediate}
\end{eqnarray}
Letting $x=1/(\lambda s)$, we will show that $0\in\left(\lambda s\int_y^{t+y}\bar{F}(u)du\pm\sqrt{s}C_1\nu(y)\right)$ for $t\in[0,U_0]$ and $y\in[0,\infty)$ in the expression is redundant. In fact, let $m(s)=\inf\left\{\sqrt{s}C_*\nu(y)<\frac{1}{2}\right\}$. When $y=m(s)$, $\lambda s\int_y^{t+y}\bar{F}(u)du$ is less than 1 for large enough $s$, and when $y\geq m(s)$ it decays faster than $\sqrt{s}C_1\nu(y)<\frac{1}{2}$ (see Remark 1 in the paper for similar argument). Hence $\left(\lambda s\int_y^{t+y}\bar{F}(u)du\pm\sqrt{s}C_1\nu(y)\right)$ contains 0 when $y\geq m(s)$. When $y<m(s)$, the choice of $x$ gives
$$\lambda s\int_y^{t+y}\bar{F}(u)du\leq\lambda st\bar{F}(y)\leq\lambda sx=1$$
for $t\in[0,U_0]$ and $U_0\leq x$. Hence $\left(\lambda s\int_y^{t+y}\bar{F}(u)du\pm\sqrt{s}C_1\nu(y)\right)$ also contains 0 when $y<m(s)$.

In fact with the same choice of $x$, by similar argument we have $\left(\lambda s\int_y^{t+y}\bar{F}(u)du\pm\sqrt{s}C_*\nu(y)\right)$ contains only 0 for $t\in[0,U_0]$ and $y\geq m(s)$, and that
$0\in\left(\lambda s\int_y^{t+U_0+y}\bar{F}(u)du\pm\sqrt{s}C_1\nu(y)\right)$ for $t\in[0,U_0]$ and $y\geq m(s)$. This will be useful later on in the proof.

The same choice of $x$, together with the fact that $\bar{F}(\cdot)$ is decreasing, also guarantees that
\begin{equation}
\lambda s\int_{t+y}^{t+U_0+y}\bar{F}(u)du\leq 2C_*\sqrt{s}\nu(y) \label{U_0 estimate}
\end{equation}
In fact, when $y=m(s)$, $\lambda s\int_{t+y}^{t+U_0+y}\bar{F}(u)du$ is less than 1 when $s$ is large enough, and when $y\geq m(s)$ it decays faster than $2C_*\sqrt{s}\nu(y)$. Hence the inequality \eqref{U_0 estimate} when $y\geq m(s)$. When $y<m(s)$ the fact that $U_0\leq x$ leads to $\lambda s\int_{t+y}^{t+U_0+y}\bar{F}(u)du\leq1$, hence the conclusion. Again this will be useful later on.

Hence \eqref{intermediate} is greater than or equal to
$$P(U_0\leq x|B(0))P\left(\bar{Q}_0^\infty(t,y)\in\left(\lambda s\int_y^{t+U_0+y}\bar{F}(u)du\pm\sqrt{s}C\tilde{C}(y)\right)\text{\ for all\ }t\in[0,t_0]\Bigg|U_0\leq x\right)$$
where $\bar{Q}_0^\infty(t,y)$ is independent of $U_0$ and has the same distribution as $\bar{Q}^\infty(t,y)$ with initial age 0 and no initial customers.

For any $U_0\leq x$, we have
\begin{eqnarray*}
&&P\left(\bar{Q}_0^\infty(t,y)\in\left(\lambda s\int_y^{t+U_0+y}\bar{F}(u)du\pm\sqrt{s}C_1\nu(y)\right)\text{\ for all\ }t\in[0,t_0],\ y\in[0,\infty)\right)\\
&\geq&P\Bigg(\bar{Q}_0^\infty(t,y)\in\left(\lambda s\int_y^{t+U_0+y}\bar{F}(u)du\pm\sqrt{s}C_1\nu(y)\right)\text{\ for all\ }t\in[0,t_0],\ y\in[0,m(s)){}\\
&&{}\bar{Q}_0^\infty(t,y)\in\left(\lambda s\int_y^{t+y}\bar{F}(u)du\pm\sqrt{s}C_*\nu(y)\right)\text{\ for all\ }t\in[0,t_0],\ y\in[m(s),\infty)\Bigg)\\
&&{}\ \ \text{(since the interval $\left(\lambda s\int_y^{t+y}\bar{F}(u)du\pm\sqrt{s}C_*\nu(y)\right)$ only contains 0 while }{}\\
&&{}\ \ \text{$0\in\left(\lambda s\int_y^{t+U_0+y}\bar{F}(u)du\pm\sqrt{s}C_1\nu(y)\right)$ when $y>m(s)$ as discussed above)}\\
&\geq&P\Bigg(\sup_{y\in[0,m(s))}\left|\frac{\bar{Q}_0^\infty(t,y)-\lambda s\int_y^{t+y}\bar{F}(u)du}{\sqrt{s}}\right|+\sup_{y\in[0,m(s))}\lambda\sqrt{s}\int_{t+y}^{t+U_0+y}\bar{F}(u)du\leq C_1\nu(y),{}\\
&&{}\bar{Q}_0^\infty(t,y)\in\left(\lambda s\int_y^{t+y}\bar{F}(u)du\pm\sqrt{s}C_*\nu(y)\right)\text{\ for all\ }t\in[0,t_0],\ y\in[m(s),\infty)\Bigg)\\
&\geq&P\left(\left|\frac{\bar{Q}_0^\infty(t,y)-\lambda s\int_y^{t+y}\bar{F}(u)du}{\sqrt{s}}\right|\leq C_*\nu(y)\text{\ for all\ }t\in[0,t_0],\ y\in[0,\infty)\right){}\\
&&{}\ \ \text{(by \eqref{U_0 estimate})}\\
&\to&P(|R(t,y)|\leq C_*\nu(y)\text{\ for all\ }t\in[0,t_0],\ y\in[0,\infty))>0
\end{eqnarray*}
by Lemma \ref{limiting process}. The convergence follows from Functional Central Limit Theorem (see Pang and Whitt (2009)) and that the set $\{f:|f(t,y)|\leq C_*\nu(y)\text{\ for all\ }t\in[0,t_0],\ y\in[0,\infty)\}$ is a continuity set.

Lastly, since $U^0$ is light-tailed, by the argument following \eqref{d p} in the proof of Proposition \ref{asymptotic}, we have
$$\inf_{b\geq0}P\left(U_0\leq\frac{1}{\lambda s}\bigg|B(0)=b\right)=\inf_{b\geq0}P\left(U^0-b\leq\frac{1}{\lambda}\bigg|U^0>b\right)\geq1-e^{-c/\lambda}>0$$
for some constant $c>0$. Hence \eqref{CLT2} holds. Inequality \eqref{CLT3} is obvious since one can isolate any point inside $S$ and the projection of the
process on the point will possess Gaussian distribution. For example, we can write
\begin{eqnarray*}
&&P\left(\bar{Q}^\infty(t,y)\notin\left(\lambda s\int_y^{t+y}\bar{F}(u)du\pm\sqrt{s}%
C_1\nu(y)\right)\text{\ for some\ }t\in[0,t_0%
],\ y\in[0,\infty)\Bigg|B(0)\right)\\
&\geq&P(U_0\leq x)P\left(\bar{Q}_0^\infty(t^*,y^*)\geq\lambda s\int_{y^*}^{t^*+x+y^*}\bar{F}(u)du+\sqrt{s}C_1\nu(y^*)\right)\\
&>&0
\end{eqnarray*}
for any $t^*\in[0,t_0]$ and $y^*\in[0,\infty)$.
\end{proof}
\bigskip

%\small

\end{document}